\documentclass[oneside,reqno]{article}

\usepackage[english]{babel}
\usepackage[T1]{fontenc}
\usepackage{lmodern}

\usepackage{geometry}
\usepackage{amsmath} 
\usepackage{etoolbox} 
\usepackage{amsthm}
\usepackage{amsfonts}
\usepackage{amssymb}
\usepackage{mathtools}
\usepackage{cmll}
\usepackage{stmaryrd}
\usepackage{aligned-overset}
\usepackage{enumitem}
\allowdisplaybreaks

\usepackage{hyperref}
\apptocmd{\thebibliography}{\raggedright}{}{}

\usepackage{tikz-cd}
\usepackage{graphicx}
\usepackage{xcolor}
\usepackage{epigraph}
\usepackage{ebproof}

\newtheorem{theorem}{Theorem}[section]
\newtheorem{definition}[theorem]{Definition}
\newtheorem{proposition}[theorem]{Proposition}
\newtheorem{lemma}[theorem]{Lemma}

\newtheorem{corollary}[theorem]{Corollary}

\newtheorem{example}[theorem]{Example}

\newtheorem{construction}[theorem]{Construction}

\title{Uniqueness and non-uniqueness of deriving transformations}
\author{Jean-Baptiste Vienney\footnote{Email: jeanbaptiste.vienney@gmail.com}}

\begin{document}

\maketitle

\begin{abstract}
A differential category is an additive symmetric monoidal category with an algebra modality, axiomatizing smooth functions, and a deriving transformation, axiomatizing the differentiation of these smooth functions. An algebra modality may or may not have both the nullary and the binary Seely isomorphisms. Lemay proved that if an algebra modality has both Seely isomorphisms, then it has at most one deriving transformation. It was then an open problem whether every algebra modality has at most one deriving transformation. We answer this question in the negative by constructing algebra modalities with infinite sets of deriving transformations. We also clarify the relations between the nullary Seely isomorphism, the binary Seely isomorphism, and uniqueness of deriving transformations (that is, the property for an algebra modality to have at most one deriving transformation). First, we refine Lemay's theorem by removing the assumption of the nullary Seely isomorphism. Then, we show by means of various examples that the six combinations of absence or presence of nullary Seely isomorphism, binary Seely isomorphism, and uniqueness of deriving transformations, which are not already ruled out by the refined uniqueness theorem, can be realized.
\end{abstract}
\section{Introduction}
Differential categories \cite{BCS2006,BCLS2020} are symmetric monoidal categories equipped with a notion of differentiation. They were originally introduced to provide the categorical semantics of differential linear logic \cite{Ehrhard2018}. The latter extends linear logic \cite{Girard1987} by adding rules which permit the differentiation of proofs. Since their introduction, differential categories have been connected to other categorical frameworks such as Cartesian differential categories \cite{BCS2009}, which axiomatize smooth functions between Cartesian spaces, and tangent categories \cite{Rosicky1984,CC2014}, which axiomatize the tangent bundle of smooth manifolds. Differential categories are thus under active investigation by a growing research community, resulting in a vast body of literature.

Despite all the research in this field, a basic question has remained open: \emph{In a differential category, all the remaining data being fixed, is it possible to define two distinct notions of differentiation?}

A differential category is an additive symmetric monoidal category $(\mathsf{C},\otimes,I)$, that is, a symmetric monoidal category enriched over commutative monoids, together with an algebra modality, axiomatizing smooth functions, and a deriving transformation, axiomatizing the differentiation of these smooth functions. In the technical language of differential categories, the basic question is thus: \emph{Can an algebra modality on an additive symmetric monoidal category admit two distinct deriving transformations?}

This question was asked by Lemay in a presentation a few years ago \cite{Lemay2022Uniqueness}. At that time, Lemay had proved that in a specific class of differential categories, namely those providing the categorical semantics of differential linear logic, there is a unique notion of differentiation. This result was subsequently published in \cite{Lemay2026Coderelictions}. Note that Lemay first proved in earlier work a more restricted result, namely that there is uniqueness of differentiation for free exponential modalities \cite[Theorem 21]{Lemay2021Coderelictions}. However, he was not able to resolve the general case, that is, the question above. We show in this paper that an algebra modality may admit several deriving transformations---in fact, infinitely many. Our main result in this direction is Theorem \ref{main-theorem} which determines the infinite set of all the deriving transformations on an algebra modality $S^f$ that we will construct in this paper.

An important concept for differential categories is that of Seely isomorphisms. In an additive symmetric monoidal category with finite coproducts, we say that an algebra modality admits Seely isomorphisms if both a specific natural transformation and a specific morphism are isomorphisms. The natural transformation is then called the binary Seely isomorphism and the morphism the nullary Seely isomorphism. Lemay's result implies that an algebra modality on an additive symmetric monoidal category with finite coproducts admits at most one deriving transformation if it has both Seely isomorphisms. We refine this result by proving Theorem \ref{prop:binary-seely-uniqueness} which states that such an algebra modality has in fact at most one deriving transformation as soon as it has the binary Seely isomorphism.

Furthermore, we clarify the relations between the nullary Seely isomorphism, the binary Seely isomorphism, and the uniqueness of deriving transformations (that is, the property for an algebra modality to have at most one deriving transformation). We show by means of various examples, which are summed up in Figure \ref{fig:comb}, that the six combinations of absence or presence of nullary Seely isomorphism, binary Seely isomorphism, and uniqueness of deriving transformations, which are not already ruled out by the refined uniqueness theorem, can be realized. 

Note that all the examples of algebra modalities in the region labelled ``At most one deriving transformation'' of the diagram admit exactly one deriving transformation. However, there exist algebra modalities admitting no deriving transformation. The paper \cite{BCLS2020} provides an example of an algebra modality which has both Seely isomorphisms and admits no deriving transformation, and an example of an algebra modality which does not have both Seely isomorphisms and admits no deriving transformation. The aforementioned paper does not discuss separately the nullary and binary Seely isomorphisms. We have not explored further in the present paper algebra modalities admitting no deriving transformation.

In Figure \ref{fig:comb}, only the algebra modalities $S$, $\mathcal{M}$, and $S^\infty$ have already been discussed in the literature, whereas the four other algebra modalities are discussed for the first time here. The notations have been chosen in order to make apparent a symmetry with respect to an axis passing through $\mathcal{M}$ and between $S'$ and $S$: in each of the pairs $(S',S)$, $(T^\infty,S^\infty)$, and $(S^f,T^f)$, one of the two algebra modalities is a variant of the other which has been conceived precisely in order to remove or add the nullary Seely isomorphism, while keeping the other properties. A variant is denoted either by replacing the $S$ by a $T$, or by adding a prime. (We have refrained from using the letter $T$ to denote our variant of the symmetric algebra modality $S$ since we will already use $T$ to denote an arbitrary algebra modality.)

Much of the information provided in this diagram is presented for the first time in our paper. As far as we know, the only information contained in this diagram which is in the previous literature is that the combination of nullary and binary Seely isomorphisms implies uniqueness of deriving transformations \cite{Lemay2026Coderelictions}, that $S$ and $\mathcal{M}$ have both Seely isomorphisms \cite{BCLS2020}, and that $S^\infty$ has the nullary but not the binary Seely isomorphism \cite{CLLW2021}. This is not enough information to confirm the shape of our diagram and prove that the various nonempty regions are indeed nonempty.

An important intermediate result that we use to construct the two algebra modalities $S^f$ and $T^f$ in Figure \ref{fig:comb}, which admit infinitely many deriving transformations, is Proposition \ref{prop:alg-mod-from-alg-theory} which constructs an algebra modality on the category of commutative monoids from any algebraic theory extending the theory of commutative rigs. Our construction is similar to, but different from the construction in \cite{BCS2006} of an algebra modality on $k$-$\mathsf{Mod}$ from any ``polynomial theory''. Indeed, first, polynomial theories are the algebraic theories $\mathbb{T}$ extending the theory of commutative $k$-algebras such that the constants of $\mathbb{T}$, that is, the terms containing no variables, are exactly the elements of $k$. Second, the underlying functor $T$ of the algebra modality on $k$-$\mathsf{Mod}$ built from a polynomial theory in the aforementioned paper is such that for every $M \in k$-$\mathsf{Mod}$, the $k$-module $TM$ is a specified subset of $\mathsf{Set}(M^*,k)$, where $M^*=k$-$\mathrm{Mod}(M,k)$. However, with our construction, which could be adapted straightforwardly to the construction of an algebra modality on $k$-$\mathsf{Mod}$ for $k \neq \mathbb{N}$, the $k$-module $TM$ would be a set of syntactic terms modulo an equivalence relation built starting from the elements of $M$ and applying iteratively the operations of the algebraic theory.
 
Let us also highlight that we prove in Section \ref{sec:unique-S-infty} that the free $C^\infty$-ring algebra modality admits a unique deriving transformation. This algebra modality and its deriving transformation were introduced in \cite{BCS2006} and further discussed in \cite{BCLS2020} and \cite{CLLW2021}, but it was not proved in any of these papers that its deriving transformation is unique. Another notable construction is provided by Theorem \ref{prop:alg-mod+adj}, which constructs a differential modality on an additive symmetric monoidal category $(\mathsf{D},\otimes,I)$ starting from a differential modality on an additive symmetric monoidal category $(\mathsf{C},\otimes,I)$ and an appropriate adjunction whose underlying functors go from $\mathsf{C}$ to $\mathsf{D}$ and vice versa. We apply this result to extension-restriction of scalars adjunctions. It lets us build the differential modalities $S'$ and $T^\infty$ in Figure \ref{fig:comb}, which do not have the nullary Seely isomorphism.

\begin{figure}
\centering
\begin{tikzpicture}[
  every node/.style={align=center},
  labelname/.style={font=\small},
  regiontext/.style={font=\small}
]

  \draw[thick] (-11.4,-4.5) rectangle (1.6,4.5);

  \draw[thick] (-6,0) circle (4);
  \draw[thick] (-2.6,0.3) circle (3.5);
  \draw[thick] (-6.5,0) circle (2.3);

  \node[labelname] at (-7.1,1.25) {Binary Seely\\isomorphism};

  \node[labelname] at (-0.8,-1) {Nullary Seely\\isomorphism};

  \node[labelname] at (-6.6,3) {At most one deriving\\transformation};

  \node[regiontext] at (-5.3,0.7) {$S$};
  \draw (-5.3,0.7) circle (0.35);
  
  \node[regiontext] at (-5.3,-0.7) {$\mathcal{M}$};
  \draw (-5.3,-0.7) circle (0.35);
  \node[regiontext] at (-6.9,-0.25) {$S'$};
  \draw (-6.9,-0.25) circle (0.35);
  \node[regiontext] at (-2.7,-0.5) {$S^\infty$};
  \draw (-2.7,-0.5) circle (0.35);
  \node[regiontext] at (-6.5,-3.1) {$T^\infty$};
  \draw (-6.5,-3.1) circle (0.35);
  \node[regiontext] at (-1.3,2) {$T^f$};
  \draw (-1.3,2) circle (0.35);
  \node[regiontext] at (-10,-3.5) {$S^f$};
  \draw (-10,-3.5) circle (0.35);

\end{tikzpicture}
\caption{All possible combinations of uniqueness or non-uniqueness of deriving transformations and Seely isomorphisms for algebra modalities. \footnotesize{The symmetric algebra modality $S$ is discussed in Examples \ref{ex:algebra-mod-sym}, \ref{ex:der-sym}, and \ref{ex:unique}. The multiset algebra modality $\mathcal{M}$ is discussed in Examples \ref{ex:algebra-mod-rel}, \ref{ex:der-rel}, and \ref{ex:unique}. The free $C^\infty$-ring on a real vector space algebra modality $S^\infty$ is discussed in Examples \ref{ex:algebra-mod-C-infty}, \ref{ex:C-infty}, \ref{ex:C-infty-Seely}, and in Section~\ref{sec:unique-S-infty}. The free commutative rig with an endofunction on an additive commutative monoid algebra modality $S^f$ is discussed in Subsections~\ref{sec:idea-uniqueness}, \ref{sec:origin}, Example \ref{ex:endofunction}, and Subsections \ref{sec:failure-Tf}--\ref{sec:Kleisli}. The algebra modality $T^f$ on $\mathsf{cMon}$ is discussed in Subsection \ref{nullary-no-binary-non-unique}. The free commutative ring on an additive commutative monoid algebra modality $S'$ is discussed in Examples \ref{ex:new-monad-free-ring}, \ref{ex:new-alg-mod-free-ring}, and \ref{example:der-free-ring}. The free $C^\infty$-ring on a rational vector space algebra modality $T^\infty$ is discussed in Examples \ref{ex:new-monad-free-C}, \ref{ex:new-alg-mod-free-C}, and \ref{example:der-free-C}, and in Subsection \ref{sec:unique-T}.}}
\label{fig:comb}
\end{figure}

\textbf{Outline.} We introduce differential categories and the uniqueness problem in Section \ref{sec:diff-cat}. We refine Lemay's theorem in Section \ref{sec:refined}. We prove the uniqueness of the deriving transformation on the free $C^\infty$-ring algebra modality $S^\infty$ in Section \ref{sec:unique-S-infty}. We introduce the counterexamples $S^f$ and $T^f$ to uniqueness in Section \ref{sec:counter}. Finally,  in Section \ref{sec:no-null}, we introduce two variants $S'$ and $T^\infty$ of algebra modalities previously discussed in the paper, which do not have the nullary Seely isomorphism. They are obtained by extension and restriction of scalars.

\textbf{Notations and conventions.} 
LHS means ``Left-Hand Side'' and RHS means ``Right-Hand Side''. We will write $;$ for the composition in diagrammatic order, and $\circ$ for the composition in functional order. When we work with a symmetric monoidal category, we will treat the associators and unitors as equalities of functors, that is, we will work as if the symmetric monoidal category was strict monoidal. We will write $\sigma \colon A \otimes B \rightarrow B \otimes A$ for the symmetry.

\textbf{Acknowledgments.} I would like to thank Richard Blute and Jean-Simon Lemay for useful discussions and their support of this research, Simon Henry for communicating the reference \cite{PV2007}, Dorette Pronk and Scott Wesley for advice on how to make a math paper readable, and the anonymous referees for their very helpful reports on the first version of this paper. This research has been partially funded by an NSERC grant awarded to Richard Blute. 
\section{Differential categories and the uniqueness problem} \label{sec:diff-cat}
This section is for background and motivation and does not contain novel material. Unless otherwise specified, our basic reference is \cite{BCLS2020}.

The definition of a differential category starts with an axiomatization of categories of linear maps. We want a tensor product, sums of linear maps, and zero maps.
\begin{definition}
An \emph{additive symmetric monoidal category} is a symmetric monoidal category $(\mathsf{C},\otimes,I)$ such that $\mathsf{C}$ is enriched over $(\mathsf{cMon},\otimes,\mathbb{N})$ and $-\otimes-\colon \mathsf{C}[A,B] \times \mathsf{C}[C,D] \rightarrow \mathsf{C}[A \otimes C,B \otimes D]$ is biadditive, that is, preserves $0$ and $+$ on the left and on the right.
\end{definition}
To introduce the basic Example \ref{ex:mod} of an additive symmetric monoidal category, we recall some definitions. A \emph{commutative rig} is a set $k$ together with two binary operations $+$ and $\cdot$ and two constants $0$ and $1$, such that $(k,+,0)$ and $(k,\cdot,1)$ are commutative monoids, $\cdot$ distributes over $+$, and $0 \cdot r=0$ for every $r \in k$. Let $k$ be a commutative rig. A \emph{$k$-module} is a set $A$ together with a binary operation $+$, a constant $0$, and an operation $\cdot \colon k \times A \rightarrow A$ such that $(A,+,0)$ is a commutative monoid,  $\cdot$ distributes over the operations $+$ of $k$ and of $A$, and we have
\begin{align*}
0 \cdot a&=0,\quad r \cdot 0=0, \\
 (rs) \cdot a&=r \cdot (s \cdot a), \\
 1 \cdot a&=a
\end{align*}
for all $r,s \in k$ and $a \in A$. We refer the reader to \cite{Golan1999Semirings} for a complete introduction to these notions, including the tensor product of $k$-modules which will be used in this paper. However, note that we use a different terminology: we call a commutative rig what is called a commutative semiring in the aforementioned reference and a $k$-module what is called a $k$-semimodule. We also introduce the following notion which is a straightforward generalization of the notion of a commutative algebra over a commutative ring. For $k$ a commutative rig, a \emph{commutative $k$-algebra} is a set $A$ together with two binary operations $+$ and $\cdot$, two constants $0$ and $1$, and an operation $\cdot\colon k \times A \rightarrow A$ (we use the notation $\cdot$ for two different operations relying on the context to make clear which one we're using) such that $(A,+,0,\cdot,1)$ is a commutative rig, $(A,+,0,\cdot)$ is a $k$-module, and $(r \cdot a) \cdot (s \cdot b)=(rs) \cdot (ab)$ for all $r,s \in k$ and $a,b \in A$.
\begin{example} \label{ex:mod}
For every commutative rig $k$, we have the additive symmetric monoidal category $(k$-$\mathsf{Mod},\otimes,k)$ of $k$-modules. 
\end{example}
\begin{example}
Specializing Example \ref{ex:mod} to a field $k$, we obtain the additive symmetric monoidal category $(k$-$\mathsf{Vec},\otimes,k)$ of $k$-vector spaces.
\end{example}
\begin{example}
Specializing Example \ref{ex:mod} to $k=\mathbb{N}$, we obtain the additive symmetric monoidal category $(\mathsf{cMon},\otimes,\mathbb{N})$ of commutative monoids.
\end{example}
\begin{example}
The category $\mathsf{Rel}$ of sets and relations can be made into an additive symmetric monoidal category $(\mathsf{Rel},\times,*)$, where $\times$ is the Cartesian product of sets and $*=\{\bullet\}$ is a one-element set. The addition of relations is given by their union and the zero morphisms are the empty relations.
\end{example}

In a second step, we add a notion of smooth function. For each object $A \in \mathsf{C}$, we require an object $TA \in \mathsf{C}$ which must be interpreted as an object of formal smooth functions in generators from $A$.\footnote{This interpretation will be adapted to the examples in the present paper which are of a rather algebraic nature. However, for examples of differential categories of a more analytical nature, where one starts with a notion of actual smooth function from a vector space $A$ to the base field $k$, it is more appropriate to interpret $TA$ as an object of smooth function on $A^*$, the dual of $A$. In any case, we can't see $TA$ as an object of actual smooth functions from $A$ to $k$ since it would make $T$ into a contravariant functor, whereas we need a covariant functor.} We add structure which lets us multiply these smooth functions and provides a constant function equal to $1$ on all of $A$. In the  definition below, the word ``modality'' comes from linear logic. Indeed, when further structure is added, this notion gives exactly the categorical semantics of the exponential modality $\wn$ of differential linear logic. 
\begin{definition} \label{alg-mod}
An \emph{algebra modality} on an additive symmetric monoidal category $(\mathsf{C},\otimes,I)$ is a tuple $(T,\mu,\eta,\mathsf{m},\mathsf{u})$, where $(T \in End(\mathsf{C}), \mu \colon TTA \rightarrow TA, \eta \colon A \rightarrow TA)$ is a monad on $\mathsf{C}$ and $\mathsf{m} \colon TA \otimes TA \rightarrow TA$ and $\mathsf{u}\colon I \rightarrow TA$ are natural transformations such that
\begin{enumerate}
\item $(TA,\mathsf{m},\mathsf{u})$ is a commutative monoid for every $A \in \mathsf{C}$;
\item $\mu$ preserves the monoid multiplication for every $A \in \mathsf{C}$, that is, $\mathsf{m}_T;\mu=\mu \otimes \mu;\mathsf{m}$.
\end{enumerate}
\end{definition}
Note that $\mu$ then preserves the monoid unit automatically so that $\mu$ is a morphism of monoids \cite[dual of Lemma 1]{BCLS2020}. Note also that the naturalities of $\mathsf{m}$ and $\mathsf{u}$ express exactly that for every morphism $f$, the morphism $Tf$ is a morphism of monoids.
\begin{example} \label{ex:algebra-mod-sym}
For every commutative rig $k$, we obtain an algebra modality $(S,\mu,\eta,\mathsf{m},\mathsf{u})$ on $(k$-$\mathsf{Mod},\otimes,k)$, where $SA=\underset{n \in \mathbb{N}}{\bigoplus}S^nA$ is the symmetric algebra on $A$. We define $S^0A=k$ and $S^1A=A$. If $n \ge 2$, for every permutation $\tau \in S_n$, we define $\tau \cdot -\colon A^{\otimes n} \rightarrow A^{\otimes n}$ as the unique $k$-linear map such that $\tau \cdot (a_1 \otimes \dots \otimes a_n)=a_{\tau^{-1}(1)} \otimes \dots \otimes a_{\tau^{-1}(n)}$.  We then define $S^nA$ as the quotient of $A^{\otimes n}$ by the smallest congruence $\sim$ on $A^{\otimes n}$ such that $x \sim y$ if there exists $\tau \in S_n$ such that $x=\tau \cdot y$. We write $a_1 \otimes_s \dots \otimes_s a_n$ for the equivalence class of $a_1 \otimes \dots \otimes a_n$ in $S^nA$. The monoid multiplication is given by
\begin{equation*}
\mathsf{m}((a_1 \otimes_s \dots \otimes_s a_n) \otimes (a_{n+1} \otimes_s \dots \otimes_s a_{n+p}))=a_1 \otimes_s \dots \otimes_s a_{n+p}.
\end{equation*}
The monoid unit $\mathsf{u}\colon k \rightarrow SA$ is the coproduct inclusion of $k=S^0A$ into $SA$. The monad unit $\eta\colon A \rightarrow SA$ is the coproduct inclusion of $A=S^1A$ into $SA$. The monad multiplication $\mu\colon SSA \rightarrow SA$ is given by
\begin{equation*}
\mu((a_1^1\otimes_s \dots \otimes_s a_{n_1}^1) \boxtimes_s \dots \boxtimes_s (a_1^p \otimes_s \dots \otimes_s a_{n_p}^p))=a_1^1\otimes_s \dots \otimes_s a_{n_1}^1 \otimes_s \dots \otimes_s a_1^p \otimes_s \dots \otimes_s a_{n_p}^p,
\end{equation*}
where $\otimes_s$ and $\boxtimes_s$ denote respectively the inner and outer symmetric tensor products.

Note that $SA$ is the $k$-module underlying the free commutative $k$-algebra on $A$.
\end{example}
\begin{example} \label{ex:algebra-mod-rel}
We have an algebra modality $(\mathcal{M},\mu,\eta,\mathsf{m},\mathsf{u})$ on $(\mathsf{Rel},\times,*)$, where $\mathcal{M}A$ is the set of all finite multisets of elements of $A$. The monoid multiplication is the function given by
\begin{equation*}
\mathsf{m}([a_1,\dots,a_n],[a_{n+1},\dots,a_{n+p}])=[a_1,\dots,a_{n+p}].
\end{equation*}
The monoid unit is the function given by $\mathsf{u}(\bullet)=[\,]$. The monad unit is the function given by $\eta(a)=[a]$.
The monad multiplication is the function given by
\begin{equation*}
\mu([[a_1^1,\dots,a_{n_1}^1],\dots,[a_1^p,\dots,a_{n_p}^p]])=[a_1^1,\dots,a_{n_1}^1,\dots,a_1^p,\dots,a_{n_p}^p].
\end{equation*}
\end{example}
\begin{example} \label{ex:algebra-mod-C-infty}
In \cite{CLLW2021}, an algebra modality $(S^\infty,\mu,\eta,\mathsf{m},\mathsf{u})$ on $\mathbb{R}$-$\mathsf{Vec}$ is defined such that $S^\infty(\mathbb{R}^n)=C^\infty(\mathbb{R}^n)$. The linear maps $\mathsf{m}_{\mathbb{R}^n}$ and $\mathsf{u}_{\mathbb{R}^n}$ are the usual commutative $\mathbb{R}$-algebra multiplication and unit of $C^\infty(\mathbb{R}^n)$ and the linear map $\eta_{\mathbb{R}^n}\colon \mathbb{R}^n \rightarrow C^\infty(\mathbb{R}^n)$ is given by $\eta_{\mathbb{R}^n}(\lambda_1,\dots,\lambda_n)=\lambda_1\pi_1+\dots+\lambda_n\pi_n$, where $\pi_i \in C^\infty(\mathbb{R}^n)$ is the $i$-th projection. In that paper, the category $\mathsf{C}^\infty$-$\mathsf{Ring}$ of $C^\infty$-rings is defined, together with a forgetful functor $V\colon\mathsf{C}^\infty$-$\mathsf{Ring} \rightarrow \mathbb{R}$-$\mathsf{Vec}$ (that is actually denoted $U$ in that paper). The monad $(S^\infty,\mu,\eta)$ on $\mathbb{R}$-$\mathsf{Vec}$ is induced by an adjunction $G \dashv V$, where $G\colon \mathbb{R}$-$\mathsf{Vec} \rightarrow \mathsf{C}^\infty$-$\mathsf{Ring}$, so that $S^\infty E$ is the free $C^\infty$-ring on the vector space $E$. We call this algebra modality the free $C^\infty$-ring (on a real vector space) algebra modality.
\end{example}
In the rest of the paper, we will often use only the underlying functor to denote an algebra modality. We will also keep implicit the additive symmetric monoidal category. We will thus talk of the algebra modalities $S$ on $k$-$\mathsf{Mod}$, $\mathcal{M}$ on $\mathsf{Rel}$, and $S^\infty$ on $\mathbb{R}$-$\mathsf{Vec}$ to designate respectively Examples \ref{ex:algebra-mod-sym}, \ref{ex:algebra-mod-rel}, and \ref{ex:algebra-mod-C-infty}.

We finally introduce the last piece in the definition of a differential category, which lets us differentiate our smooth functions. This is called a deriving transformation. In the typical examples where $\mathsf{C}=k$-$\mathsf{Mod}$, if $A$ is a free $k$-module, then the deriving transformation sends $f \in TA$ to $\underset{i \in I}{\sum}\frac{\partial f}{\partial x_i} \otimes e_i \in TA \otimes A$, where the $e_i$ form a basis of $A$.
\begin{definition} \label{def-deriving}
Let $(T,\mu,\eta,\mathsf{m},\mathsf{u})$ be an algebra modality on an additive symmetric monoidal category $(\mathsf{C},\otimes,I)$. A \emph{deriving transformation} on this algebra modality is a natural transformation $\mathsf{d}\colon TA \rightarrow TA \otimes A$ satisfying the following equations:
\begin{enumerate}
\item linear rule:
\begin{equation*}
\eta;\mathsf{d}=\mathsf{u} \otimes \mathsf{id},
\end{equation*}
\item product rule:
\begin{equation*}
\mathsf{m};\mathsf{d}=\mathsf{id}_T
 \otimes \mathsf{d};\mathsf{m} \otimes \mathsf{id}
+\mathsf{d} \otimes \mathsf{id}_T;\mathsf{id}_T \otimes \sigma_{\mathsf{id},T};\mathsf{m} \otimes \mathsf{id},
\end{equation*}
\item chain rule:
\begin{equation*}
\mu;\mathsf{d}=\mathsf{d}_T;\mu \otimes \mathsf{d};\mathsf{m} \otimes \mathsf{id},
\end{equation*}
\item interchange rule:
\begin{equation*}
\mathsf{d};\mathsf{d} \otimes \mathsf{id};\mathsf{id}_T \otimes \sigma=\mathsf{d};\mathsf{d} \otimes \mathsf{id}.
\end{equation*}
\end{enumerate}
\end{definition}
A \emph{differential modality} is an algebra modality together with a deriving transformation on this algebra modality. A \emph{differential category} is an additive symmetric monoidal category $(\mathsf{C},\otimes,I)$ together with a differential modality.

Let $T$ be an algebra modality on $(k$-$\mathsf{Mod},\otimes,k)$. The product rule for a deriving transformation $\mathsf{d}$ on $T$ can be expressed in a simpler way. First note that $TA \otimes A$ is a $TA$-module with action on pure tensors given by $a(b \otimes c):=ab \otimes c$. The product rule is then
\begin{equation*}
\mathsf{d}(ab)=a\mathsf{d}(b)+b\mathsf{d}(a)
\end{equation*}
for all $a,b \in TA$.

For differential modalities on categories of modules, we will use a few times Sweedler notation, which is a notation usually used to work with coalgebras. It consists in omitting sum symbols and summand indices and using special indices for the applications of the comultiplication, here the applications of the deriving transformation. We explain the notation through an example: We will write $s \otimes t$ instead of $\underset{i \in I}{\sum}s_i \otimes t_i$ for an element of $TA \otimes A$. Then we will write $s_{(1)} \otimes s_{(2)} \otimes t$ for $(\mathsf{d} \otimes \mathsf{id})(s \otimes t)$ instead of $\underset{i \in I, j \in J_i}{\sum}s_{i,j}^1 \otimes s_{i,j}^2 \otimes t_i$, and we will write $s_{(1)(1)} \otimes s_{(1)(2)} \otimes s_{(2)} \otimes t$ for $(\mathsf{d} \otimes \mathsf{id} \otimes \mathsf{id})(s_{(1)} \otimes s_{(2)} \otimes t)$.

The following is a basic property of deriving transformations.
\begin{proposition}
The constant rule 
\begin{equation*}
\mathsf{u};\mathsf{d}=0
\end{equation*}
is satisfied by every deriving transformation.
\end{proposition}
\begin{proof}
This is \cite[dual of Lemma 3]{BCLS2020}.
\end{proof}
We will now present examples of deriving transformations and then Lemay's question.
\begin{example} \label{ex:der-sym}
There is a deriving transformation $\mathsf{d}$ on the algebra modality $S$ on $k$-$\mathsf{Mod}$ defined by
$\mathsf{d}(\lambda)=0$ for every $\lambda \in k$ and, for every $n \ge 1$, by
\begin{equation} \label{formula-d-sym}
\mathsf{d}( m_1 \otimes_s \dots \otimes_s m_n)=\underset{1 \le k \le n}{\sum}\Big(m_1 \otimes_s \dots \otimes_s  m_{k-1} \otimes_s m_{k+1} \otimes_s \dots \otimes_s m_n \Big) \otimes m_k.
\end{equation}
If $n=1$, then Equation \ref{formula-d-sym} must be read as $\mathsf{d}(m)=1 \otimes m$. If $A$ is a free $k$-module with basis $(e_i)_{i \in I}$ then $SA \simeq k[x_i,~i \in I]$ and under this isomorphism, we have
\begin{equation*}
\mathsf{d}_A(f)=\underset{i \in I}{\sum}\frac{\partial f}{\partial x_i} \otimes e_i,
\end{equation*}
where $\frac{\partial f}{\partial x_i}$ denotes the standard partial derivative of $f$ with respect to $x_i$.\footnote{Note that any polynomial only contains a finite number of variables so that the sum defining $\mathsf{d}_A(f)$ is finite.}
\end{example}
\begin{example} \label{ex:der-rel}
There is a deriving transformation $\mathsf{d}$ on the algebra modality $\mathcal{M}$ on $\mathsf{Rel}$ defined by
\begin{equation*}
\mathsf{d}_X=\{([x_1,\dots,x_n],([x_1,\dots,x_{k-1},x_{k+1},\dots,x_n],x_k)),~n \ge 1,~1 \le k \le n,~(x_1,\dots,x_n) \in X^n \}
\end{equation*}
as a subset of $\mathcal{M}X \times (\mathcal{M}X \times X)$.
\end{example}
\begin{example}[\cite{CLLW2021}] \label{ex:C-infty}
There is a deriving transformation $\mathsf{d}$ on the algebra modality $S^\infty$ on $\mathbb{R}$-$\mathsf{Vec}$ which satisfies
\begin{equation*}
\mathsf{d}_{\mathbb{R}^n}(f)=\underset{1 \le k \le n}{\sum}\frac{\partial f}{\partial x_k} \otimes e_k,
\end{equation*}
where $(e_1,\dots,e_n)$ is the canonical basis of $\mathbb{R}^n$.
\end{example}
The question asked by Lemay is:  \emph{Can an algebra modality on an additive symmetric monoidal category admit two distinct deriving transformations?}

Lemay proved that for a specific class of algebra modalities, this is impossible.
\begin{definition} \label{def:Seely}
Let $(T,\mu,\eta,\mathsf{m},\mathsf{u})$ be an algebra modality on an additive symmetric monoidal category $(\mathsf{C},\otimes,I)$ such that $\mathsf{C}$ admits an initial object $0$ and binary coproducts\footnote{The object $0$ is then automatically a zero object and $A \oplus B$ a binary biproduct. But we prefer to introduce them as an initial object and a binary coproduct, since this matches the definitions in \cite{BCLS2020}, where Seely isomorphisms are introduced for an algebra modality on an arbitrary symmetric monoidal category. Indeed, the notion of an algebra modality can be defined without requiring the symmetric monoidal category to be additive, but we have included the additivity since our interest is in differential categories, which are always additive.} $A \oplus B$. We say that this algebra modality has the \emph{nullary Seely isomorphism} if the morphism $\chi_I:=\mathsf{u}_0\colon I \rightarrow T0$ is an isomorphism. We say that it has the \emph{binary Seely isomorphism} if the natural transformation $\chi_\otimes\colon TA \otimes TB \rightarrow T(A \oplus B)$ defined as
\begin{equation} \label{eq-binary}
\begin{tikzcd}
TA \otimes TB \arrow[rr, "T\iota^1 \otimes T\iota^2"] &  & T(A \oplus B) \otimes T(A \oplus B) \arrow[rr, "\mathsf{m}"] &  & T(A \oplus B)
\end{tikzcd}
\end{equation}
is a natural isomorphism. We say that the algebra modality $(T,\mu,\eta,\mathsf{m},\mathsf{u})$ has the \emph{Seely isomorphisms} if it has both the nullary and the binary Seely isomorphism.
\end{definition}
By the dual of \cite[Corollary~8.4]{Lemay2026Coderelictions} and using the
dual of the equivalence recalled in Section~2 of that paper between Seely isomorphisms and monoidal coalgebra modalities, we have the following.
\begin{theorem} \label{th-uniqueness}
Let $(\mathsf{C},\otimes,I)$ be an additive symmetric monoidal category such that $\mathsf{C}$ admits an initial object $0$ and binary coproducts $A \oplus B$. Let $(T,\mu,\eta,\mathsf{m},\mathsf{u})$ be an algebra modality on $(\mathsf{C},\otimes,I)$ which has the Seely isomorphisms. If $\mathsf{d}_1$ and $\mathsf{d}_2$ are two deriving transformations on this algebra modality then $\mathsf{d}_1=\mathsf{d}_2$.
\end{theorem}
\begin{example} \label{ex:unique}
The algebra modalities $S$ on $k$-$\mathsf{Mod}$ and $\mathcal{M}$ on $\mathsf{Rel}$ have the Seely isomorphisms. Therefore the deriving transformations in Examples \ref{ex:der-sym} and \ref{ex:der-rel} are the unique deriving transformations on these algebra modalities.
\end{example}
\begin{example} \label{ex:C-infty-Seely}
The algebra modality $S^\infty$ on $\mathbb{R}$-$\mathsf{Vec}$ has the nullary Seely isomorphism but it does not have the binary Seely isomorphism \cite[Remark 5.16]{CLLW2021}.
\end{example}
We will show in Section \ref{sec:unique-S-infty} that despite $S^\infty$ not having the binary Seely isomorphism, it admits a unique deriving transformation. This shows that having the Seely isomorphisms is a sufficient but not necessary condition for an algebra modality to admit at most one deriving transformation. In other words, the implication expressed by Theorem \ref{th-uniqueness} is not an equivalence. But first we will show that Theorem \ref{th-uniqueness} can be refined: the nullary Seely isomorphism is not needed as a hypothesis in that theorem.
\section{A refined uniqueness theorem} \label{sec:refined}
In informal discussions with Lemay, we have explored whether a direct computational proof of the following theorem is possible. Calculations of Lemay seem to indicate that such a direct proof exists, however we have not yet checked all the details. This potential alternative approach will be explored in further joint work with Lemay. In the present paper, we offer an indirect proof which makes use of Theorem \ref{th-uniqueness}.
\begin{theorem}\label{prop:binary-seely-uniqueness}
Let $(\mathsf C,\otimes,I)$ be an additive symmetric monoidal category
with a chosen initial object $0$ and chosen coproducts $A \oplus B$ for $A,B \in \mathsf{C}$. Let $(T,\mu,\eta,\mathsf{m},\mathsf{u})$ be an algebra modality
on $\mathsf C$ which has the binary Seely natural isomorphism.
Then $T$ admits at most one deriving transformation.
\end{theorem}
\begin{proof}
\textbf{i. Proof overview.} We will work in a symmetric monoidal category $(\mathsf{D},\otimes,T0)$, where $\mathsf{D}$ is a full subcategory of $\mathsf{C}$ and the tensor product is inherited from $(\mathsf{C},\otimes,I)$. We will show that the algebra modality $(T,\mu,\eta,\mathsf{m},\mathsf{u})$ restricts to an algebra modality on $(\mathsf{D},\otimes,T0)$ (modulo the monoid unit, which is not obtained by restriction) and that the deriving transformations on the original algebra modality restrict to deriving transformations on this restricted algebra modality. We will prove that the restricted algebra modality has both Seely isomorphisms, so that if we start with two deriving transformations on the original algebra modality, then the two restricted deriving transformations are equal. Then we will be able to deduce that the two original deriving transformations are also equal.

\textbf{ii. Construction of the subcategory and restricted algebra modality.} We define a natural transformation
\[
\ell_A:=\mathsf{u}_0 \otimes \mathsf{id}_A:A\to T0 \otimes A.
\]
Let $\mathsf D$ be the full subcategory of $\mathsf C$ whose objects are all the $A \in \mathsf{C}$ for which $\ell_A$ is an isomorphism. We will first make $\mathsf{D}$ into a symmetric monoidal category. For any two $A, B \in \mathsf{D}$, we have $\mathsf{u}_0 \otimes \mathsf{id}_{A \otimes B}=\mathsf{u}_0 \otimes \mathsf{id}_A \otimes \mathsf{id}_B$, where both $\mathsf{u}_0 \otimes \mathsf{id}_A$ and $\mathsf{id}_B$ are isomorphisms, therefore $\mathsf{u}_0 \otimes \mathsf{id}_{A \otimes B}$ is an isomorphism, hence $A \otimes B \in \mathsf{D}$. We thus obtain a bifunctor $\otimes \colon \mathsf{D} \times \mathsf{D} \rightarrow \mathsf{D}$, where both the tensor product of objects and of morphisms are defined as in $(\mathsf{C},\otimes,I)$. 

We will now prove that for every $A \in \mathsf{C}$, we have $TA \in \mathsf{D}$, which will show in particular that $T0 \in \mathsf{D}$. We must prove that for every $A \in \mathsf{C}$, the morphism $\ell_{TA}\colon TA \rightarrow T0 \otimes TA$ is an isomorphism. We will work as if the coproduct injections $\iota^1\colon A \rightarrow A \oplus 0$ and $\iota^2\colon A \rightarrow 0 \oplus A$ were equalities. For every $A \in \mathsf{C}$, we have
\begin{align*}
&T0_{0,A} \otimes \mathsf{id}_{TA};\mathsf{m}_A \\
&=T(\iota^1\colon 0 \rightarrow 0 \oplus A=A) \otimes T(\iota^2 \colon A \rightarrow 0 \oplus A=A);\mathsf{m}_A \\
&=\chi_{0,A} \colon T0 \otimes TA \rightarrow TA.
\end{align*}
Therefore, $T0_{0,A} \otimes \mathsf{id}_{TA};\mathsf{m}_A$ is an isomorphism for every $A \in \mathsf{C}$. But it is easy to prove that $\ell_{TA};T0_{0,A} \otimes \mathsf{id}_{TA};\mathsf{m}_A=\mathsf{id}_{TA}$ so that $\ell_{TA}=\chi_{0,A}^{-1}$. Hence, $\ell_{TA}$ is an isomorphism.

Our new symmetric monoidal category will be $(\mathsf{D},\otimes,\mathsf{T}0)$. The associator and the symmetry are defined as in $(\mathsf{C},\otimes,I)$. The unitors are defined as $\lambda_A:=\ell_A^{-1}\colon T0 \otimes A \rightarrow A$ and $\rho_A:=\sigma;\lambda_A\colon A \otimes T0 \rightarrow A$. We have $\lambda_A^{-1}=\ell_A$ and $\rho_A^{-1}=\mathsf{id}_A \otimes \mathsf{u}_0$. The coherence diagrams which do not involve unitors commute since they are exactly the same coherence diagrams as for $(\mathsf{C},\otimes,I)$. We thus have only two coherence diagrams to check. Recall that we work as if $(\mathsf{C},\otimes,I)$ was a symmetric strict monoidal category. We thus replace associators by equalities. One of the two coherence diagrams therefore reduces to the identity
\begin{equation} \label{non-trivial-coherence}
\mathsf{id}_A \otimes \lambda_B=\rho_A \otimes \mathsf{id}_B\colon A \otimes T0 \otimes B \rightarrow A \otimes B.
\end{equation}
But we have $\mathsf{id}_A \otimes \lambda_B^{-1}=\mathsf{id}_A \otimes \ell_B=\mathsf{id}_A \otimes \mathsf{u}_0 \otimes \mathsf{id}_B=\rho_A^{-1} \otimes \mathsf{id}_B$. We deduce that
\begin{equation} \label{non-trivial-coherence-proof}
\mathsf{id}_A \otimes \lambda_B^{-1};\mathsf{id}_A \otimes \lambda_B=\mathsf{id}_{A \otimes B}=\rho_A^{-1} \otimes \mathsf{id}_B;\rho_A \otimes \mathsf{id}_B=\mathsf{id}_A \otimes \lambda_B^{-1};\rho_A \otimes 1_B.
\end{equation}
Since $\mathsf{id}_A \otimes \lambda_B^{-1}$ is an isomorphism, we obtain Equation \ref{non-trivial-coherence} from Equation \ref{non-trivial-coherence-proof}. The second coherence diagram we have to check is
\begin{equation*}
\begin{tikzcd}
A \otimes T0 \arrow[rr, "\sigma"] \arrow[rd, "\rho_A"'] &   & T0 \otimes A \arrow[ld, "\lambda_A"] \\
                                                        & A &                                     
\end{tikzcd}
\end{equation*}
but this diagram commutes by definition of $\rho_A$. We conclude that we have obtained a symmetric monoidal category $(\mathsf{D},\otimes,T0)$. By using the zero morphism and sum of morphisms from $(\mathsf{C},\otimes,I)$, we obtain that $(\mathsf{D},\otimes,T0)$ is in fact an additive symmetric monoidal category.

We will now define an algebra modality on $(\mathsf{D},\otimes,T0)$. We know that for every $A \in \mathsf{C}$, we have $TA \in \mathsf{D}$, in particular for every $A \in \mathsf{D}$, we have $TA \in \mathsf{D}$. It follows that the monad $(T,\mu,\eta)$ on $\mathsf{C}$ restricts to a monad $(T,\mu,\eta)$ on $\mathsf{D}$. 

We will keep the same monoid multiplication as in $(\mathsf{C},\otimes,I)$ but the monoid unit will be the natural transformation $T(0_{0,A})\colon T0 \rightarrow TA$. The monoid associativity and commutativity are the same equations as before. We just have to prove unitality. But we have already noted that
\begin{align*}
&\lambda_{TA}^{-1};T0_{0,A} \otimes \mathsf{id}_{TA};\mathsf{m}_A \\
&=\ell_{TA};T0_{0,A} \otimes \mathsf{id}_{TA};\mathsf{m}_A \\
&=\mathsf{id}_{TA}
\end{align*}
which shows unitality. Since the monad and monoid multiplications are the same as in $(\mathsf{C},\otimes,I)$, item (2) in Definition \ref{alg-mod} is also satisfied, so that we obtain an algebra modality $(T,\mu,\eta,\mathsf{m},T(0_{0,A}))$ on $(\mathsf{D},\otimes,T0)$.

\textbf{iii. Finite coproducts in $\mathsf{D}$ and Seely isomorphisms.} We would like to prove that $\mathsf{D}$ has the Seely isomorphisms. We must first prove that $\mathsf{D}$ has finite coproducts. We will start by showing that $0 \in \mathsf{D}$ which is the same as proving that $\ell_0\colon 0 \rightarrow T0 \otimes 0$ is an isomorphism. First, we have $\ell_0;0_{T0 \otimes 0,0}=0_{0,0}=\mathsf{id}_0$. Note that for every $A \in \mathsf{C}$, we have $\mathsf{id}_{A \otimes 0}=\mathsf{id}_A \otimes \mathsf{id}_0=\mathsf{id}_A \otimes 0_{0,0}=0_{A \otimes 0}$, so that $A \otimes 0$ is a zero object. It follows that $0_{T0 \otimes 0,0};\ell_0=0_{T0 \otimes 0,T0 \otimes 0}=\mathsf{id}_{T0 \otimes 0}$. Thus, $\ell_0$ is an isomorphism with inverse $0_{T0 \otimes 0,0}$. Recall that since $\mathsf{C}$ is enriched over commutative monoids, we can define natural transformations $\pi^1_{A,B}\colon A\oplus B \rightarrow A$ and $\pi^2_{A,B}\colon A \oplus B \rightarrow B$ by the formulas $\pi^1_{A,B}:=(\mathsf{id}_A,0_{B,A})$ and $\pi^2_{A,B}:=(0_{A,B},\mathsf{id}_B)$. The following equations are then satisfied:
\begin{equation*}
\iota^k;\pi^l=\delta^{k,l}, \qquad \pi^1;\iota^1+\pi^2;\iota^2=\mathsf{id}_{A \oplus B},
\end{equation*}
where $k,l \in \{1,2\}$ and $\delta^{k,l}$ is defined by $\delta^{k,l}=0$ if $k \neq l$ and $\delta^{k,l}=\mathsf{id}$ if $k=l$. Let $A,B \in \mathsf{D}$. 
We will prove that $A \oplus B \in \mathsf{D}$. We must prove that $\ell_{A \oplus B}\colon A\oplus B \rightarrow T0 \otimes (A \oplus B)$ is an isomorphism. Define
\begin{equation*}
v:=\mathsf{id}_{T0} \otimes \pi^1;\lambda_A;\iota^1+\mathsf{id}_{T0} \otimes \pi^2;\lambda_B;\iota^2\colon T0 \otimes (A \oplus B) \rightarrow A \oplus B.
\end{equation*}
First, we have
\begin{align*}
\ell_{A \oplus B};v&=\mathsf{u}_0 \otimes \mathsf{id}_{A \oplus B};\mathsf{id}_{T0} \otimes \pi^1;\lambda_A;\iota^1+\mathsf{u}_0 \otimes \mathsf{id}_{A \oplus B};\mathsf{id}_{T0} \otimes \pi^2;\lambda_B;\iota^2 \\
&=\pi^1;\mathsf{u}_0 \otimes \mathsf{id}_A;\lambda_A;\iota^1+\pi^2;\mathsf{u}_0 \otimes \mathsf{id}_B;\lambda_B;\iota^2 \\
&=\pi^1;\iota^1+\pi^2;\iota^2 \\
&=\mathsf{id}_{A \oplus B}.
\end{align*}
Then, by naturality of $\ell$, we have
\begin{equation*}
\iota^1;\ell_{A \oplus B}=\ell_A;\mathsf{id}_{T0} \otimes \iota^1
\end{equation*}
and
\begin{equation*}
\iota^2;\ell_{A \oplus B}=\ell_B;\mathsf{id}_{T0} \otimes \iota^2.
\end{equation*}
It follows that
\begin{align*}
v;\ell_{A \oplus B}&=\mathsf{id}_{T0} \otimes \pi^1;\lambda_A;\iota^1;\ell_{A \oplus B}+\mathsf{id}_{T0} \otimes \pi^2;\lambda_B;\iota^2;\ell_{A \oplus B} \\
&=\mathsf{id}_{T0} \otimes \pi^1;\lambda_A;\ell_A;\mathsf{id}_{T0} \otimes \iota^1+\mathsf{id}_{T0} \otimes \pi^2;\lambda_B;\ell_B;\mathsf{id}_{T0} \otimes \iota^2 \\
&=\mathsf{id}_{T0} \otimes \pi^1;\mathsf{id}_{T0} \otimes \iota^1+\mathsf{id}_{T0} \otimes \pi^2;\mathsf{id}_{T0} \otimes \iota^2 \\
&=\mathsf{id}_{T0} \otimes (\pi^1;\iota^1)+\mathsf{id}_{T0} \otimes (\pi^2;\iota^2) \\
&=\mathsf{id}_{T0} \otimes (\pi^1;\iota^1+\pi^2;\iota^2) \\
&=\mathsf{id}_{T0 \otimes (A \oplus B)}.
\end{align*}
We conclude that $\ell_{A \oplus B}$ is an isomorphism with $\ell_{A \oplus B}^{-1}=v$, so that $A \oplus B \in \mathsf{D}$. Since the full subcategory $\mathsf{D}$ of $\mathsf{C}$ contains $0$ and is closed under the operation $-\oplus -$, we easily deduce that $0$ is an initial object in $\mathsf{D}$ and that for all $A,B \in \mathsf{D}$, the object $A \oplus B \in \mathsf{D}$ is a coproduct of $A$ and $B$ in $\mathsf{D}$. 

We will now prove that our algebra modality on $(\mathsf{D},\otimes,T0)$ has the Seely isomorphisms. The nullary Seely map in $(\mathsf{D},\otimes,T0)$ is given by $T(0_{0,0})=T(\mathsf{id}_0)=\mathsf{id}_{T0}$, therefore we have the nullary Seely isomorphism. The binary Seely natural transformation is the same as in $(\mathsf{C},\otimes,I)$. We thus also have the binary Seely isomorphism.

\textbf{iv. Uniqueness of deriving transformations.} Now, let $\mathsf{d}$ and $\mathsf{d}'$ be deriving transformations on the original algebra modality on $(\mathsf{C},\otimes,I)$. We check that they restrict to deriving transformations on $\mathsf D$. For $X\in\mathsf D$, both
$TX$ and $TX\otimes X$ belong to $\mathsf D$. We thus obtain two natural transformations from $TX$ to $TX \otimes X$ in $(\mathsf{D},\otimes,T0)$. The product, chain, and interchange rules are unchanged. We only have to prove the linear rules. Let $A \in \mathsf{D}$. We have
\begin{align*}
\eta_A;\mathsf{d}_A&=\mathsf{u}_A \otimes \mathsf{id}_A \\
&=\mathsf{u}_0 \otimes \mathsf{id}_A;T(0_{0,A}) \otimes \mathsf{id}_A \\
&=\lambda_A^{-1};T(0_{0,A}) \otimes \mathsf{id}_A
\end{align*}
which is exactly the linear rule in $(\mathsf{D},\otimes,T0)$. The same argument applies to $\mathsf{d}'$.

By Theorem \ref{th-uniqueness}, we obtain $\mathsf{d}_X=\mathsf{d}_{X}'$ for every $X \in \mathsf{D}$.

It remains to prove equality at an arbitrary $A\in\mathsf C$. For every $X \in \mathsf{D}$, we have $X \otimes A \in \mathsf{D}$. Indeed, the above proof that $\mathsf{D}$ is closed under tensor product shows in fact more generally that the tensor product of an object of $\mathsf{D}$ on the left with an object of $\mathsf{C}$ on the right is always an object of $\mathsf{D}$. Since $0 \in \mathsf{D}$, we have $T0 \in \mathsf{D}$, hence $LA:=T0 \otimes A \in \mathsf{D}$.

Naturality of $\mathsf{d}$ and $\mathsf{d}'$, together with the equality of $\mathsf{d}$ and $\mathsf{d}'$ at objects in $\mathsf{D}$ gives
\begin{equation}\label{eq:binary-seely-naturality-comparison}
d_A;T\ell_A\otimes\ell_A
= T\ell_A;\mathsf{d}_{LA}
= T\ell_A;\mathsf{d}'_{LA}
= d_A';T\ell_A\otimes\ell_A.
\end{equation}
It suffices now to show that $T\ell_A \otimes \ell_A\colon TA \otimes A \rightarrow TLA \otimes LA$ has a right inverse to obtain that $\mathsf{d}_A=\mathsf{d}'_A$. Define the morphism $r_A\colon TLA \otimes LA \rightarrow TA \otimes A$ as follows:
\begin{equation*}
r_A=\Big(T(\mathsf{id}_{T0} \otimes \eta_A);T(\chi_{0,A});\mu_A\Big) \otimes \mathsf{id}_{LA};\chi_{A,0} \otimes \mathsf{id}_A.
\end{equation*}
We first note that similarly to a previous computation, we have $\chi_{A,0}=\mathsf{id}_{TA} \otimes T0_{0,A};\mathsf{m}_A$. We now compute
\begin{align*}
&T\ell_A \otimes \ell_A;r_A \\
&=T(\mathsf{u}_0 \otimes \mathsf{id}_A) \otimes \mathsf{u}_0 \otimes \mathsf{id}_A;\Big(T(\mathsf{id}_{T0} \otimes \eta_A);T(\chi_{0,A});\mu_A\Big) \otimes \mathsf{id}_{LA};\chi_{A,0} \otimes \mathsf{id}_A \\
&=\Big(T(\mathsf{u}_0 \otimes \mathsf{id}_A);
T(\mathsf{id}_{T0} \otimes \eta_A);
T(\chi_{0,A});\mu_A\Big) \otimes \mathsf{u}_0 \otimes \mathsf{id}_A;\chi_{A,0} \otimes \mathsf{id}_A \\
&=\Big(T(\mathsf{u}_0 \otimes \eta_A);
T(T0_{0,A} \otimes \mathsf{id}_{TA});T\mathsf{m}_A;\mu_A\Big) \otimes \mathsf{u}_0 \otimes \mathsf{id}_A;\Big(\mathsf{id}_{TA} \otimes T0_{0,A};\mathsf{m}_A\Big) \otimes \mathsf{id}_A \\
&=\Bigg(\Big(T(\mathsf{u}_0 \otimes \eta_A);
T(T0_{0,A} \otimes \mathsf{id}_{TA});T\mathsf{m}_A;\mu_A\Big) \otimes \mathsf{u}_A;\mathsf{m}_A \Bigg) \otimes \mathsf{id}_A \\
&=\Big(T(\mathsf{u}_0 \otimes \eta_A);
T(T0_{0,A} \otimes \mathsf{id}_{TA});T\mathsf{m}_A;\mu_A\Big) \otimes \mathsf{id}_A \\
&=\Big(T(\mathsf{u}_A \otimes \eta_A);T\mathsf{m}_A;\mu_A\Big) \otimes \mathsf{id}_A \\
&=(T\eta_A;\mu_A) \otimes \mathsf{id}_A \\
&=\mathsf{id}_{TA} \otimes \mathsf{id}_A \\
&=\mathsf{id}_{TA \otimes A}.
\end{align*}
Since $r_A$ is a right inverse of $T\ell_A \otimes \ell_A$, we obtain from Equation \ref{eq:binary-seely-naturality-comparison} that $\mathsf{d}_A=\mathsf{d}'_A$. We conclude that $\mathsf{d}=\mathsf{d}'$.
\end{proof}
\section{Uniqueness of the deriving transformation on $S^\infty$} \label{sec:unique-S-infty}
Proving the uniqueness of the deriving transformation on $S^\infty$ will be of interest since as stated in Example \ref{ex:C-infty-Seely}, the algebra modality $S^\infty$ on $\mathbb{R}$-$\mathsf{Vec}$, which is a category with finite coproducts, does not have the binary Seely isomorphism. In addition to Theorem \ref{prop:binary-seely-uniqueness}, it will show that for an algebra modality on an additive symmetric monoidal category with finite coproducts, the binary Seely isomorphism is a sufficient but not necessary condition for uniqueness of deriving transformations. 

In addition to the information provided in Examples \ref{ex:algebra-mod-C-infty}, \ref{ex:C-infty}, and \ref{ex:C-infty-Seely}, we will need some results from \cite{CLLW2021}. Let $\mathsf{Lin}_\mathbb{R}$ be the category whose objects are the real vector spaces $\mathbb{R}^n$ for $n \ge 0$ and whose morphisms are the linear maps between them. There is an obvious embedding functor $\iota\colon \mathsf{Lin}_\mathbb{R} \rightarrow \mathbb{R}$-$\mathsf{Vec}$. Denote by $\mathsf{Fin}(\mathbb{R}$-$\mathsf{Vec},\mathbb{R}$-$\mathsf{Vec})$ the category of finitary endofunctors on $\mathbb{R}$-$\mathsf{Vec}$ and natural transformations between them. The notion of finitary functor is defined in \cite[Subsection 3.1]{CLLW2021}. For our purposes, it will be sufficient to know that both $S^\infty$ and $S^\infty(-)\otimes (-)$ are finitary endofunctors on $\mathbb{R}$-$\mathsf{Vec}$ \cite[Definition 4.2, Lemma 5.7]{CLLW2021}. We will use the following \cite[Example 3.3]{CLLW2021}: The functor
\begin{equation*}
\iota^* \colon \mathsf{Fin}(\mathbb{R}\text{-}\mathsf{Vec},\mathbb{R}\text{-}\mathsf{Vec}) \longrightarrow [\mathsf{Lin}_\mathbb{R},\mathbb{R}\text{-}\mathsf{Vec}]
\end{equation*}
given by restriction along $\iota$ is part of an equivalence of categories, in particular it is faithful. Note also that the functor $C^\infty\colon \mathsf{Lin}_\mathbb{R} \rightarrow \mathbb{R}$-$\mathsf{Vec}$ satisfies $C^\infty =S^\infty \circ \iota=\iota^*(S^\infty)$ \cite[Remark 4.6]{CLLW2021} and that $\iota^*(S^\infty(-)\otimes (-))=C^\infty(-) \otimes \iota(-)$ \cite[Lemma 5.7]{CLLW2021}. We will combine the faithfulness of $\iota^*$ with Proposition \ref{prop-rn} below to obtain the uniqueness of the deriving transformation on $S^\infty$. We will start with the following well-known result \cite[Lemma 2.8]{Nestruev2020}.
\begin{lemma}[Hadamard] \label{Had}
Let $f \in C^\infty(\mathbb{R}^n)$ for some $n \ge 1$ and let $a=(a_1,\dots,a_n) \in \mathbb{R}^n$. There exist $g_1,\dots,g_n \in C^\infty(\mathbb{R}^n)$ such that
\begin{equation} \label{eq-Had}
f(x)=f(a)+\underset{1 \le i \le n}{\sum}(x_i-a_i)g_i(x)\qquad \forall x \in \mathbb{R}^n
\end{equation}
and $g_i(a)=\frac{\partial f}{\partial x_i}(a)$.
\end{lemma}
From now on and until the end of this section, we will suppose that $\mathsf{d'}$ is a deriving transformation on the algebra modality $S^\infty$, not necessarily equal to the deriving transformation $\mathsf{d}$ from Example \ref{ex:C-infty}. For all $n \ge 1$ and $1 \le i \le n$, we define the linear map $D_i^n:=(\mathsf{id} \otimes \pi_i) \circ \mathsf{d}'_{\mathbb{R}^n} \colon C^\infty(\mathbb{R}^n) \rightarrow C^\infty(\mathbb{R}^n)$, where $\pi_i\colon \mathbb{R}^n \rightarrow \mathbb{R}$ is the $i$-th projection. Recall that an $\mathbb{R}$-linear derivation on a commutative $\mathbb{R}$-algebra $A$ is an $\mathbb{R}$-linear map $D\colon A \rightarrow A$ such that $D(ab)=D(a)b+aD(b)$.
\begin{lemma} \label{lemma-der}
For all $n \ge 1$ and $1 \le i \le n$, $D_i^n \colon C^\infty(\mathbb{R}^n) \rightarrow C^\infty(\mathbb{R}^n)$ is an $\mathbb{R}$-linear derivation.
\end{lemma}
\begin{proof}
We first note that defining $fv=(\mathsf{m} \otimes \mathsf{id})(f \otimes v) \in C^\infty(\mathbb{R}^n) \otimes \mathbb{R}^n$ for all $f \in C^\infty(\mathbb{R}^n)$ and $v \in C^\infty(\mathbb{R}^n) \otimes \mathbb{R}^n$ makes $C^\infty(\mathbb{R}^n) \otimes \mathbb{R}^n$ into a $C^\infty(\mathbb{R}^n)$-module. This is just the usual extension of scalars. In Sweedler notation, if $v=v_{(1)} \otimes v_{(2)}$, then $fv=fv_{(1)} \otimes v_{(2)}$.  Moreover, with $f$ and $v$ as before, we have $\Big((\mathsf{m} \otimes \mathsf{id}) \circ (\mathsf{id}_{C^\infty} \otimes \sigma_{\mathsf{id},C^\infty})\Big)(v \otimes f)=fv$ which is easily checked using Sweedler notation.

The product rule for $\mathsf{d}'_{\mathbb{R}^n}$ can now be written as
\begin{equation} \label{product-explicit}
\mathsf{d}'(fg)=f\mathsf{d}'(g)+g\mathsf{d}'(f)
\end{equation}
for all $f,g \in C^\infty(\mathbb{R}^n)$. Now, note that $(1 \otimes \pi_i)(fv)=f(1 \otimes \pi_i)(v)$ for all $f \in C^\infty(\mathbb{R}^n)$ and $v \in C^\infty(\mathbb{R}^n) \otimes \mathbb{R}^n$, which is proved by writing $v=v_{(1)} \otimes (v_{(2)}^1,\dots,v_{(2)}^n)$. Applying $1 \otimes \pi_i$ to both sides of Equation \ref{product-explicit}, we obtain
\begin{equation*}
D_i^n(fg)=fD_i^n(g)+gD_i^n(f).
\end{equation*}
\end{proof}
\begin{lemma} \label{lemma-uniqueness}
For all $n \ge 1$ and $1 \le i,j \le n$, we have
\begin{equation*}
D_i^n(\pi_j)=\delta_{i,j}=
\left\{
\begin{aligned}
1 \text{ if } & i=j, \\
0 \text{ if } & i\neq j. \\
\end{aligned}
\right.
\end{equation*}
\end{lemma}
\begin{proof}
The linear rule for $\mathsf{d}'_{\mathbb{R}^n}$ is $\mathsf{d}'_{\mathbb{R}^n} \circ \eta_{\mathbb{R}^n}=\mathsf{u}_{\mathbb{R}^n} \otimes \mathsf{id}_{\mathbb{R}^n}$. We thus have
\begin{equation*}
\mathsf{d}'_{\mathbb{R}^n}(\lambda_1\pi_1+\dots+\lambda_n\pi_n)=\mathsf{u}_{\mathbb{R}^n}(1) \otimes (\lambda_1,\dots,\lambda_n).
\end{equation*}
In particular, for every $1 \le j \le n$, we have
\begin{equation} \label{eq-delta}
\mathsf{d}'_{\mathbb{R}^n}(\pi_j)=\mathsf{u}_{\mathbb{R}^n}(1) \otimes e_j,
\end{equation}
where $(e_1,\dots,e_n)$ denotes the canonical basis of $\mathbb{R}^n$. Applying $1 \otimes \pi_i$ to both sides of Equation \ref{eq-delta}, we obtain
\begin{equation*}
D_i^n(\pi_j)=\delta_{i,j}\mathsf{u}_{\mathbb{R}^n}(1).
\end{equation*}
Denoting the multiplicative identity of the $\mathbb{R}$-algebra $C^\infty(\mathbb{R}^n)$ by $1$ as usual, we can simply write
\begin{equation*}
D_i^n(\pi_j)=\delta_{i,j}.
\end{equation*}
\end{proof}
\begin{lemma} \label{lemma-Di}
For all $n \ge 1$ and $1 \le i \le n$, we have $D_i^n=\frac{\partial }{\partial x_i}\colon C^\infty(\mathbb{R}^n) \rightarrow C^\infty(\mathbb{R}^n)$.
\end{lemma}
\begin{proof}
Let $a \in \mathbb{R}^n$. For any $f \in C^\infty(\mathbb{R}^n)$, we can rewrite Equation \ref{eq-Had} in the equivalent form
\begin{equation} \label{conv}
f=f(a)\cdot 1_{C^\infty(\mathbb{R}^n)}+\underset{1 \le j \le n}{\sum}(\pi_j-a_j\cdot 1_{C^\infty(\mathbb{R}^n)})g_j.
\end{equation}
Applying $D_i^n$ to both sides of Equation \ref{conv} and using Lemmas \ref{lemma-der} and \ref{lemma-uniqueness}, we obtain	
\begin{equation} \label{before-app}
D_i^n(f)=\underset{1 \le j \le n}{\sum}\Big((\pi_j-a_j\cdot 1_{C^\infty(\mathbb{R}^n)})D_i^n(g_j)+\delta_{i,j}g_j\Big)=\Big(\underset{1 \le j \le n}{\sum}(\pi_j-a_j\cdot 1_{C^\infty(\mathbb{R}^n)})D_i^n(g_j)\Big)+g_i.
\end{equation}
We have $(\pi_j-a_j\cdot 1_{C^\infty(\mathbb{R}^n)})(a)=a_j-a_j=0$ and from Lemma \ref{Had}, we have $g_i(a)=\frac{\partial f}{\partial x_i}(a)$. Applying both sides of Equation \ref{before-app} to $a$ yields
\begin{equation*}
D_i^n(f)(a)=\frac{\partial f}{\partial x_i}(a).
\end{equation*}
\end{proof}
\begin{proposition} \label{prop-rn}
For every $n \ge 0$, we have $\mathsf{d}'_{\mathbb{R}^n}=\mathsf{d}_{\mathbb{R}^n}\colon C^\infty(\mathbb{R}^n) \rightarrow C^\infty(\mathbb{R}^n) \otimes \mathbb{R}^n$.
\end{proposition}
\begin{proof}
In the case $n=0$, since $C^\infty(\mathbb{R}^0) \otimes \mathbb{R}^0 \simeq C^\infty(\mathbb{R}^0) \otimes 0 \simeq 0$, we have $\mathsf{d}'_{\mathbb{R}^0}=\mathsf{d}_{\mathbb{R}^0}=0$. Suppose now that $n \ge 1$ and let $f \in C^\infty(\mathbb{R}^n)$. Write 
\begin{equation} \label{final}
\mathsf{d}'_{\mathbb{R}^n}(f)=f_1 \otimes e_1+\dots+f_n \otimes e_n.
\end{equation}

Applying $1 \otimes \pi_i$ to both sides of Equation \ref{final}, we obtain by Lemma \ref{lemma-Di} that
\begin{equation*}
\frac{\partial f}{\partial x_i}=f_i,
\end{equation*}
so that
\begin{equation*}
\mathsf{d}'_{\mathbb{R}^n}(f)=\frac{\partial f}{\partial x_1} \otimes e_1+\dots+\frac{\partial f}{\partial x_n} \otimes e_n=\mathsf{d}_{\mathbb{R}^n}(f).
\end{equation*}
\end{proof}
\begin{corollary}
We have $\mathsf{d}'=\mathsf{d}$.
\end{corollary}
\begin{proof}
For every $n \ge 0$, we have $(\iota^*(\mathsf{d}'))_{\mathbb{R}^n}=\mathsf{d'}_{\mathbb{R}^n}=\mathsf{d}_{\mathbb{R}^n}=(\iota^*(\mathsf{d}))_{\mathbb{R}^n}$. 
We deduce the equality $\iota^*(\mathsf{d}')=\iota^*(\mathsf{d})\colon C^\infty \rightarrow C^\infty(-) \otimes \iota(-)$ in $[\mathsf{Lin}_\mathbb{R},\mathbb{R}$-$\mathsf{Vec}]$. 
By faithfulness of $\iota^*\colon \mathsf{Fin}(\mathbb{R}$-$\mathsf{Vec},\mathbb{R}$-$\mathsf{Vec}) \rightarrow [\mathsf{Lin}_\mathbb{R},\mathbb{R}$-$\mathsf{Vec}]$, we conclude that $\mathsf{d}'=\mathsf{d}\colon S^\infty \rightarrow S^\infty(-) \otimes (-)$ in $\mathsf{Fin}(\mathbb{R}$-$\mathsf{Vec},\mathbb{R}$-$\mathsf{Vec})$.
\end{proof}
We conclude that $S^\infty$ admits a unique deriving transformation.
\section{The counterexamples $S^f$ and $T^f$ to uniqueness} \label{sec:counter}
\subsection{Idea of the counterexample $S^f$ to uniqueness} \label{sec:idea-uniqueness}
We will now start the discussion of our first example $S^f$ of an algebra modality admitting infinitely many deriving transformations. This subsection is dedicated to explaining why $S^f$ is a good candidate for an algebra modality admitting several deriving transformations. The main idea is to take inspiration from derivations on polynomial rigs.

Let $k$ be a commutative rig with $1 \neq 0$.
It is fairly easy to determine the set $\mathsf{Der}(k[x])$ of all the $k$-linear derivations from $k[x]$ to $k[x]$ (that is, all the $k$-linear maps $D\colon k[x] \rightarrow k[x]$ such that $D(ab)=D(a)b+aD(b)$ and $D(1)=0$).

Suppose that $D \in \mathsf{Der}(k[x])$. Then we can prove by induction that $D(x^n)=D(x)nx^{n-1}$ for every $n \ge 1$, so that $D(x^n)=D(x)\frac{\mathrm{d}}{\mathrm{d}x}(x^n)$ for every $n \ge 0$. By linearity of $D$ and $\frac{\mathrm{d}}{\mathrm{d}x}$, it follows that 
\begin{equation*} \label{eq-for-D}
D(p)=D(x)\frac{\mathrm{d}p}{\mathrm{dx}}
\end{equation*}
for every $p \in k[x]$. From this formula, it is immediate that two $k$-linear derivations $D_1$ and $D_2$ on $k[x]$ are equal iff $D_1(x)=D_2(x)$. We have thus obtained a bijection $\phi\colon \mathsf{Der}(k[x]) \rightarrow k[x]$ given by $\phi(D)=D(x)$ and $\phi^{-1}(p)=p\frac{\mathrm{d}}{\mathrm{d}x}$. We conclude that $k[x]$ is a commutative $k$-algebra admitting infinitely many $k$-linear derivations. It should be relevant to our goal of finding an algebra modality admitting two distinct deriving transformations. We also note that given a $k$-linear derivation $D$ on $k[x]$, the following chain rule is satisfied: 
\begin{equation} \label{chain-der}
D(p[q])=\frac{\mathrm{d}p}{\mathrm{d}x}[q]D(q),
\end{equation}
where $p[q]$ denotes the polynomial obtained by substituting $x$ with $q \in k[x]$ in $p \in k[x]$. Equation \ref{chain-der} is not specific to $k$-linear derivations on $k[x]$. It is satisfied by every $k$-linear derivation $D$ on a commutative $k$-algebra $A$, for all $p \in k[x]$ and $q \in A$. 

We could think that since the commutative $k$-algebra $k[x]$ has infinitely many $k$-linear derivations, we could obtain as well infinitely many deriving transformations on the symmetric algebra modality $S$ on $k$-$\mathsf{Mod}$. Indeed, recall that $k[x]$ is the free commutative $k$-algebra on a single generator, whereas for any $k$-module $A$, the $k$-module $SA$ is the free commutative $k$-algebra on the $k$-module $A$. But we already know this is false by Theorem \ref{prop:binary-seely-uniqueness} since $S$ has the binary Seely isomorphism. We could have obtained a hint that $S$ admits a unique deriving transformation without the consideration of Seely isomorphisms by looking at $k$-linear derivations on $k[x]$. Every deriving transformation $\mathsf{d}$ on $S$ induces a $k$-linear derivation $\mathsf{d}_k\colon Sk \simeq k[x] \rightarrow Sk \otimes k \simeq k[x]$ such that $\mathsf{d}_k(x)=1$ by the linear rule. But the unique $k$-linear derivation $D$ on $k[x]$ such that $D(x)=1$ is $\frac{\mathrm{d}}{\mathrm{d}x}$. All deriving transformations on $S$ thus induce the same $k$-linear derivation on $k[x]$. Since all deriving transformations act the same way on one-variable polynomials, it looks unlikely that two of them could be distinct.

We will obtain an algebra modality with several deriving transformations by looking at algebraic expressions slightly more involved than polynomials. In addition to multiplication, addition, $0$, and $1$, we will add a unary function symbol $f$. It follows that $k[x]$ will be replaced by a set $S^f\mathbb{N}$ containing all polynomials in $\mathbb{N}[x]$ but also algebraic expressions like $f(x^2)$ or $f(f(3x+5)+4f(3))+x^3$. There will now be a lot of $\mathbb{N}$-linear derivations $D$ with $D(x)=1$ on $S^f\mathbb{N}$ since we have some freedom on the interaction of $D$ with $f$. Let $g \in S^f\mathbb{N}$. Our results will imply that there exists an $\mathbb{N}$-linear derivation $D$ on $S^f\mathbb{N}$ such that $D(x)=1$ and 
\begin{equation*}
D(f(p))=g[p]D(p)
\end{equation*}
for every $p \in S^f\mathbb{N}$, where $g[p]$ denotes the element of $S^f\mathbb{N}$ obtained by substituting $x$ with $p$ in $g$. If we choose two distinct $g_1,g_2 \in S^f\mathbb{N}$, then the $\mathbb{N}$-linear derivations $D_1$ and $D_2$ thus obtained are distinct since we have $D_i(f(x))=g_i[x]D_i(x)=g_i[x]=g_i$. The intuition here is that we're free to choose any $g \in S^f\mathbb{N}$ for the derivative of $f(x)$. For instance, if we choose $g=f(x)$, then $D$ sees $f(x)$ as if it were $\mathrm{exp}(x)$. If we choose $g=2x$, then $D$ sees $f(x)$ as if it were $x^2$ (although $f(x) \neq x^2$). 

In order to obtain an algebra modality $S^f$, we will consider the algebraic theory of commutative rigs with an endofunction. These are the tuples $(R,f)$, where $R$ is a commutative rig and $f\colon R \rightarrow R$ is any function. By a classical construction, we will obtain a free commutative rig with an endofunction monad $S^f$ on $(\mathsf{cMon},\otimes,\mathbb{N})$ which will be an algebra modality. We will then prove that there is a bijection between the set $\mathsf{Der}(S^f)$ of deriving transformations on $S^f$ and the set $S^f\mathbb{N}$. We obtain an informal correspondence between ($\mathbb{N}$-linear) derivations on $\mathbb{N}[x]$ and deriving transformations on $S^f$ which is summed up in Table \ref{table}.
\begin{table}[h]
\centering
\begin{tabular}{|c|c|}
\hline
commutative rig $\mathbb{N}$ & algebraic theory of commutative rigs OR symmetric algebra monad \\
\hline
$\mathbb{N}[x]$ & alg.\ th.\ of comm.\  rigs with an endofunction $f$ OR algebra modality $S^f$ \\
\hline
$x \in \mathbb{N}[x]$ & unary function symbol $f$ in the alg.\ th.\ of comm.\ rigs with an endofunction \\
\hline
derivations on $\mathbb{N}[x]$ & deriving transformations on $S^f$ \\
\hline
$D(x)$ for a derivation $D$ & $\mathsf{d}(f(x)) = g \otimes 1 \in S^f\mathbb{N} \otimes \mathbb{N}$ for a deriving transformation $\mathsf{d}$ \\
\hline
$D(p[q])=\frac{\mathrm{d}p}{\mathrm{d}x}[q]D(q)$ & $\mathsf{d}(f(t))=g[t]\mathsf{d}(t)$ \\
\hline
$\mathsf{Der}(\mathbb{N}[x]) \simeq \mathbb{N}[x]$ & $\mathsf{Der}(S^f) \simeq S^f\mathbb{N}$ \\
\hline
\end{tabular}
\caption{Analogy between derivations on $\mathbb{N}[x]$ and deriving transformations on $S^f$}
\label{table}
\end{table}
\subsection{The origin of the counterexample $S^f$} \label{sec:origin}
Lemay gave his presentation online at Octoberfest 2022 \cite{Lemay2022Uniqueness}. He motivated his talk by first introducing derivations on a polynomial ring, stating that the $k$-linear derivations $D$ on a polynomial ring $k[x]$ are completely determined by $D(x)$, thus there is a unique $k$-linear derivation such that $D(x)=1$. He asked whether in the same way the deriving transformation is unique in a differential category. He then explained his uniqueness result for deriving transformations on an algebra modality with both Seely isomorphisms. At the end, he stated the open question.

Some time later, I was studying Galois theory.
I also read the first pages of \cite{Magid1994} on differential Galois theory.
The author immediately explained how starting with a differential field $(k,D_1)$, that is, a field $k$ together with a $\mathbb{Z}$-linear\footnote{In this subsection, as in differential Galois theory, derivations will by default only be $\mathbb{Z}$-linear, unless otherwise specified. } derivation $D_1$ from $k$ to $k$, we can define a derivation $D_2$ on the extension field $k(x_0,\dots,x_{n-1})$ which extends $D_1$ on $k$ and such that $x_0$ satisfies any given linear differential equation $a_0y+a_1y'+\dots+a_ny^{(n)}=b$, where $n \ge 1$ and $a_n \neq 0$. It was making use of a principle similar to the bijection between the set of $k$-linear derivations on $k[x]$ and the set $k[x]$.

It thus became clear to me that the philosophy is the same for differential fields and for fields: starting from an algebraic structure, we can make it bigger by adding a solution to an equation. In each case, there are two steps: making the algebraic structure bigger, then forcing an equation to be satisfied. In the case of fields, we add a root of an irreducible polynomial $f(x) \in k[x]$, by first adding the element $x$ to the ring $k$, that is, by first considering the ring $k[x]$, then quotienting to obtain the field $k'=\frac{k[x]}{(f(x))}$. In the case of a differential field and a linear differential equation $a_0y+a_1y'+\dots+a_ny^{(n)}=b$, we first add elements $x_0,\dots,x_{n-1}$ to a differential field $(k,D_1)$, that is, we consider the field $k'=k(x_0,\dots,x_{n-1})$. Then we consider the derivation $D_2$ on $k'$ defined by $D_2(a)=D_1(a)$ if $a \in k$, $D_2(x_i)=x_{i+1}$ if $i<n-1$, and $D_2(x_{n-1})=\frac{1}{a_n}(b-a_0x_0-a_1x_1-\dots-a_{n-1}x_{n-1})$. We thus obtain a differential field $(k',D_2)$ which contains a solution $x_0$ of the differential equation $a_0y+a_1y'+\dots+a_ny^{(n)}=b$ while being an extension of the differential field $(k,D_1)$. 

The bijection between the $k$-linear derivations on $k[x]$ and the set $k[x]$ is a similar construction. We start with a field $k$, we add an element $x$ to $k$, so that we consider the ring $k[x]$, and then we consider the $k$-linear derivation $D$ on $k[x]$ defined by $D(x)=f(x)$. We thus obtain a solution $x$ of the differential equation $y'=f(y)$. We could also start with a differential field $(k,D_1)$, then define a derivation $D_2$ on $k[x]$ by $D_2(a)=D_1(a)$ if $a \in k$ and $D_2(x)=f(x)$. We could then consider the fraction field $k(x)$ of $k[x]$ and consider the extension of the derivation $D_2$ on $k[x]$ to a derivation $D_2$ on $k(x)$ (a derivation on an integral domain always extends to its fraction field). We would thus obtain an extension $(k(x),D_2)$ of the differential field $(k,D_1)$ containing a solution $x$ of the differential equation $y'=f(y)$.

We have seen a few examples where we start with an algebraic structure and make it bigger by adding a solution to an equation. In each case, we add one or more generic elements $x_i$ in the algebraic structure and then have some freedom on how to force an equation to be satisfied by one of the $x_i$'s. In the two cases involving differential fields, this freedom translates into the existence of many derivations. Applying this philosophy to deriving transformations leads to the construction of an algebra modality admitting several deriving transformations, as explained in the previous subsection. Derivations on polynomial rings (or more generally polynomial rigs) are thus a useful inspiration both to obtain uniqueness results for deriving transformations and to construct algebra modalities admitting multiple deriving transformations.

\subsection{Algebra modalities on $\mathsf{cMon}$ from algebraic theories} \label{sub:alg-th}
In this subsection, we introduce material on algebraic theories that will be used to construct our examples $S^f$ and $T^f$ of algebra modalities with many deriving transformations. We start by recalling the basic concepts of universal algebra \cite{BurrisSankappanavar2012}. The goal is to be able to sketch a proof of the left adjoint Proposition \ref{prop:left-adjoint}, in enough detail so that we will have in hand everything we need to perform our constructions of $S^f$, $T^f$, and the deriving transformations on them. Since our main proof technique for these constructions is the proof by induction, we need to know the inductive definition of the left adjoint in Proposition \ref{prop:left-adjoint}. We are not able to perform our construction from the mere existence of this left adjoint.    

We fix a countable set of variables $\mathcal{V}=\{x_1,x_2,x_3,\dots\}$. We will use the shorthands $x:=x_1$, $y:=x_2$, and $z:=x_3$ for readability. A \emph{signature} $\Sigma$ is given by a set $\Sigma(n)$ for every $n \in \mathbb{N}$, whose elements are called function symbols of arity $n$. We define \emph{terms} for the signature $\Sigma$ by induction: every variable is a term and if $f \in \Sigma(n)$ and $t_1,\dots,t_n$ are terms, then $f(t_1,\dots,t_n)$ is a term. An \emph{equation} for the signature $\Sigma$ is a string of the form $s=t$, where $s,t$ are any two $\Sigma$-terms. An \emph{algebraic theory} $\mathbb{T}$ is a pair $(\Sigma_\mathbb{T},E_\mathbb{T})$, where $\Sigma_\mathbb{T}$ is a signature and $E_\mathbb{T}$ is a set of equations for the signature $\Sigma_\mathbb{T}$.
\begin{example} \label{ex:theory-cm}
Consider the signature $\Sigma$ given by $\Sigma(0)=\{0\}$, $\Sigma(2)=\{+\}$, and $\Sigma(n)=\emptyset$ for every $n \in \mathbb{N}-\{0,2\}$. The theory of commutative monoids is given by $\mathbb{T}=(\Sigma,E)$ where
\begin{equation*}
E=\{(x+y)+z=x+(y+z),~x+y=y+x,~x+0=x\}.
\end{equation*}
\end{example}
\begin{example} \label{ex:theory-cr}
Consider the signature $\Sigma''$ given by $\Sigma''(0)=\{0,1\}$, $\Sigma''(2)=\{+,\cdot\}$, and $\Sigma''(n)=\emptyset$ for every $n \in \mathbb{N}-\{0,2\}$. The theory of commutative rigs is $\mathbb{T}''=(\Sigma'',E')$ where
\begin{align*}
E'=E~\cup~\{(x\cdot y)\cdot z=x\cdot (y \cdot z),~x \cdot y=y \cdot x,~x \cdot 1=x,~x \cdot 0=0,~(x+y)\cdot z=(x \cdot z)+(y \cdot z)\}.
\end{align*}
\end{example}
\begin{example} \label{ex:theory-cre}
The theory of commutative rigs with an endofunction is $\mathbb{T}'=(\Sigma',E')$ where $\Sigma'(n)=\Sigma''(n)$ for every $n \in \mathbb{N}-\{1\}$ and $\Sigma'(1)=\{f\}$. 
\end{example}
Suppose that $\mathbb{T}=(\Sigma_\mathbb{T},E_\mathbb{T})$ is a theory. We will define the notion of a $\mathbb{T}$-algebra. Consider a set $A$ together with an operation $\llbracket f\rrbracket_A\colon A^n \rightarrow A$ for every $f \in \Sigma_\mathbb{T}(n)$. We define an operation $\llbracket t\rrbracket_A\colon A^p \rightarrow A$ for all terms $t$ for the signature $\Sigma_\mathbb{T}$ and integers $p \ge \mathsf{sup}(t)$ where $\mathsf{sup}(t) \in \mathbb{N}$ is the least upper bound in $\mathbb{N}$ of the integers $k$ such that $x_k$ appears in $t$ (so that $\mathsf{sup}(t)$ is the largest integer $k$ such that $x_k$ appears in $t$ if there exists such an integer and $\mathsf{sup}(t)=0$ if $t$ contains no variable). This operation is defined by induction: $\llbracket x_i\rrbracket_A(a_1,\dots,a_p):=a_i$ and 
\begin{equation*}
\llbracket f(t_1,\dots,t_n)\rrbracket_A(a_1,\dots,a_p):=\llbracket f\rrbracket_A(\llbracket t_1\rrbracket_A(a_1,\dots,a_p),\dots,\llbracket t_n\rrbracket_A(a_1,\dots,a_p)).
\end{equation*}

We call the set $A$ together with the operations $\llbracket f\rrbracket_A$ a $\mathbb{T}$-\emph{algebra} if for every equation $s=t$ in $E_\mathbb{T}$, we have $\llbracket s\rrbracket_A(a_1,\dots,a_p)=\llbracket t\rrbracket_A(a_1,\dots,a_p)$ for all $a_1,\dots,a_p \in A$ where $p=\mathsf{max}(\mathsf{sup}(s),\mathsf{sup}(t))$.
\begin{example} \label{example:algebras}
Let $\mathbb{T}$, $\mathbb{T}''$, and $\mathbb{T}'$ be the algebraic theories defined in Examples \ref{ex:theory-cm}, \ref{ex:theory-cr}, and \ref{ex:theory-cre}. A $\mathbb{T}$-algebra is a commutative monoid, a $\mathbb{T}''$-algebra is a commutative rig, and a $\mathbb{T}'$-algebra is a commutative rig with an endofunction, that is, a commutative rig $R$ with an arbitrary function $f \colon R \rightarrow R$.  
\end{example}
Given any algebraic theory $\mathbb{T}$, a \emph{homomorphism} from a $\mathbb{T}$-algebra $A$ to a $\mathbb{T}$-algebra $B$ is a function $\rho\colon A \rightarrow B$ such that
\begin{equation*}
\rho(\llbracket f \rrbracket_A(a_1,\dots,a_n))=\llbracket f\rrbracket_B(\rho(a_1),\dots,\rho(a_n))
\end{equation*}
for all $f \in \Sigma_\mathbb{T}(n)$ and $(a_1,\dots,a_n) \in A^n$.
We obtain a category $\mathsf{Alg}_\mathbb{T}$ whose objects are the $\mathbb{T}$-algebras and whose morphisms are the $\mathbb{T}$-algebra homomorphisms.
\begin{example}
With $\mathbb{T}$, $\mathbb{T}''$, and $\mathbb{T}'$ as in Examples \ref{ex:theory-cm}, \ref{ex:theory-cr}, and \ref{ex:theory-cre}, we obtain the categories $\mathsf{Alg}_\mathbb{T}=\mathsf{cMon}$ of commutative monoids, $\mathsf{Alg}_{\mathbb{T}''}=\mathsf{cRig}$ of commutative rigs, and $\mathsf{Alg}_{\mathbb{T}'}=\mathsf{cRig}^\circlearrowleft$ of commutative rigs with an endofunction. A homomorphism of commutative rigs with an endofunction from $(R,f)$ to $(S,g)$ is a homomorphism of commutative rigs $\rho\colon R \rightarrow S$ such that $\rho \circ f=g \circ \rho$. 
\end{example}
Given two theories $\mathbb{T}$ and $\mathbb{T}'$, we will write $\mathbb{T} \subseteq \mathbb{T}'$ if $\Sigma_\mathbb{T}(n) \subseteq \Sigma_{\mathbb{T}'}(n)$ for every $n \in \mathbb{N}$ and $E_\mathbb{T} \subseteq E_{\mathbb{T}'}$. We then have an evident forgetful functor 
\begin{equation} \label{forget}
U\colon\mathsf{Alg}_{\mathbb{T}'} \rightarrow \mathsf{Alg}_\mathbb{T}.
\end{equation}
\begin{example}
With $\mathbb{T}$, $\mathbb{T}''$, and $\mathbb{T}'$ as in Examples \ref{ex:theory-cm}, \ref{ex:theory-cr}, and \ref{ex:theory-cre}, we have $\mathbb{T} \subseteq \mathbb{T''} \subseteq \mathbb{T}'$. We obtain corresponding forgetful functors between the categories of algebras.
\end{example}
We can now state our left adjoint proposition.
\begin{proposition} \label{prop:left-adjoint}
Let $\mathbb{T}$ and $\mathbb{T}'$ be two theories such that $\mathbb{T} \subseteq \mathbb{T}'$. The functor (\ref{forget}) admits a left adjoint
\begin{equation*}
F\colon\mathsf{Alg}_{\mathbb{T}} \rightarrow \mathsf{Alg}_{\mathbb{T}'}.
\end{equation*}
\end{proposition}
We will give a partial proof which defines explicitly all the structure but does not contain all the verifications. Indeed, these verifications are not necessary since \cite[Theorem 29]{PV2007} is a more general statement which implies Proposition \ref{prop:left-adjoint}. The following proof is a way to make explicit the definition of the left adjoint in our setting.
\begin{proof}[Proof]
\textbf{1. The set $\mathfrak{F}A$.} Let $A$ be any set. We will define a set $\mathfrak{F}A$ made of strings of characters by induction: for every element $a \in A$, it contains the string $e_a$ and for all $f \in \Sigma_{\mathbb{T}'}(n)$ and $t_1,\dots,t_n \in \mathfrak{F}A$, it contains the string $f(t_1,\dots,t_n)$. Note that by construction, an element of $\mathfrak{F}A$ is either of the form $e_a$ or of the form $f(t_1,\dots,t_n)$ and not both at the same time since these are two distinct syntactic constructions.

\textbf{2. Operations on $\mathfrak{F}A$.} Let $A$ be any set. We define
\begin{equation*} \label{def-prealgebra-onF0}
\llbracket f\rrbracket_{\mathfrak{F}A}(t_1,\dots,t_n):=f(t_1,\dots,t_n)
\end{equation*}
for all $f \in \Sigma_{\mathbb{T}'}(n)$ and $t_1,\dots,t_n \in \mathfrak{F}A$. 
We thus obtain an element $\llbracket s \rrbracket_{\mathfrak{F}A}(t_1,\dots,t_p) \in \mathfrak{F}A$ for all terms $s$, integers $p \ge \mathsf{sup}(s)$, and $t_1,\dots,t_p \in \mathfrak{F}A$.

\textbf{3. The set $FA$.} Let $A$ be a $\mathbb{T}$-algebra. 
Define $FA:=\mathfrak{F}A/\hspace{-1mm}\sim$ where $\sim$ is the smallest equivalence relation on the set $\mathfrak{F}A$ such that:
\begin{enumerate}
\item $f(e_{a_1},\dots,e_{a_n}) \sim e_{\llbracket f\rrbracket_A(a_1,\dots,a_n)}$ for all $f \in \Sigma_\mathbb{T}(n)$ and $a_1,\dots,a_n \in A$; \label{item-hom}
\item $\llbracket s\rrbracket_{\mathfrak{F}A}(t_1,\dots,t_p) \sim \llbracket t\rrbracket_{\mathfrak{F}A}(t_1,\dots,t_p)$ for all equations $s=t$ in $E_{\mathbb{T}'}$ and $t_1,\dots,t_p \in \mathfrak{F}A$, where $p\ge \mathsf{max}(\mathsf{sup}(s),\mathsf{sup}(t))$;
\item$f(t_1,\dots,t_n) \sim f(t_1',\dots,t_n')$ for all $f \in \Sigma_{\mathbb{T}'}(n)$ and $t_1,\dots,t_n,t_1',\dots,t_n' \in \mathfrak{F}A$ such that $t_i \sim t_i'$ for every $1 \le i \le n$. \label{item-cong}
\label{item-equation}
\end{enumerate}
\textbf{4. Homomorphism $F\phi\colon FA \rightarrow FB$.} Let $\phi\colon A \rightarrow B$ be a $\mathbb{T}$-algebra homomorphism. We first define a function $\mathfrak{F}\phi\colon \mathfrak{F}A \rightarrow \mathfrak{F}B$ by induction: $(\mathfrak{F}\phi)(e_a):=e_{\phi(a)}$ and $(\mathfrak{F}\phi)(f(t_1,\dots,t_n)):=f(\mathfrak{F}\phi(t_1),\dots,\mathfrak{F}\phi(t_n))$.
We then define $F\phi([t])=[\mathfrak{F}\phi(t)]$ where $[t]$ denotes the equivalence class in $FA$ of the element $t \in \mathfrak{F}A$. 

\textbf{5. Operations 
on $FA$.} Let $A$ be a $\mathbb{T}$-algebra. We define 
\begin{equation*} \label{operation-in-quotient}
\llbracket f\rrbracket_{FA}([t_1],\dots,[t_n]):=[\,\llbracket f\rrbracket_{\mathfrak{F}A}(t_1,\dots,t_n)\,]
\end{equation*}
for every $f \in \Sigma_{\mathbb{T}'}(n)$.

\textbf{6. The unit $\eta_A\colon A \rightarrow UFA$.} For every $\mathbb{T}$-algebra $A$, we set $\eta_A(a):=[e_a]$.

\textbf{7. The counit $\epsilon_A\colon FUA \rightarrow A$.} Let $A$ be a $\mathbb{T}'$-algebra. We first define a function $\delta_A\colon \mathfrak{F}UA \rightarrow A$ by induction: $\delta_A(e_a):=a$ and $\delta_A(f(s_1,\dots,s_n)):=f(\delta_A(s_1),\dots,\delta_A(s_n))$. We then define $\epsilon_A([s]):=\delta_A(s)$.
\end{proof}
Now that the partial proof of Proposition \ref{prop:left-adjoint} has been given, we will stop using brackets for the elements of $TA$ most of the time. For instance, we will write $e_a$ and $f(e_a)$ instead of $[e_a]$ and $[f(e_a)]$.
\begin{corollary} \label{cor-univ-algebra}
In the setting of Proposition \ref{prop:left-adjoint}, we obtain a monad $(T,\mu\colon TTA \rightarrow TA,\eta\colon A \rightarrow TA)$ on $\mathsf{Alg}_\mathbb{T}$ where $T:=U \circ F$ and $\mu:=U\epsilon_F$. Both $\mu$ and $T\phi$ for every homomorphism $\phi$ of $\mathbb{T}$-algebras are homomorphisms of $\mathbb{T}'$-algebras. For every $A \in \mathsf{Alg}_\mathbb{T}$, there is a bijective correspondence between:
\begin{enumerate}
\item Families of operations
\begin{equation*}
S=(\llbracket f \rrbracket_A \colon A^n \rightarrow A)_{f \in \Sigma_{\mathbb{T}'}(n) - \Sigma_\mathbb{T}(n),~n \ge 0}
\end{equation*}
such that combining the $\mathbb{T}$-algebra operations of $A$ with $S$ makes $A$ into a $\mathbb{T}'$-algebra.
\item Homomorphisms of $\mathbb{T}$-algebras $\nu\colon TA \rightarrow A$ such that $(A,\nu)$ is a $T$-algebra.
\end{enumerate}
Moreover, given $\mathbb{T}'$-algebras (equivalently $T$-algebras) $A$ and $B$, a homomorphism of $\mathbb{T}$-algebras from $A$ to $B$ is a homomorphism of $\mathbb{T}'$-algebras iff it is a homomorphism of $T$-algebras from $(A,\nu)$ to $(B,\nu)$.
\end{corollary}
\begin{proof}
We obtain a monad from the classical construction of a monad from an adjunction. That $\mu$ and $T\phi$ are homomorphisms of $\mathbb{T}'$-algebras follows directly from the definitions of $T$ and $\mu$, together with the fact that $\epsilon_A$ is a morphism in $\mathsf{Alg}_{\mathbb{T}'}$.

If $(A,\nu)$ is a $T$-algebra, we obtain a $\mathbb{T}'$-algebra by defining $\llbracket f\rrbracket_A(a_1,\dots,a_n):=\nu(f(e_{a_1},\dots,e_{a_n}))$ for the new function symbols $f \in \Sigma_{\mathbb{T}'} -\Sigma_{\mathbb{T}}$.

To prove that $A$ together with both the old and new operations satisfies the equations of $\mathbb{T}'$, we first prove by induction that 
\begin{equation*}
\llbracket s\rrbracket_A(a_1,\dots,a_n)=\nu(\llbracket s\rrbracket_{FA}(e_{a_1},\dots,e_{a_n})). 
\end{equation*}
Then for every equation $s=t$ in $E_{\mathbb{T}'} - E_\mathbb{T}$, we have
\begin{equation*}
\llbracket s\rrbracket_A(a_1,\dots,a_n)=\nu(\llbracket s\rrbracket_{FA}(e_{a_1},\dots,e_{a_n}))=\nu(\llbracket t\rrbracket_{FA}(e_{a_1},\dots,e_{a_n}))=\llbracket t\rrbracket_A(a_1,\dots,a_n).
\end{equation*}
In the opposite direction, if a $\mathbb{T}$-algebra $A$ is equipped with new operations which make $A$ into a $\mathbb{T}'$-algebra, we will define a $T$-algebra $(A,\nu)$. We first define a function $\mu\colon\mathfrak{F}A \rightarrow A$ by induction: $\mu(e_a):=a$ and $\mu(f(s_1,\dots,s_n)):=\llbracket f \rrbracket_A(\mu(s_1),\dots,\mu(s_n))$. We then define $\nu\colon FA \rightarrow A$ by $\nu([s]):=\mu(s)$.

It can be proved that the correspondence between $(1)$ and $(2)$ is bijective. Then it can also be proved that homomorphisms of $\mathbb{T}$-algebras are homomorphisms of $\mathbb{T}'$-algebras iff they are homomorphisms of $T$-algebras.
\end{proof}
Using Corollary \ref{cor-univ-algebra}, we can now construct substitution using the language of monads. 
\begin{construction} \label{def-sub}
Let $(T,\mu,\eta)$ be the monad obtained from Proposition \ref{prop:left-adjoint} in the case where $\mathbb{T}$ is the theory of commutative monoids. For every $T$-algebra $(A,\nu)$ and every $t \in A$, we define a homomorphism of $\mathbb{T}$-algebras $-[t]\colon T\mathbb{N} \rightarrow A$ as follows:
\begin{equation*}
\begin{tikzcd}
T\mathbb{N} \arrow[r, "Tt"] & TA \arrow[r, "\nu"] & A
\end{tikzcd}
\end{equation*}
where $t$ is identified with the homomorphism of $\mathbb{T}$-algebras from $\mathbb{N}$ to $A$ given by $1 \mapsto t$. 
\end{construction}
\begin{proposition} \label{formulas-sub}
In the context of Construction \ref{def-sub}, we have the following:
\begin{enumerate}
\item $-[t]$ is a homomorphism of $\mathbb{T}'$-algebras;
\item $-[\eta(a)]=Ta \colon T\mathbb{N} \rightarrow TA$ for all $\mathbb{T}$-algebras $A$ and $a \in A$;  
\item $s=s[x]$ for every $s \in T\mathbb{N}$, where $x:=\eta(1)=e_1 \in T\mathbb{N}$;
\item $x[s]=s$ for every $s \in TA$, where $x:=\eta(1)=e_1 \in T\mathbb{N}$ and $A$ is any $\mathbb{T}$-algebra;
\item $s[\phi(t)]=\phi(s[t])$ for all $\mathbb{T}'$-algebras $A$ and $B$, homomorphisms of $\mathbb{T}'$-algebras $\phi$ from $A$ to $B$, $s \in T\mathbb{N}$, and $t \in A$.
\end{enumerate}
\end{proposition}
\begin{proof}
These are obtained by simple diagram chasing by viewing homomorphisms of $\mathbb{T}'$-algebras as morphisms of $T$-algebras. Note that $(2)$ follows from $(1)$. Note also that it is often helpful in order to prove an equality of the form $a=b$ where $a,b$ are elements of a $\mathbb{T}$-algebra $X$, to see $a$ and $b$ as monoid homomorphisms from $\mathbb{N}$ to $X$. 
\end{proof}
\begin{proposition} \label{prop:alg-mod-from-alg-theory}
Let $\mathbb{T}'$ be an algebraic theory which contains the theory $\mathbb{T}''$ of commutative rigs and let $\mathbb{T}$ be the theory of commutative monoids. We have $\mathbb{T} \subseteq \mathbb{T}'$. Moreover, the monad $(T,\mu,\eta)$ on $\mathsf{cMon}$ from Corollary \ref{cor-univ-algebra} can be made into an algebra modality $(T,\mu,\eta,\mathsf{m},\mathsf{u})$ on $(\mathsf{cMon},\otimes,\mathbb{N})$.
\end{proposition}
\begin{proof}
Since $\mathbb{T}'' \subseteq \mathbb{T}'$, and $\mathbb{T} \subseteq \mathbb{T}''$, we have $\mathbb{T} \subseteq \mathbb{T}'$. It follows that for every commutative monoid $A$, $FA$ is a commutative rig. We thus obtain maps
\begin{equation*}
\mathsf{m}\colon TA \otimes TA \rightarrow TA
\end{equation*}
and
\begin{equation*}
\mathsf{u}\colon \mathbb{N} \rightarrow TA
\end{equation*}
which make $TA$ into a commutative monoid in $(\mathsf{cMon},\otimes,\mathbb{N})$. We must prove that $\mathsf{m}$ and $\mathsf{u}$ are natural transformations and that $\mu$ preserves the monoid multiplication. But the naturalities of $\mathsf{m}$ and $\mathsf{u}$ express exactly that for any additive monoid homomorphism $f\colon A \rightarrow B$, the additive monoid homomorphism $Tf\colon TA \rightarrow TB$ preserves the multiplicative unit and multiplication of $TA$. Since $Tf$ has the same underlying function as $Ff$ which is a homomorphism of $\mathbb{T}'$-algebras, hence a homomorphism of commutative rigs, this is evident. We also have to prove that $\mu$ preserves the monoid multiplication. Since for every commutative monoid $A$, $\mu_A$ is a homomorphism of $\mathbb{T}'$-algebras, hence a homomorphism of commutative rigs, this is also evident.
\end{proof}
\begin{proposition} \label{prop:chain-sub}
Let $\mathsf{d}$ be a deriving transformation on the algebra modality $(T,\mu,\eta,\mathsf{m},\mathsf{u})$ from Proposition \ref{prop:alg-mod-from-alg-theory}. For all commutative monoids $A$, $s \in T\mathbb{N}$, and $t \in TA$, we have
\begin{equation*}
\mathsf{d}_A(s[t])=\mathsf{d}_\mathbb{N}(s)[t]\mathsf{d}_A(t).
\end{equation*}
\end{proposition}
\begin{proof}
The following diagram commutes by naturality of $\mathsf{d}$ and the chain rule:
\begin{equation*}
\begin{tikzcd}
\mathbb{N} \arrow[r, "s"] & T\mathbb{N} \arrow[d, "\mathsf{d}"'] \arrow[r, "Tt"] & TTA \arrow[d, "\mathsf{d}_T"] \arrow[rr, "\mu"]     &                                                                  & TA \arrow[d, "\mathsf{d}"] \\
                          & T\mathbb{N} \arrow[r, "Tt \otimes t"']               & TTA \otimes TA \arrow[r, "\mu \otimes \mathsf{d}"'] & TA \otimes TA \otimes A \arrow[r, "\mathsf{m} \otimes \mathsf{id}"'] & TA \otimes A.              
\end{tikzcd}
\end{equation*}
The external path at the top is $\mathsf{d}_A(s[t])$, seen as a map from $\mathbb{N}$ to $TA \otimes A$. The external path at the bottom is $\mathsf{d}_\mathbb{N}(s)[t]\mathsf{d}_A(t)$, seen as a map from $\mathbb{N}$ to $TA \otimes A$.
\end{proof}
\begin{example} \label{ex:endofunction}
Let $\mathbb{T}'$ be the theory of commutative rigs with an endofunction. By Proposition \ref{prop:alg-mod-from-alg-theory}, we obtain an algebra modality $(S^f,\mu,\eta,\mathsf{m},\mathsf{u})$ on $(\mathsf{cMon},\otimes,\mathbb{N})$.
\end{example}
We will discuss the algebra modality $(S^f,\mu,\eta,\mathsf{m},\mathsf{u})$ in the next subsections. Following Proposition \ref{prop:left-adjoint}, for every commutative monoid $A$, the commutative monoid $S^fA$ is defined as $S^fA:=\mathfrak{F}^fA/\sim$, where the set $\mathfrak{F}^fA$ and the equivalence relation $\sim$ are defined by induction. We deduce three induction principles that we will use in our proofs. For every commutative monoid $A$, the inductive definition of $\mathfrak{F}^fA$ is as follows:
\begin{enumerate}
\item for every $a \in A$, $e_a \in \mathfrak{F}^fA$,
\item $0 \in \mathfrak{F}^fA$,
\item $1 \in \mathfrak{F}^fA$,
\item for all $s,t \in \mathfrak{F}^fA$, $s+t \in \mathfrak{F}^fA$,
\item for all $s,t \in \mathfrak{F}^fA$, $st \in \mathfrak{F}^fA$,
\item for every $s \in \mathfrak{F}^fA$, $f(s) \in \mathfrak{F}^fA$.
\end{enumerate}

\textbf{Induction principle 1:} To define a function $\phi \colon \mathfrak{F}^fA \rightarrow X$, it suffices to define it on the terms introduced in items $1$--$3$ and, for each of items $4$--$6$, to express $\phi$ of the introduced term as a function of the parameters and their images under $\phi$. 

The relation $\sim$ is the smallest equivalence relation on $\mathfrak{F}^fA$ such that:
\begin{enumerate}
\item $0 \sim e_0$,
\item $e_a+e_b \sim e_{a+b}$,
\item $(s+t)+u \sim s+(t+u)$,
\item $(st)u \sim s(tu)$,
\item $s+t \sim t+s$,
\item $st \sim ts$,
\item $s+0 \sim s$,
\item $s1 \sim s$,
\item $(s+t)u \sim su+tu$,
\item $s0 \sim 0$,
\item if $s_1 \sim t_1$ and $s_2 \sim t_2$ then $s_1+s_2 \sim t_1+t_2$,
\item if $s_1 \sim t_1$ and $s_2 \sim t_2$ then $s_1s_2 \sim t_1t_2$, 
\item if $s \sim t$ then $f(s) \sim f(t)$.
\end{enumerate}
\textbf{Induction principle 2:} To prove that a function $\phi \colon \mathfrak{F}^fA \rightarrow X$ is invariant on equivalence classes, it suffices to prove that $\phi$ of the LHS is equal to $\phi$ of the RHS in items $1$--$10$ and, for items $11$--$13$, to prove that whenever the hypotheses are satisfied and $\phi$ of the LHS is equal to $\phi$ of the RHS in the hypotheses, then $\phi$ of the LHS is equal to $\phi$ of the RHS in the conclusion. 

The set $S^fA$ is defined as $S^fA:=\mathfrak{F}^fA/\sim$. 

\textbf{Induction principle 3:} To prove a statement $P(s)$ for every $s \in S^fA$, it suffices to prove $P$ for the terms introduced in items $1$--$3$ of the inductive definition of $\mathfrak{F}^fA$ and, for each of items $4$--$6$, to prove that if $P$ is satisfied by each of the parameters, then $P$ is satisfied by the introduced term.

We will now prove that the algebra modality $S^f$ does not have any of the two Seely isomorphisms, so that Theorem \ref{prop:binary-seely-uniqueness} does not apply.
\subsection{Failure of Seely isomorphisms for $S^f$} \label{sec:failure-Tf}
\begin{proposition}
The homomorphism of additive commutative monoids $\chi_\mathbb{N}\colon \mathbb{N} \rightarrow S^f0$ is not an isomorphism.
\end{proposition}
\begin{proof}
Suppose that $\chi_\mathbb{N}=\mathsf{u}_0\colon \mathbb{N} \rightarrow S^f0$ is an isomorphism. Then $f(0)=\mathsf{u}(n)=n \in S^f0$ for a unique $n \in \mathbb{N}$. Let $(R,g) \in \mathsf{cRig}^\circlearrowleft$. By the adjunction $F \dashv U$, where $U\colon \mathsf{cRig}^\circlearrowleft \rightarrow \mathsf{cMon}$ is the forgetful functor which forgets the multiplication and the endofunction, there exists a unique morphism $\phi$ in $\mathsf{cRig}^\circlearrowleft$ from $(S^f0,f)$ to $(R,g)$ such that $U(\phi) \circ \eta_0=0 \in \mathsf{cMon}((0,+,0),(R,+,0))$. This equality is trivially satisfied, so that there exists a unique morphism $\phi$ in $\mathsf{cRig}^\circlearrowleft$ from $(S^f0,f)$ to $(R,g)$. We obtain that $n=\phi(n)=\phi(f(0))=g(\phi(0))=g(0)$. Choosing $R=\mathbb{N}$ and any function $g$ from $\mathbb{N}$ to $\mathbb{N}$ such that $g(0) \neq n$ leads to a contradiction. We conclude that $\chi_\mathbb{N}$ is not an isomorphism.
\end{proof}
\begin{proposition} \label{Tf-not-binary}
The homomorphism of additive commutative monoids $\chi_\otimes\colon S^f\mathbb{N} \otimes S^f\mathbb{N} \rightarrow S^f(\mathbb{N} \oplus \mathbb{N})$ is not an isomorphism.
\end{proposition}
\begin{proof}
Write $x:=e_{(1,0)}=\eta_{\mathbb{N} \oplus \mathbb{N}}(1,0)$ and $y:=e_{(0,1)}=\eta_{\mathbb{N} \oplus \mathbb{N}}(0,1)$ in $S^f(\mathbb{N}\oplus \mathbb{N})$. We then have $\chi_\otimes(s \otimes t)=s[x] t[y]$. Indeed, first, by Proposition \ref{formulas-sub} (3), we have $s=s[x] $ and $t=t[x]$ where $x:=e_1=\eta_\mathbb{N}(1)$. Then by Proposition \ref{formulas-sub} (5), we have $(S^f\iota_1)(s[x])=s[S^f\iota_1(x)]=s[x]$. In the same way, we have $(S^f\iota_2)(t[x])=t[S^f\iota_2(x)]=t[y]$. It follows that $\chi_\otimes(s \otimes t)$ is the product of $s[x]$ and $t[y]$.

Suppose now that $\chi_\otimes$ is an isomorphism. We obtain that 
\begin{equation} \label{eq:decomp}
f(xy)=s_1[x]t_1[y]+\dots+s_n[x]t_n[y] \in S^f(\mathbb{N} \oplus \mathbb{N})
\end{equation}
for some elements $s_i,t_i \in S^f\mathbb{N}$ where $1 \le i \le n$ and $n \ge 0$. Consider the commutative rig $C^\infty(\mathbb{R}^2)$ with the endofunction $\mathsf{exp}$ defined by $\mathsf{exp}(\phi)(a,b)=e^{\phi(a,b)}$ for all $\phi \in C^\infty(\mathbb{R}^2)$ and $a,b \in \mathbb{R}$. Consider also the two elements $\pi_1, \pi_2 \in C^\infty(\mathbb{R}^2)$. 
There exists a unique morphism $\phi$ in $\mathsf{cRig}^\circlearrowleft$ from $(S^f(\mathbb{N} \oplus \mathbb{N}),f)$ to $(C^\infty(\mathbb{R}^2),\mathsf{exp})$ 
such that $\phi(x)=\pi_1$ and $\phi(y)=\pi_2$. We have $\phi(f(xy))=\mathsf{exp}(\pi_1\pi_2)$, thus 
$\mathsf{exp}(\pi_1\pi_2)=\phi(s_1[x])\phi(t_1[y])+\dots+\phi(s_n[x])\phi(t_n[y])$. Applying Proposition \ref{formulas-sub} (5), we obtain
\begin{equation*}
\mathsf{exp}(\pi_1\pi_2)=s_1[\pi_1]t_1[\pi_2]+\dots+s_n[\pi_1]t_n[\pi_2].
\end{equation*}
The function $(a,b) \mapsto e^{ab}$ thus decomposes as
\begin{equation} \label{dec-exp}
e^{ab}=\phi_1(a)\psi_1(b)+\dots+\phi_n(a)\psi_n(b)
\end{equation}
for some smooth functions $\phi_i,\psi_i \in C^\infty(\mathbb{R})$. Indeed, it can be proved by induction principle 3 from Subsection \ref{sub:alg-th} that, for every $s \in S^f\mathbb{N}$, the function $s[\pi_1] \in C^\infty(\mathbb{R}^2)$ depends only on its first variable and that the function $s[\pi_2] \in C^\infty(\mathbb{R}^2)$ depends only on its second variable. But no decomposition like Equation \ref{dec-exp} exists \cite{Stack}. We conclude that $\chi_\otimes$ is not an isomorphism.
\end{proof}
Since $S^f$ does not have the binary Seely isomorphism, Theorem \ref{prop:binary-seely-uniqueness} does not apply to it. This is thus a good candidate for an algebra modality with two distinct deriving transformations. 
\subsection{Deriving transformations on $S^f$} \label{endofunction}
We are now ready to prove one of our main results. We will prove that there is a bijection between the set of deriving transformations on $S^f$ and the set $S^f\mathbb{N}$. It follows directly from the following theorem. We will write $x:=e_1=\eta_\mathbb{N}(1) \in S^f\mathbb{N}$.
\begin{theorem} \label{main-theorem}
For every $g \in S^f\mathbb{N}$,\footnote{Here, we identify $S^f\mathbb{N}$ with $S^f\mathbb{N} \otimes \mathbb{N}$ in agreement with our convention of treating symmetric monoidal categories as symmetric strict monoidal categories.} there exists a unique deriving transformation $\mathsf{d}$ on $S^f$ such that 
\begin{equation} \label{eq-for-bij}
\mathsf{d}(f(x))=g.
\end{equation} 
\end{theorem}
\begin{proof}
\textbf{i. Uniqueness.} Let $\mathsf{d}^1$ and $\mathsf{d}^2$ be two deriving transformations satisfying Equation \ref{eq-for-bij}. Let $A$ be a commutative monoid. Note that $f(x)[t]=f(t)$ for every $t \in S^fA$. Proposition \ref{prop:chain-sub} implies that for every $t \in S^fA$ where $A$ is any commutative monoid, we have
\begin{equation} \label{the-eq}
\mathsf{d}^i(f(t))=g[t]\mathsf{d}^i(t).
\end{equation}
Then the following equations are satisfied for all $i \in \{1,2\}$, $a \in A$ and $s,t \in S^fA$:
\begin{enumerate}
\item $\mathsf{d}_A^i(e_a)=1 \otimes a$,
\item $\mathsf{d}_A^i(0)=0$,
\item $\mathsf{d}_A^i(1)=0$,
\item $\mathsf{d}_A^i(s+t)=\mathsf{d}_A^i(s)+\mathsf{d}_A^i(t)$,
\item $\mathsf{d}_A^i(st)=s\mathsf{d}^i_A(t)+t\mathsf{d}^i_A(s)$,
\item $\mathsf{d}_A^i(f(t))=g[t]\mathsf{d}^i_A(t)$.
\end{enumerate} 
Some comments: $1.$ is the linear rule, $2.$ and $4.$ are the linearity of $\mathsf{d}^i$, $3.$ is the constant rule, $5.$ is the product rule, and $6.$ is Equation \ref{the-eq}. A straightforward proof by induction using induction principle $3$ from Subsection \ref{sub:alg-th} now implies that $\mathsf{d}^1_A=\mathsf{d}^2_A$. We conclude that $\mathsf{d}^1=\mathsf{d}^2$.

\textbf{ii. Existence.} Now that uniqueness has been proved, we will prove existence. %
Let $A$ be a commutative monoid. We start by defining a function $\mathsf{d}_A^0\colon \mathfrak{F}^fA \rightarrow S^fA \otimes A$ using induction principle $1$ from Subsection \ref{sub:alg-th}:
\begin{enumerate}
\item $\mathsf{d}_A^0(e_a):=[1] \otimes a$,
\item $\mathsf{d}_A^0(0):=0$,
\item $\mathsf{d}_A^0(1):=0$,
\item $\mathsf{d}_A^0(s+t):=\mathsf{d}_A^0(s)+\mathsf{d}_A^0(t)$,
\item $\mathsf{d}_A^0(st):=[s]\mathsf{d}^0_A(t)+[t]\mathsf{d}^0_A(s)$,
\item $\mathsf{d}_A^0(f(t)):=g[t]\mathsf{d}^0_A(t)$.
\end{enumerate}
We then have to prove that $\mathsf{d}^0_A$ is invariant on equivalence classes, which will show that defining $\mathsf{d}_A \colon S^fA \rightarrow S^fA \otimes A$ by $\mathsf{d}_A([t])=\mathsf{d}^0_A(t)$ is allowed. It follows from a straightforward proof using induction principle 2 from Subsection \ref{sub:alg-th}.

All that remains to be proved is that we obtain a deriving transformation, that is, we must prove that $\mathsf{d}_A$ is a homomorphism of additive commutative monoids, that $\mathsf{d}$ is a natural transformation, and that the four identities from Definition \ref{def-deriving} are satisfied. The property of being a homomorphism of additive commutative monoids, the linear, and the product rules follow directly from the definition of $\mathsf{d}_A$. The three remaining properties, that is, the naturality of $\mathsf{d}$, the chain, and interchange rules can be proved by induction principle 3. Since the naturality of $\mathsf{d}$ is straightforward, we only give the proofs of the chain and interchange rules.

\textbf{a. Chain rule.}
\begin{enumerate}
\item We must prove that $(\mu;\mathsf{d})(0)=(\mathsf{d}_{S^f};\mu \otimes \mathsf{d};\mathsf{m} \otimes \mathsf{id})(0)$. Since both the LHS and the RHS are additive monoid homomorphisms applied to $0$, they have the same value $0$.
\item We must prove that $(\mu;\mathsf{d})(1)=(\mathsf{d}_{S^f};\mu \otimes \mathsf{d};\mathsf{m} \otimes \mathsf{id})(1)$. The RHS is equal to $0$ since $\mathsf{d}_{S^f}(1)=0$. The LHS is equal to $\mathsf{d}(\mu(1))=\mathsf{d}(1)=0$.
\item We must prove that $(\mu;\mathsf{d})(e_s)=(\mathsf{d}_{S^f};\mu \otimes \mathsf{d};\mathsf{m} \otimes \mathsf{id})(e_s)$ for any $s \in S^fA$. The LHS is equal to $\mathsf{d}(\mu(e_s))=\mathsf{d}(\mu(\eta(s)))=\mathsf{d}(s)$. The RHS is equal to $(\mathsf{m} \otimes \mathsf{id})((\mu \otimes \mathsf{d})(\mathsf{d}_{S^f}(e_s)))=(\mathsf{m} \otimes \mathsf{id})((\mu \otimes \mathsf{d})(1 \otimes s))=(\mathsf{m} \otimes \mathsf{id})(1 \otimes \mathsf{d}(s))=1\mathsf{d}(s)=\mathsf{d}(s)$.
\item Suppose that 
\begin{equation} \label{hyp-1}
(\mu;\mathsf{d})(s)=(\mathsf{d}_{S^f};\mu \otimes \mathsf{d};\mathsf{m} \otimes \mathsf{id})(s)
\end{equation}
and 
\begin{equation} \label{hyp-2}
 (\mu;\mathsf{d})(t)=(\mathsf{d}_{S^f};\mu \otimes \mathsf{d};\mathsf{m} \otimes \mathsf{id})(t)
\end{equation}
for two elements $s,t \in S^fS^fA$. We obtain that $(\mu;\mathsf{d})(s+t)=(\mathsf{d}_{S^f};\mu \otimes \mathsf{d};\mathsf{m} \otimes \mathsf{id})(s+t)$ by additivity of the monoid homomorphisms.
\item Suppose that Equations \ref{hyp-1} and \ref{hyp-2} are satisfied for two elements $s,t \in S^fS^fA$. We must prove that 
\begin{equation} \label{chain-to-be-proved}
(\mu;\mathsf{d})(st)=(\mathsf{d}_{S^f};\mu \otimes \mathsf{d};\mathsf{m} \otimes \mathsf{id})(st).
\end{equation}
From Equation \ref{hyp-1}, we obtain
\begin{align} \label{deduced-from-1}
\mathsf{d}(\mu(s))&=(\mathsf{m} \otimes \mathsf{id})\Big((\mu \otimes \mathsf{d})(\mathsf{d}_{S^f}(s))\Big) \notag \\
&=(\mathsf{m} \otimes \mathsf{id})\Big((\mu \otimes \mathsf{d})(s_{(1)} \otimes s_{(2)})\Big) \notag \\
&=(\mathsf{m} \otimes \mathsf{id})(\mu(s_{(1)}) \otimes s_{(2)(1)} \otimes s_{(2)(2)}) \notag \\
&=\mu(s_{(1)})s_{(2)(1)} \otimes s_{(2)(2)}.
\end{align}
In the same way, we obtain from Equation \ref{hyp-2} that
\begin{equation} \label{deduced-from-2}
\mathsf{d}(\mu(t))=\mu(t_{(1)})t_{(2)(1)} \otimes t_{(2)(2)}.
\end{equation}
We deduce from Equations \ref{deduced-from-1} and \ref{deduced-from-2} that 
\begin{align} \label{comp-1}
\mathsf{d}(\mu(st))&=\mathsf{d}(\mu(s)\mu(t)) \notag \\
&=\mu(s)\mathsf{d}(\mu(t))+\mu(t)\mathsf{d}(\mu(s)) \notag \\
&=\mu(s)\mu(t_{(1)})t_{(2)(1)}\otimes t_{(2)(2)}+\mu(t)\mu(s_{(1)})s_{(2)(1)}\otimes s_{(2)(2)}.
\end{align}
We also compute 
\begin{align} \label{comp-2}
(\mathsf{d}_{S^f};\mu \otimes \mathsf{d};\mathsf{m} \otimes \mathsf{id})(st)&=(\mathsf{m} \otimes \mathsf{id})\Big((\mu \otimes \mathsf{d})(\mathsf{d}_{S^f}(st))\Big) \notag \\
&=(\mathsf{m} \otimes \mathsf{id})\Big((\mu \otimes \mathsf{d})(s\mathsf{d}_{S^f}(t)+t\mathsf{d}_{S^f}(s))\Big) \notag \\
&=(\mathsf{m} \otimes \mathsf{id})\Big((\mu \otimes \mathsf{d})(st_{(1)} \otimes t_{(2)}+ts_{(1)} \otimes s_{(2)})\Big) \notag \\
&=\mu(s)\mu(t_{(1)})t_{(2)(1)} \otimes t_{(2)(2)}+\mu(t)\mu(s_{(1)})s_{(2)(1)} \otimes s_{(2)(2)}.
\end{align}
We deduce Equation \ref{chain-to-be-proved} from Equations \ref{comp-1} and \ref{comp-2}.
\item Suppose that Equation \ref{hyp-1} is satisfied for some element $s \in S^fS^fA$. We deduce that Equation \ref{deduced-from-1} is satisfied. We must prove that
\begin{equation} \label{last-one}
(\mu;\mathsf{d})(f(s))=(\mathsf{d}_{S^f};\mu \otimes \mathsf{d};\mathsf{m} \otimes \mathsf{id})(f(s)).
\end{equation}
We deduce Equation \ref{last-one} from 
\begin{align*}
\mathsf{d}(\mu(f(s)))&=\mathsf{d}(f(\mu(s))) \\
&=g[\mu(s)]\mathsf{d}(\mu(s)) \\
&= \mu(g[s])\mu(s_{(1)})s_{(2)(1)} \otimes s_{(2)(2)} \qquad \footnotesize{\text{(Prop.\ \ref{formulas-sub} (5) and Eq.\ \ref{deduced-from-1})}}
\end{align*}
and
\begin{align*}
&(\mathsf{m} \otimes \mathsf{id})\Big((\mu \otimes \mathsf{d})(\mathsf{d}_{S^f}(f(s)))\Big) \\
&=(\mathsf{m} \otimes \mathsf{id})\Big((\mu \otimes \mathsf{d})(g[s]\mathsf{d}(s))\Big) \\
&=(\mathsf{m} \otimes \mathsf{id})\Big( (\mu \otimes \mathsf{d})(g[s]s_{(1)} \otimes s_{(2)}) \Big) \\
&=\mu(g[s])\mu(s_{(1)})s_{(2)(1)} \otimes s_{(2)(2)}.
\end{align*}
\end{enumerate}
\textbf{b. Interchange rule.}
\begin{enumerate}
\item We must prove that $(\mathsf{d};\mathsf{d} \otimes \mathsf{id};\mathsf{id} \otimes \sigma)(0)=(\mathsf{d};\mathsf{d} \otimes \mathsf{id})(0)$. Since both the LHS and the RHS are additive monoid homomorphisms applied to $0$, they have the same value $0$.
\item We must prove that $(\mathsf{d};\mathsf{d} \otimes \mathsf{id};\mathsf{id} \otimes \sigma)(1)=(\mathsf{d};\mathsf{d} \otimes \mathsf{id})(1)$. Since $\mathsf{d}(1)=0$, both the LHS and the RHS are equal to $0$.
\item We must prove that $(\mathsf{d};\mathsf{d} \otimes \mathsf{id};\mathsf{id} \otimes \sigma)(e_a)=(\mathsf{d};\mathsf{d} \otimes \mathsf{id})(e_a)$ for any $a \in A$. We have $(\mathsf{d};\mathsf{d} \otimes \mathsf{id})(e_a)=(\mathsf{d} \otimes \mathsf{id})(1 \otimes a)=0$ since $\mathsf{d}(1)=0$. Thus both the LHS and the RHS are equal to $0$.
\item Suppose that 
\begin{equation} \label{hyp-11}
(\mathsf{d};\mathsf{d} \otimes \mathsf{id};\mathsf{id} \otimes \sigma)(s)=(\mathsf{d};\mathsf{d} \otimes \mathsf{id})(s)
\end{equation}
and 
\begin{equation} \label{hyp-22}
(\mathsf{d};\mathsf{d} \otimes \mathsf{id};\mathsf{id} \otimes \sigma)(t)=(\mathsf{d};\mathsf{d} \otimes \mathsf{id})(t)
\end{equation}
for two elements $s,t \in S^fA$. By additivity, we obtain immediately that $(\mathsf{d};\mathsf{d} \otimes \mathsf{id};\mathsf{id} \otimes \sigma)(s+t)=(\mathsf{d};\mathsf{d} \otimes \mathsf{id})(s+t)$.
\item Suppose that Equations \ref{hyp-11} and \ref{hyp-22} are satisfied by two elements $s,t \in S^fA$. Write 
\begin{equation*}
(\mathsf{d};\mathsf{d} \otimes \mathsf{id})(s)=s_{(1)(1)} \otimes s_{(1)(2)} \otimes s_{(2)}
\end{equation*}
and
\begin{equation*}
(\mathsf{d};\mathsf{d} \otimes \mathsf{id})(t)=t_{(1)(1)} \otimes t_{(1)(2)} \otimes t_{(2)}.
\end{equation*}
Equations \ref{hyp-11} and \ref{hyp-22} can be written in Sweedler notation as
\begin{equation} \label{11-sweedler}
s_{(1)(1)} \otimes s_{(1)(2)} \otimes s_{(2)}=s_{(1)(1)} \otimes s_{(2)} \otimes s_{(1)(2)}
\end{equation}
and
\begin{equation} \label{22-sweedler}
t_{(1)(1)} \otimes t_{(1)(2)} \otimes t_{(2)}=t_{(1)(1)} \otimes t_{(2)} \otimes t_{(1)(2)}.
\end{equation}
By the product rule, we have
\begin{equation*}
\mathsf{d}(st)=st_{(1)} \otimes t_{(2)}+ts_{(1)} \otimes s_{(2)}.
\end{equation*}
By using again the product rule, we obtain that $(\mathsf{d} \otimes \mathsf{id})(\mathsf{d}(st))$ is equal to
\begin{equation} \label{d2-prod}
st_{(1)(1)} \otimes t_{(1)(2)} \otimes t_{(2)}+t_{(1)}s_{(1)} \otimes s_{(2)} \otimes t_{(2)}+ts_{(1)(1)} \otimes s_{(1)(2)} \otimes s_{(2)}+s_{(1)}t_{(1)} \otimes t_{(2)} \otimes s_{(2)}.
\end{equation}
Hence $(\mathsf{id} \otimes \sigma)\Big((\mathsf{d} \otimes \mathsf{id})(\mathsf{d}(st))\Big)$ is equal to
\begin{equation} \label{d2-prod-swap}
st_{(1)(1)} \otimes t_{(2)} \otimes t_{(1)(2)}+t_{(1)}s_{(1)} \otimes t_{(2)} \otimes s_{(2)}+ts_{(1)(1)} \otimes s_{(2)} \otimes s_{(1)(2)}+s_{(1)}t_{(1)} \otimes s_{(2)} \otimes t_{(2)}.
\end{equation}
The second and fourth terms in Equations \ref{d2-prod} and \ref{d2-prod-swap} are the same up to exchanging their positions. The first terms in Equations \ref{d2-prod} and \ref{d2-prod-swap} are equal by applying the action of $s \in S^fA$ on the $S^fA$-module $S^fA \otimes A \otimes A$ to Equation \ref{22-sweedler}. Finally, the third terms are also equal in the same way by using Equation \ref{11-sweedler} and replacing $s$ with $t$.
\item Suppose that Equation \ref{hyp-11} is satisfied by an element $s \in S^fA$. In Sweedler notation, it corresponds to Equation \ref{11-sweedler}. We have $\mathsf{d}(f(s))=g[s]\mathsf{d}(s)$, that is, $\mathsf{d}(f(s))=g[s]s_{(1)} \otimes s_{(2)}$. Write $h:=\mathsf{d}(g) \in S^f\mathbb{N}$. By Proposition \ref{prop:chain-sub},\footnote{Proposition \ref{prop:chain-sub} is stated for deriving transformations, whereas we still have not completely proved that $\mathsf{d}$ is a deriving transformation. But the proof of Proposition \ref{prop:chain-sub} uses in fact only the naturality of $\mathsf{d}$ and the chain rule, that we have already proved. We can thus use this proposition here.} we have $\mathsf{d}(g[s])=h[s]s_{(1)} \otimes s_{(2)}$. By the product rule, we obtain
\begin{equation*}
(\mathsf{d} \otimes \mathsf{id})(\mathsf{d}(f(s)))=g[s]s_{(1)(1)} \otimes s_{(1)(2)} \otimes s_{(2)}+s_{(1)}h[s]s_{(1)}' \otimes s_{(2)}' \otimes s_{(2)}
\end{equation*}
where the prime indicates an independent summation index. Applying $\mathsf{id} \otimes \sigma$ to $(\mathsf{d} \otimes \mathsf{id})(\mathsf{d}(f(s)))$ keeps the first term unchanged by Equation \ref{11-sweedler} (up to using the action of $S^fA$ on $S^fA \otimes A \otimes A$). For the second term, we obtain
\begin{align*}
&s_{(1)}h[s]s_{(1)}' \otimes s_{(2)} \otimes s_{(2)}' \\
&=s_{(1)}'h[s]s_{(1)} \otimes s_{(2)}' \otimes s_{(2)} \\
&=s_{(1)}h[s]s_{(1)}' \otimes s_{(2)}' \otimes s_{(2)}
\end{align*}
which is equal to the original second term. We conclude that
\begin{equation*}
(\mathsf{id} \otimes \sigma)\Big((\mathsf{d} \otimes \mathsf{id})(\mathsf{d}(f(s)))\Big)=(\mathsf{d} \otimes \mathsf{id})(\mathsf{d}(f(s))).
\end{equation*}
\end{enumerate}
\end{proof}
It follows from Theorem \ref{main-theorem} that the function $\mathsf{d} \mapsto \mathsf{d}(f(x))$ from the set $\mathsf{Der}(S^f)$ of deriving transformations on the algebra modality $S^f$ to the set $S^f\mathbb{N}$ is a bijection.
\begin{corollary} \label{cor:inf}
The algebra modality $S^f$ admits infinitely many deriving transformations.
\end{corollary}
\begin{proof}
It suffices to prove that the set $S^f\mathbb{N}$ is infinite. We have an injection $\mathbb{N}[x] \rightarrow S^f\mathbb{N}$ given by sending $\sum a_kx^k$ to the element written in the same way in $S^f\mathbb{N}$. Indeed, it suffices to consider the unique homomorphism $\phi$ in $\mathsf{cRig}^\circlearrowleft$ from $(S^f\mathbb{N},f)$ to $(\mathbb{N}[x],\mathsf{id})$ such that $\phi(x)=x$. Suppose that $\sum a_kx^k=\sum b_kx^k$ in $S^f\mathbb{N}$. By applying $\phi$ to both sides of this identity, we obtain that the same equality is satisfied in $\mathbb{N}[x]$, which shows the injectivity. Since the underlying set of $\mathbb{N}[x]$ is $\mathbb{N}^{(\mathbb{N})}$, we conclude that $S^f\mathbb{N}$, and thus $\mathsf{Der}(S^f)$, are at least as large as $\mathbb{N}^{(\mathbb{N})}$.
\end{proof}
\subsection{The Kleisli category of $S^f$} \label{sec:Kleisli}
A basic fact about differential categories is that the opposite of the Kleisli category of a differential category with finite coproducts is a \emph{Cartesian differential category}. (This is usually phrased for the dual notion of a \emph{coalgebraic differential category}, in the form that the coKleisli category of a coalgebraic differential category is a Cartesian differential category \cite[Lemmas 3.2.2 and 3.2.3]{BCS2009}.) We will not go into the details of the definition of a Cartesian differential category. Let us just say that a Cartesian differential category is a category with finite products, whose hom-sets are commutative monoids, and with an operator
\begin{equation*}
\begin{prooftree}
\hypo{X \overset{\phi}{\longrightarrow} Y}
\infer1[]{X \times X \overset{D_\times[\phi]}{\longrightarrow} Y}
\end{prooftree}
\end{equation*}
satisfying some basic equations from calculus. 

From Theorem \ref{main-theorem}, we know that, for every $g \in S^f\mathbb{N}$, we obtain a deriving transformation $\mathsf{d}$ on the algebra modality $S^f$ such that $\mathsf{d}(f(x))=g$. The category $\mathrm{Kl}(S^f)^{\mathrm{op}}$ is thus a Cartesian differential category. Recall that the objects of $\mathrm{Kl}(S^f)^{\mathrm{op}}$ are the objects of $\mathsf{cMon}$ and a morphism from $A$ to $B$ in $\mathrm{Kl}(S^f)^{\mathrm{op}}$ is a morphism from $B$ to $S^fA$ in $\mathsf{cMon}$.

The product of $A$ and $B$ in $\mathrm{Kl}(S^f)^\mathrm{op}$ is given by the object $A \oplus B$. The projection maps $\pi_1\colon A \rightarrow S^f(A \oplus B)$ and $\pi_2\colon B \rightarrow S^f(A \oplus B)$, understood as maps from $A \oplus B$ to $A$ and from $A \oplus B$ to $B$ in  $\mathrm{Kl}(S^f)^\mathrm{op}$, are given by $\pi_1=\iota_1;\eta_{A \oplus B}$ and $\pi_2=\iota_2;\eta_{A \oplus B}$. That is, $\pi_1(a)=e_{(a,0)}$ and $\pi_2(b)=e_{(0,b)}$. Given two maps $\phi\colon A \rightarrow S^fX$ and $\psi\colon B \rightarrow S^fX$ in $\mathsf{cMon}$, understood as maps from $X$ to $A$ and from $X$ to $B$ in $\mathrm{Kl}(S^f)^{\mathrm{op}}$, the pairing $(\phi,\psi) \colon A \oplus B \rightarrow S^fX$, understood as a map from $X$ to $A \oplus B$ in $\mathrm{Kl}(S^f)^{\mathrm{op}}$, is just the copairing of $\phi$ and $\psi$ in $\mathsf{cMon}$. The zero map from $A$ to $B$ is $0\colon B \rightarrow S^fA$ and the sum of two maps $\phi,\psi\colon B \rightarrow S^fA$ is their sum in $\mathsf{cMon}$.

We obtain a Cartesian differential operator
\begin{equation*}
\begin{prooftree}
\hypo{B \overset{\phi}{\longrightarrow} S^fA}
\infer1[]{B \overset{D_\times[\phi]}{\longrightarrow} S^f(A \oplus A)}
\end{prooftree}
\end{equation*}
given by $D_\times[\phi]:=$
\begin{equation*}
B \overset{\phi}{\longrightarrow} S^fA \overset{\mathsf{d}}{\longrightarrow} S^fA \otimes A \overset{1 \otimes \eta}{\longrightarrow} S^fA \otimes S^fA \overset{S^f\iota_1 \otimes S^f\iota_2}{\longrightarrow} S^f(A \oplus A) \otimes S^f(A \oplus A) \overset{\mathsf{m}}{\longrightarrow} S^f(A \oplus A).
\end{equation*}
For instance, let $\phi\colon \mathbb{N}^2 \rightarrow S^f(\mathbb{N}^3)$ in $\mathsf{cMon}$ be defined by $\phi(1,0)=f(x^2yz)$ and $\phi(0,1)=f(z)^2$. In the opposite of the Kleisli category of $S^f$, $\phi$ is a map from $\mathbb{N}^3$ to $\mathbb{N}^2$. Indeed, $\phi$ can be seen as a map of the form 
\begin{equation} \label{form}
(x,y,z) \longmapsto (f(x^2yz),f(z)^2).
\end{equation}

Let us write $x:=e_{(1,0,0,0,0,0)}$, $y:=e_{(0,1,0,0,0,0)}$, $z:=e_{(0,0,1,0,0,0)}$, $dx:=e_{(0,0,0,1,0,0)}$, $dy:=e_{(0,0,0,0,1,0)}$, and $dz:=e_{(0,0,0,0,0,1)}$ in $S^f(\mathbb{N}^6)$. Suppose we have chosen $g=f(x)$, in which case $f(x)$ may be interpreted as $e^x$. Then the derivative $D_\times[\phi]\colon \mathbb{N}^2 \rightarrow S^f(\mathbb{N}^6)$ is given by 
\begin{equation*}
D_\times[\phi](1,0)=f(x^2yz)(2xyzdx+x^2zdy+x^2ydz)
\end{equation*}
and
\begin{equation*}
D_\times[\phi](0,1)=2f(z)^2dz.
\end{equation*}
In the opposite of the Kleisli category, $D_\times(\phi)$ is a map from $\mathbb{N}^6$ to $\mathbb{N}^2$. If we see $\phi$ as a map like in Equation \ref{form}, then $D_\times[\phi]\colon\mathbb{N}^2 \rightarrow S^f(\mathbb{N}^6)$ may be seen as a map of the form
\begin{equation*}
(x,y,z,dx,dy,dz) \mapsto (f(x^2yz)(2xyzdx+x^2zdy+x^2ydz),2f(z)^2dz).
\end{equation*}
The opposites of the Kleisli categories of the other algebra modalities considered in this paper are Cartesian differential categories very similar to $\mathrm{Kl}(S^f)^\mathrm{op}$. We will not discuss them further.
\subsection{A variant $T^f$ with the nullary Seely isomorphism} \label{nullary-no-binary-non-unique}
We present here a variant of $S^f$ which has the nullary Seely isomorphism, while still having neither the binary Seely isomorphism nor uniqueness of deriving transformations.

If $R$ and $S$ are commutative rigs, let us say that a function $f\colon R \rightarrow S$ preserves $1$ if $f(1)=1$. Let us consider the theory of commutative rigs with an endomorphism of additive monoids preserving $1$ which is obtained by adding the equations 
\begin{equation} \label{new-equations}
f(x+y)=f(x)+f(y),\quad f(0)=0,\quad\text{and }f(1)=1
\end{equation}
to the theory of commutative rigs with an endofunction. Let us write $\mathsf{cRig}_{f,+,0,1}$ for the category of algebras for this algebraic theory. A morphism $\rho\colon (R,f) \rightarrow (S,g)$ in $\mathsf{cRig}_{f,+,0,1}$ is a rig homomorphism $\rho\colon R \rightarrow S$ such that $\rho \circ f=g \circ \rho$.

Since the theory of commutative rigs with an endomorphism of additive monoids preserving $1$ contains the theory of commutative rigs, we obtain by Proposition \ref{prop:alg-mod-from-alg-theory} an algebra modality $T^f$ on $(\mathsf{cMon},\otimes,\mathbb{N})$. Compared to $S^f$, the new Equations \ref{new-equations} will give the nullary Seely isomorphism to $T^f$.

Following Proposition \ref{prop:left-adjoint}, we will define $T^fA:=\mathfrak{F}^fA/\sim'$ where both the set $\mathfrak{F}^fA$ and the equivalence relation $\sim'$ are defined by induction. The set $\mathfrak{F}^fA$ is the set defined in Subsection \ref{sub:alg-th} so that we will still be able to use induction principles $1$ and $3$ from Subsection \ref{sub:alg-th}. Now, the equivalence relation $\sim'$ is defined using the same inductive clauses as for $\sim$, to which we add the three following ones, expressing the new equations:
\begin{enumerate}
\item $f(0) \sim' 0$,
\item $f(1) \sim' 1$,
\item $f(s+t) \sim' f(s)+f(t)$.
\end{enumerate}
We will still be able to use induction principle $2$ from Subsection \ref{sub:alg-th}, provided that the properties corresponding to these three new clauses are added.

\begin{proposition} \label{nullary-Seely-f+01}
The homomorphism of additive commutative monoids $\chi_\mathbb{N}\colon \mathbb{N} \rightarrow T^f0$ is an isomorphism.
\end{proposition}
\begin{proof}
Let $\phi$ be the unique morphism in $\mathsf{cRig}_{f,+,0,1}$ from $(T^f0,f)$ to $(\mathbb{N},\mathsf{id})$. Write $n \cdot 1:=1+\dots+1$ ($n$ times) in $T^f(0)$. We have $\phi(n \cdot 1)=n$. But $n \cdot 1=\chi_\mathbb{N}(n)$. 
Thus $\phi \circ \chi_\mathbb{N}=\mathsf{id}_\mathbb{N}$. It follows that $\chi_\mathbb{N}$ is an injection. We can thus identify $n \cdot 1$ in $T^f(0)$ with $n \in \mathbb{N}$.

We can then prove by induction principle $3$ that for every $s \in T^f(0)$, there exists $n \in \mathbb{N}$ such that $s=n$. It follows mainly from the fact that $f(n)=n$ for every $n \in \mathbb{N}$ since $f$ preserves $+,0$, and $1$. Thus $T^f(0)=\chi_\mathbb{N}(\mathbb{N})$. Hence $\chi_\mathbb{N}$ is a surjection. We conclude that $\chi_\mathbb{N}$ is an isomorphism.
\end{proof}
The first part of the proof of Proposition \ref{nullary-Seely-f+01} can be adapted to show that $\mathsf{u}_A\colon \mathbb{N} \rightarrow T^fA$ is an injection for every commutative monoid $A$. It suffices to consider the unique morphism $\phi$ in $\mathsf{cRig}_{f,+,0,1}$ from $(T^fA,f)$ to $(\mathbb{N},\mathsf{id})$ preserving the endofunction and $+,0,1$, such that $\phi(e_a)=0$ for every $a \in A$. We have $\phi(\mathsf{u}_A(n))=\phi(n \cdot 1)=n$ i.e.~$\phi \circ \mathsf{u}_A=\mathsf{id}_\mathbb{N}$. Hence $\mathsf{u}_A$ is an injection. We can thus identify $n \cdot 1$ in $T^fA$ with $n \in \mathbb{N}$. 
\begin{proposition} \label{prop:no-Seely-2-for-T-f+01}
The homomorphism of additive commutative monoids $\chi_\otimes \colon T^f\mathbb{N} \otimes T^f\mathbb{N} \rightarrow T^f\mathbb{N}^2$ is not an isomorphism.
\end{proposition}
\begin{proof}
Suppose that $\chi_\otimes$ is an isomorphism. Like in the proof of Proposition \ref{Tf-not-binary}, we obtain that Equation \ref{eq:decomp} is satisfied in $T^f(\mathbb{N} \oplus \mathbb{N})$ for some elements $s_i,t_i \in T^f\mathbb{N}$, where $1 \le i \le n$ and $n \ge 0$. Consider now the commutative rig $\mathbb{N}[[x,y]]$ with the endofunction $d\colon \mathbb{N}[[x,y]] \rightarrow \mathbb{N}[[x,y]]$ given by
\begin{equation*}
d\left(\sum_{i,j \in \mathbb{N}} a_{i,j}x^iy^j\right)=\sum_{i \in \mathbb{N}} a_{i,i}x^iy^i.
\end{equation*}
The endofunction $d$ is an endomorphism of additive monoids preserving $1$. Consider the unique morphism $\phi$ in $\mathsf{cRig}_{f,+,0,1}$ from $(T^f\mathbb{N}^2,f)$ to $(\mathbb{N}[[x,y]],d)$ such that $\phi(x)=\underset{i \in \mathbb{N}}{\sum}x^i$ and $\phi(y)=\underset{j \in \mathbb{N}}{\sum}y^j$. Applying $\phi$ to Equation \ref{eq:decomp}, we obtain
\begin{equation*}
\underset{i \in \mathbb{N}}{\sum}x^iy^i=a_1b_1+\dots+a_nb_n
\end{equation*}
where $a_k=s_k\left[\underset{i \in \mathbb{N}}{\sum}x^i\right] \in \mathbb{N}[[x]] \subseteq \mathbb{N}[[x,y]]$ and $b_k=t_k\bigg[\underset{j \in \mathbb{N}}{\sum}y^j\bigg] \in \mathbb{N}[[y]] \subseteq \mathbb{N}[[x,y]]$. Indeed, it can be proved by induction principle $3$ from Subsection \ref{sub:alg-th} that, for every $s \in T^f\mathbb{N}$, we have $s\left[\underset{i \in \mathbb{N}}{\sum}x^i\right] \in \mathbb{N}[[x]]$ and $s\left[\underset{i \in \mathbb{N}}{\sum}y^i\right] \in \mathbb{N}[[y]]$.

For every $S=\underset{\alpha \in \mathbb{N}^n}{\sum}S_\alpha x^\alpha \in \mathbb{N}[[x_1,\dots,x_n]]$, define
\begin{equation*}
\mathrm{Supp}(S):=\{\alpha \in \mathbb{N}^n~|~S_\alpha \neq 0\}.
\end{equation*} 
We have $\mathrm{Supp}(0)=\emptyset$ and $\mathrm{Supp}(S+T)=\mathrm{Supp}(S) \cup \mathrm{Supp}(T)$. Moreover, if
$S \in \mathbb{N}[[x]] \subseteq \mathbb{N}[[x,y]]$ and $T \in \mathbb{N}[[y]] \subseteq \mathbb{N}[[x,y]]$ then $\mathrm{Supp}(ST)=\mathrm{Supp}(S) \times \mathrm{Supp}(T)$. It follows that
\begin{equation*}
D=\mathrm{Supp}\left(\underset{i \in \mathbb{N}}{\sum}x^iy^i\right)=\mathrm{Supp}(a_1b_1+\dots+a_nb_n)=(\mathrm{Supp}(a_1) \times \mathrm{Supp}(b_1)) \cup \dots \cup (\mathrm{Supp}(a_n) \times \mathrm{Supp}(b_n))
\end{equation*}
where $D=\{(i,i)~|~i \in \mathbb{N}\}$ and the sets $\mathrm{Supp}(a_i)$ and $\mathrm{Supp}(b_i)$ are subsets of $\mathbb{N}$. Since $D$ is an infinite set, there exists some $i \in \{1,2,\dots,n\}$ such that $\mathrm{Supp}(a_i) \times \mathrm{Supp}(b_i)$ contains two distinct elements $(r,r)$ and $(r',r')$. We have $r \neq r'$ and both $r$ and $r'$ are in $\mathrm{Supp}(a_i)\cap\mathrm{Supp}(b_i)$. It follows that $(r,r') \in \mathrm{Supp}(a_i) \times \mathrm{Supp}(b_i) \subseteq D$. Thus $(r,r') \in D$, a contradiction. We conclude that $\chi_\otimes$ is not an isomorphism.
\end{proof}
\begin{lemma} \label{lem:g0-is-int}
For every $g \in T^f\mathbb{N}$, we have $g[0]=n$ for some $n \in \mathbb{N}$.
\end{lemma}
\begin{proof}
We first note that since $f$ preserves $+,0$, and $1$, we have $f(n)=n$ for every $n \in \mathbb{N}$, where $n$ is seen as an element of $T^f\mathbb{N}$. We can then use induction principle 3 from Subsection \ref{sub:alg-th}.
\end{proof}
\begin{proposition} \label{prop:g-is-n}
For every deriving transformation $\mathsf{d}$ on $T^f$, we have $\mathsf{d}(f(x))=n$ for some $n \in \mathbb{N}$, where $n$ is seen as an element of $T^f\mathbb{N}$.
\end{proposition}
\begin{proof}
Let $\mathsf{d}$ be a deriving transformation on $T^f$. Write $g:=\mathsf{d}(f(x)) \in T^f\mathbb{N}$ where $x:=\eta(1)=e_1 \in T^f\mathbb{N}$. Write also $x:=\eta(1,0)=e_{(1,0)} \in T^f\mathbb{N}^2$ and $y:=\eta(0,1)=e_{(0,1)}$. We have $\mathsf{d}_{\mathbb{N}^2}(f(x+y))=\mathsf{d}_{\mathbb{N}^2}(f(x)[x+y])=g[x+y]\mathsf{d}_{\mathbb{N}^2}(x+y)=g[x+y] \otimes (1,0)+g[x+y] \otimes (0,1)$ by Proposition \ref{prop:chain-sub}. 

But we also have $\mathsf{d}_{\mathbb{N}^2}(f(x+y))=\mathsf{d}_{\mathbb{N}^2}(f(x)+f(y))=\mathsf{d}_{\mathbb{N}^2}(f(x)) + \mathsf{d}_{\mathbb{N}^2}(f(y))=g[x]\mathsf{d}_{\mathbb{N}^2}(x)+g[y]\mathsf{d}_{\mathbb{N}^2}(y)=g[x] \otimes (1,0)+g[y] \otimes (0,1)$. By applying the two projections from $T^f(\mathbb{N}^2) \otimes \mathbb{N}^2$ to $T^f(\mathbb{N}^2)$, we obtain that $g[x+y]=g[x]=g[y]$. By applying $T^f\pi_1\colon T^f\mathbb{N}^2 \rightarrow T^f\mathbb{N}$ to the equality $g[x+y]=g[y]$ and using Proposition \ref{formulas-sub} (3) and (5), we obtain that $g=g[x]=g[0]$. It follows by Lemma \ref{lem:g0-is-int} that $g=n$ for some $n \in \mathbb{N}$.
\end{proof}
\begin{proposition} \label{der-T-f+01}
For every $n \in \mathbb{N}$, there exists a unique deriving transformation $\mathsf{d}$ on $T^f$ such that 
\begin{equation*}
\mathsf{d}(f(x))=n
\end{equation*} 
where $n$ is seen as an element of $T^f\mathbb{N}$.
\end{proposition}
\begin{proof}
\textbf{i. Uniqueness.} Let $n \in \mathbb{N}$. Let $\mathsf{d}^i$ for $i \in \{1,2\}$ be two deriving transformations on $T^f$ such that $\mathsf{d}^i(f(x))=n$. Let $A$ be a commutative monoid. The following formulas are then satisfied for all $a \in A$ and $s,t \in T^fA$:
\begin{enumerate}
\item $\mathsf{d}_A^i(e_a)=1 \otimes a$,
\item $\mathsf{d}_A^i(0)=0$,
\item $\mathsf{d}_A^i(1)=0$,
\item $\mathsf{d}_A^i(s+t)=\mathsf{d}_A^i(s)+\mathsf{d}_A^i(t)$,
\item $\mathsf{d}_A^i(st)=s\mathsf{d}^i_A(t)+t\mathsf{d}^i_A(s)$,
\item $\mathsf{d}_A^i(f(t))=n\mathsf{d}^i_A(t)$.
\end{enumerate} 
Using induction principle $3$, we can now prove that $\mathsf{d}^1=\mathsf{d}^2$. We thus have uniqueness. 

\textbf{ii. Existence.} For every commutative monoid $A$, we can define a function $\mathsf{d}_A^0\colon \mathfrak{F}^fA \rightarrow T^fA \otimes A$ using induction principle $1$:
\begin{enumerate}
\item $\mathsf{d}_A^0(e_a):=[1] \otimes a$,
\item $\mathsf{d}_A^0(0):=0$,
\item $\mathsf{d}_A^0(1):=0$,
\item $\mathsf{d}_A^0(s+t):=\mathsf{d}_A^0(s)+\mathsf{d}_A^0(t)$,
\item $\mathsf{d}_A^0(st):=[s]\mathsf{d}^0_A(t)+[t]\mathsf{d}^0_A(s)$,
\item $\mathsf{d}_A^0(f(t)):=[n]\mathsf{d}^0_A(t)$.
\end{enumerate}
We can prove by induction principle $2$ that $\mathsf{d}^0_A$ is invariant on equivalence classes. We can thus define a function $\mathsf{d}_A\colon T^fA \rightarrow T^fA \otimes A$ by the formula $\mathsf{d}_A([s])=\mathsf{d}_A^0(s)$.

We obtain directly from the definition of $\mathsf{d}_A$ that this is a homomorphism of commutative monoids and that the linear and product rules are satisfied. We can then prove by induction principle $3$ that $\mathsf{d}$ is a natural transformation and that the chain and interchange rules are satisfied. These proofs are exactly the same as in the proof of Theorem \ref{main-theorem}, except that $g[t]$ must be replaced with $n$.
\end{proof}
By Proposition \ref{prop:g-is-n}, every deriving transformation on $T^f$ is given by Proposition \ref{der-T-f+01}. The latter proposition thus describes all the deriving transformations on $T^f$. Since $u_\mathbb{N}\colon \mathbb{N} \rightarrow T^f\mathbb{N}$ is an injection, any two deriving transformations in Proposition \ref{der-T-f+01} associated to two different integers are distinct. We conclude that $T^f$ admits infinitely many deriving transformations.
\section{Two variants $S'$ and $T^\infty$ without the nullary Seely isomorphism} \label{sec:no-null}
\subsection{Extension, restriction, and coextension of scalars, and derivations} \label{sec:ext-res-der}
The notions of extension and restriction of scalars will play an important role in Subsections \ref{no-nullary-no-binary-unique} and \ref{sec:unique-T}. The notion of a derivation from a commutative $k$-algebra $A$ to an $A$-module $M$ will also appear in Subsection \ref{sec:unique-T} and uses restriction of scalars. This is usually kept implicit, but the combination of deriving transformations, extension, and restriction of scalars in Subsection \ref{no-nullary-no-binary-unique} forces us to make this explicit. We thus jointly introduce these concepts in this subsection.

Let $k$ be a commutative rig and let $(A,\mathsf{m}\colon A \otimes_k A \rightarrow A,\mathsf{u}\colon k \rightarrow A)$ be a commutative $k$-algebra. We will denote by $\otimes_k$ the tensor product of $k$-modules. There is a \emph{restriction of scalars} functor $U\colon A$-$\mathsf{Mod} \rightarrow k$-$\mathsf{Mod}$. If $M$ is an $A$-module, then we define $UM$ as the commutative monoid $M$, together with the action of $k$ given by $\lambda \cdot m:=\mathsf{u}(\lambda)m$. For every $k$-module $M$, the $k$-module $A \otimes_k M$ can be made into an $A$-module. The action of $A$ on $A \otimes_k M$ is given on pure tensors by $a(b \otimes m):=ab \otimes m$. We obtain the \emph{extension of scalars} functor $A \otimes_k - \colon k$-$\mathsf{Mod} \rightarrow A$-$\mathsf{Mod}$. Furthermore, extension and restriction of scalars form an adjunction $A \otimes_k - \dashv U$. The unit $v\colon M \rightarrow U(A \otimes_k M)$ is given by $v(m)=1 \otimes m$. The counit $\epsilon\colon A \otimes_k UE \rightarrow E$ is given on pure tensors by $\epsilon(a \otimes e)=ae$.

If $A$ is a commutative $k$-algebra and $M$ is an $A$-module, then a \emph{$k$-linear derivation} from $A$ to $M$ is defined as a $k$-linear map
\begin{equation*}
D\colon A \rightarrow UM
\end{equation*}
such that $D(1)=0$ and, for all $a,b \in A$,
\begin{equation*}
D(ab)=aD(b)+bD(a).
\end{equation*}
Note that if $k$ is a commutative ring, then every $k$-module is an abelian group. This applies in particular to $UM$ for any $A$-module $M$. In this case, the equation $D(1)=0$ follows automatically from the product rule since the latter implies $D(1)=2D(1)$ which gives $D(1)=0$ by subtracting $D(1)$ from each side of the equation.

A special case will play an important role in this paper. If $A$ is a commutative $k$-algebra and $M$ is a $k$-module, we obtain an $A$-module $A \otimes_k M$. By restriction of scalars, we obtain a $k$-module $U(A \otimes_k M)$ with action of $k$ on $U(A \otimes_k M)$ 
given by $\lambda \cdot (a \otimes m)=\mathsf{u}(\lambda)(a \otimes m)=\mathsf{u}(\lambda)a \otimes m=\lambda a \otimes m$. This is the same as the action of $k$ on the $k$-module $A \otimes_k M$. Therefore $U(A \otimes_k M)$ is just $A \otimes_k M$ seen as a $k$-module. A $k$-linear derivation from $A$ to $A \otimes_k M$ is thus a $k$-linear map
\begin{equation*}
D\colon A \rightarrow A \otimes_k M
\end{equation*}
such that $D(1)=0$ and $D(ab)=aD(b)+bD(a)$ for all $a,b \in A$.

Finally, we will also need another adjunction involving the restriction of scalars functor $U\colon A$-$\mathsf{Mod} \rightarrow k$-$\mathsf{Mod}$. This functor has not only a left adjoint but also a right adjoint which is the \emph{coextension of scalars} functor $k$-$\mathsf{Mod}(A,-) \colon k$-$\mathsf{Mod} \rightarrow A$-$\mathsf{Mod}$ \cite{FauskHuMay2003}. Indeed, to start with, we have a functor $k$-$\mathsf{Mod}(A,-) \colon k$-$\mathsf{Mod} \rightarrow k$-$\mathsf{Mod}$. But for every $M \in k$-$\mathsf{Mod}$, the $k$-module $k$-$\mathsf{Mod}(A,M)$ is an $A$-module with action of $A$ given by $(a\phi)(x)=\phi(ax)$. Moreover, for every $k$-linear map $\phi\colon M\rightarrow N$, the $k$-linear map $k$-$\mathsf{Mod}(A,\phi)\colon k$-$\mathsf{Mod}(A,M) \rightarrow k$-$\mathsf{Mod}(A,N)$ is in fact $A$-linear. The functor $k$-$\mathsf{Mod}(A,-) \colon k$-$\mathsf{Mod} \rightarrow k$-$\mathsf{Mod}$ thus becomes a functor $k$-$\mathsf{Mod}(A,-) \colon k$-$\mathsf{Mod} \rightarrow A$-$\mathsf{Mod}$. The unit $v\colon E \rightarrow k$-$\mathsf{Mod}(A,UE)$ of the adjunction $U \dashv k$-$\mathsf{Mod}(A,-)$ is given by $v(e)(a)=ae$. The counit $\epsilon\colon U(k$-$\mathsf{Mod}(A,M)) \rightarrow M$ is given by $\epsilon(f)=f(1)$.
\subsection{New algebra modalities by extension and restriction of scalars} \label{no-nullary-no-binary-unique}
We will describe a general construction which constructs an algebra modality on an additive symmetric monoidal category $(\mathsf{D},\otimes,I)$ starting from an algebra modality on an additive symmetric monoidal category $(\mathsf{C},\otimes,I)$ and an adjunction $F \dashv U\colon \mathsf{C} \rightarrow \mathsf{D}$ where $U$ is an additive symmetric monoidal functor. It applies in particular to the case where $\mathsf{C}$ and $\mathsf{D}$ are categories of modules related by an extension and restriction of scalars adjunction, as described in Subsection \ref{sec:ext-res-der}.
We will apply this result to obtain a free $C^\infty$-ring algebra modality on $\mathbb{Q}$-$\mathsf{Vec}$ and a free commutative ring algebra modality on $\mathsf{cMon}$. These examples are interesting since they do not possess the nullary Seely isomorphism while retaining the other properties of the free $C^\infty$-ring algebra modality on $\mathbb{R}$-$\mathsf{Vec}$ and of the free commutative ring algebra modality on $\mathsf{Ab}$: existence of a unique deriving transformation in the two cases, no binary Seely isomorphism for the free $C^\infty$-ring algebra modality, and binary Seely isomorphism for the free commutative ring algebra modality.

We start with the following lemma which can be obtained by first breaking down the monad into an adjunction, then composing two adjunctions, and finally considering the induced monad. It can also be checked directly.
\begin{lemma} \label{lemma-monad-adj}
Let $(T,\mu,\eta)$ be a monad on a category $\mathsf{C}$ and let $F \dashv U$ be an adjunction where $U\colon \mathsf{C} \rightarrow \mathsf{D}$ and $F\colon \mathsf{D} \rightarrow \mathsf{C}$. Write $v\colon A \rightarrow UFA$ and $\epsilon\colon FUA \rightarrow A$ for the unit and counit of the adjunction. We obtain a monad on $\mathsf{D}$ with underlying functor $K:=UTF$, multiplication $\mu':=UT\epsilon TF;U\mu F$, and unit $\eta':=v;U\eta F$.
\end{lemma}
\begin{example} \label{ex:new-monad-free-ring}
Consider the extension and restriction of scalars adjunction $\mathbb{Z} \otimes_\mathbb{N} - \dashv U$ where $U\colon \mathsf{Ab} \rightarrow \mathsf{cMon}$ is the forgetful functor. (Note that a commutative monoid is the same as an $\mathbb{N}$-module and an abelian group is the same as a $\mathbb{Z}$-module.) Consider also the free commutative ring monad $(S,\mu,\eta)$ on $\mathsf{Ab}$. This monad is induced by an adjunction $G \dashv V$ where $V\colon \mathsf{cRing} \rightarrow \mathsf{Ab}$ is the forgetful functor. By Lemma \ref{lemma-monad-adj}, we obtain the monad $(S',\mu',\eta')$ on $\mathsf{cMon}$ induced by the composite adjunction. Since $S(\mathbb{Z} \otimes_\mathbb{N} M)$ is the free commutative ring on $M$, we call $(S',\mu',\eta')$ the free commutative ring monad on $\mathsf{cMon}$.
\end{example}
\begin{example} \label{ex:new-monad-free-C}
Consider the extension and restriction of scalars adjunction $\mathbb{R} \otimes_\mathbb{Q} - \dashv U$ where $\otimes_\mathbb{Q}$ is the tensor product of rational vector spaces and $U\colon \mathbb{R}$-$\mathsf{Vec} \rightarrow \mathbb{Q}$-$\mathsf{Vec}$ is the restriction of scalars functor. 
Recall from Example \ref{ex:algebra-mod-C-infty} that $(S^\infty,\mu,\eta)$ is the monad induced by an adjunction $G \dashv V$, where $V\colon\mathsf{C}^\infty$-$\mathsf{Ring} \rightarrow \mathbb{R}$-$\mathsf{Vec}$ is the forgetful functor. By composing, we obtain an adjunction $G(\mathbb{R} \otimes_\mathbb{Q} -) \dashv UV$ where 
$UV\colon\mathsf{C}^\infty$-$\mathsf{Ring} \rightarrow \mathbb{Q}$-$\mathsf{Vec}$ is the forgetful functor. We obtain the monad $(T^\infty,\mu',\eta')$ induced by this adjunction by Lemma \ref{lemma-monad-adj}. Since $G(\mathbb{R} \otimes_\mathbb{Q} M)$ is the free $C^\infty$-ring on $M$, we call $(T^\infty,\mu',\eta')$ the free $C^\infty$-ring monad on $\mathbb{Q}$-$\mathsf{Vec}$.
\end{example}
We now want to transfer the structure of an algebra modality through our construction. Recall that given two symmetric monoidal categories $(\mathsf{C},\otimes,I)$ and $(\mathsf{D},\otimes,I)$, a (lax) symmetric monoidal functor $(U,m_\otimes,m_I)$ from $(\mathsf{C},\otimes,I)$ to $(\mathsf{D},\otimes,I)$ consists of a functor $U\colon \mathsf{C} \rightarrow \mathsf{D}$, a natural transformation $m_\otimes\colon UA \otimes UB \rightarrow U(A \otimes B)$, and a morphism $m_I\colon I \rightarrow UI$, satisfying a commutativity, an associativity, and two unitality equations. The proof of the next lemma consists in checking the two conditions in Definition \ref{alg-mod} by direct computations.
\begin{lemma} \label{lemma-alg-mod}
Let $(T,\mu,\eta,\mathsf{m},\mathsf{u})$ be an algebra modality on an additive symmetric monoidal category $(\mathsf{C},\otimes,I)$, let $(\mathsf{D},\otimes,I)$ be another additive symmetric monoidal category, and let $F \dashv U$ be an adjunction where $U\colon\mathsf{C} \rightarrow \mathsf{D}$ underlies a symmetric monoidal functor $(U,m_\otimes,m_I)$. We obtain an algebra modality $(K,\mu',\eta',\mathsf{m}',\mathsf{u}')$ on $(\mathsf{D},\otimes,I)$ where $(K,\mu',\eta')$ is the monad obtained by Lemma \ref{lemma-monad-adj} from $(T,\mu,\eta)$ and $F \dashv U$. The monoid multiplication is $\mathsf{m}':=m_\otimes TF;U\mathsf{m} F$ and the monoid unit is $\mathsf{u}':=m_I;U\mathsf{u} F$.
\end{lemma}
\begin{example} \label{ex:new-alg-mod-free-ring}
In Example \ref{ex:new-monad-free-ring}, the forgetful functor $U\colon \mathsf{Ab} \rightarrow \mathsf{cMon}$ can be made into a symmetric monoidal functor $(U,m_\otimes,m_I)$ where $m_\otimes\colon UE \otimes_\mathbb{N} UF \rightarrow U(E \otimes_\mathbb{Z} F)$ is given by $m_\otimes(e \otimes_\mathbb{N} f)=e \otimes_\mathbb{Z} f$ and $m_\mathbb{Z}\colon \mathbb{N} \rightarrow U\mathbb{Z}$ is the inclusion. By Lemma \ref{lemma-alg-mod}, we obtain an algebra modality $(S',\mu',\eta',\mathsf{m}',\mathsf{u}')$ on $(\mathsf{cMon},\otimes_\mathbb{N},\mathbb{N})$.
\end{example}
\begin{example} \label{ex:new-alg-mod-free-C}
In Example \ref{ex:new-monad-free-C}, the restriction of scalars functor $U\colon \mathbb{R}$-$\mathsf{Vec} \rightarrow \mathbb{Q}$-$\mathsf{Vec}$ can be made into a symmetric monoidal functor $(U,m_\otimes,m_I)$ where $m_\otimes\colon UE \otimes_\mathbb{Q} UF \rightarrow U(E \otimes_\mathbb{R} F)$ is given by $m_\otimes(e \otimes_\mathbb{Q} f)=e \otimes_\mathbb{R} f$ and $m_\mathbb{Q}\colon \mathbb{Q} \rightarrow U\mathbb{R}$ is the inclusion. By Lemma \ref{lemma-alg-mod}, we obtain an algebra modality $(T^\infty,\mu',\eta',\mathsf{m}',\mathsf{u}')$ on $(\mathbb{Q}$-$\mathsf{Vec},\otimes_\mathbb{Q},\mathbb{Q})$.
\end{example}
We will now study the Seely morphisms for the two algebra modalities on $\mathsf{cMon}$ and $\mathbb{Q}$-$\mathsf{Vec}$ from Examples \ref{ex:new-alg-mod-free-ring} and \ref{ex:new-alg-mod-free-C}. We will write $\chi_I'$ and $\chi_\otimes'$ for these Seely morphisms in order to distinguish them from the Seely morphisms $\chi_I$ and $\chi_\otimes$ on the free commutative ring algebra modality on $\mathsf{Ab}$ and on the free $C^\infty$-ring algebra modality on $\mathbb{R}$-$\mathsf{Vec}$.  
\begin{proposition}
The morphism $\chi_\mathbb{N}'\colon \mathbb{N} \rightarrow S'0$ is not an isomorphism.
\end{proposition}
\begin{proof}
We have $S'0=US(\mathbb{Z} \otimes_\mathbb{N} 0) \simeq US0 \simeq U\mathbb{Z}$. If $\chi_\mathbb{N}'$ was an isomorphism, the commutative monoids $(\mathbb{N},+,0)$ and $(\mathbb{Z},+,0)$ would be isomorphic. Since these commutative monoids are not isomorphic, we conclude that $\chi_\mathbb{N}'$ is not an isomorphism.
\end{proof}
\begin{proposition}
For all $M, N \in \mathsf{cMon}$, the morphism $\chi_\otimes'\colon S'M \otimes_\mathbb{N} S'N \rightarrow S'(M \oplus N)$ is an isomorphism.
\end{proposition}
\begin{proof}
Recall from Definition \ref{def:Seely} that $\chi_\otimes'$ is the following natural transformation:
\begin{equation*}
\begin{tikzcd}
S'M \otimes_\mathbb{N} S'N \arrow[rr, "S'\iota^1 \otimes S'\iota^2"] &  & S'(M \oplus N) \otimes_\mathbb{N} S'(M \oplus N) \arrow[rr, "\mathsf{m}'"] &  & S'(M \oplus N).
\end{tikzcd}
\end{equation*}
After unpacking the definitions of $S'$ and $\mathsf{m}'$ and some simplifications, we obtain that
\begin{equation*}
\chi_\otimes'=m_\otimes;U\chi_\otimes;US\phi
\end{equation*}
where $\phi_{M,N}\colon (\mathbb{Z} \otimes_\mathbb{N} M) \oplus (\mathbb{Z} \otimes_\mathbb{N} N) \rightarrow \mathbb{Z} \otimes_\mathbb{N} (M \oplus N)$ is the natural isomorphism whose inverse is given on pure tensors by $\phi^{-1}(n \otimes (a,b))=(n \otimes a,n \otimes b)$. The natural transformation $m_\otimes \colon UA \otimes_\mathbb{N} UB \rightarrow U(A \otimes_\mathbb{Z} B)$ is a natural isomorphism. Indeed, if $A$ and $B$ are any two abelian groups, it is easy to see that $UA \otimes_\mathbb{N} UB$ is an abelian group. We can then define a homomorphism of abelian groups $k\colon A \otimes_\mathbb{Z} B \rightarrow UA \otimes_\mathbb{N} UB$ given on pure tensors by $k(a \otimes_\mathbb{Z} b)=a \otimes_\mathbb{N} b$. We have $m_\otimes \circ k=\mathsf{id}$ and $k \circ m_\otimes=\mathsf{id}$, which shows that $m_\otimes$ is a natural isomorphism. But we know from Example \ref{ex:unique} that $\chi_\otimes$ is a natural isomorphism. By composition of natural isomorphisms, we conclude that $\chi_\otimes'$ is a natural isomorphism.
\end{proof}
\begin{proposition}
The morphism $\chi_\mathbb{Q}'\colon \mathbb{Q} \rightarrow T^\infty 0$ is not an isomorphism.
\end{proposition}
\begin{proof}
We have $T^\infty 0=US^\infty(\mathbb{R} \otimes_\mathbb{Q} 0) \simeq US^\infty(0) \simeq US^\infty(\mathbb{R}^0)=UC^\infty(\mathbb{R}^0) \simeq U\mathbb{R}$, that is, $T^\infty 0 \in \mathbb{Q}$-$\mathsf{Vec}$ is isomorphic to $\mathbb{R}$ as a rational vector space. But $\mathbb{R}$ is not isomorphic to $\mathbb{Q}$ in $\mathbb{Q}$-$\mathsf{Vec}$. For instance, the family $(1,\sqrt{2})$ is $\mathbb{Q}$-linearly independent since $\sqrt{2}$ is irrational, thus $\mathsf{dim}_\mathbb{Q}(\mathbb{R}) \ge 2$, but $\mathsf{dim}_\mathbb{Q}(\mathbb{Q})=1$.
\end{proof}  
\begin{proposition}
The morphism $\chi_\otimes'\colon T^\infty\mathbb{Q} \otimes_\mathbb{Q} T^\infty\mathbb{Q} \rightarrow T^\infty(\mathbb{Q} \oplus \mathbb{Q})$ is not an isomorphism.
\end{proposition}
\begin{proof}
Recall from Definition \ref{def:Seely} that $\chi_\otimes'$ is the following natural transformation:
\begin{equation*}
\begin{tikzcd}
T^\infty\mathbb{Q} \otimes_\mathbb{Q} T^\infty\mathbb{Q} \arrow[rr, "T^\infty\iota^1 \otimes T^\infty\iota^2"] &  & T^\infty(\mathbb{Q} \oplus \mathbb{Q}) \otimes_\mathbb{Q} T^\infty(\mathbb{Q} \oplus \mathbb{Q}) \arrow[rr, "\mathsf{m}'"] &  & T^\infty(\mathbb{Q} \oplus \mathbb{Q}).
\end{tikzcd}
\end{equation*}
We first note that $T^\infty\mathbb{Q}=US^\infty(\mathbb{R} \otimes_\mathbb{Q} \mathbb{Q})=US^\infty(\mathbb{R})=UC^\infty(\mathbb{R})$ and $T^\infty(\mathbb{Q} \oplus \mathbb{Q})=US^\infty(\mathbb{R} \otimes_\mathbb{Q} (\mathbb{Q} \oplus \mathbb{Q})) \simeq US^\infty((\mathbb{R} \otimes_\mathbb{Q} \mathbb{Q}) \oplus (\mathbb{R} \otimes_\mathbb{Q} \mathbb{Q}))=US^\infty(\mathbb{R} \oplus \mathbb{R})=UC^\infty(\mathbb{R}^2)$. After unpacking the definition of $\mathsf{m}'$ and some simplifications, we obtain that
\begin{equation*}
\chi_\otimes'=m_\otimes;U\chi_\otimes;\phi,
\end{equation*}
where 
$\phi\colon UC^\infty(\mathbb{R}^2) \simeq T^\infty(\mathbb{Q} \oplus \mathbb{Q})$ is the isomorphism we mentioned above. From \cite[Remark 5.16]{CLLW2021}, we know that $\chi_\otimes$ is not surjective. It follows that $U\chi_\otimes$ is not surjective either. Since $\phi$ is an isomorphism, we obtain that $U\chi_\otimes;\phi$ is not surjective. Hence $m_\otimes;U\chi_\otimes;\phi$ is not surjective either. We conclude that $m_\otimes;U\chi_\otimes;\phi$ is not an isomorphism.
\end{proof}
We have proved that the algebra modality $S'$ has the binary Seely isomorphism but not the nullary Seely isomorphism. We have also proved that the algebra modality $T^\infty$ has neither the binary nor the nullary Seely isomorphism. We now want to prove that both of these algebra modalities have a unique deriving transformation. We will first prove that each of them has at least one deriving transformation. It follows from a general construction.
Recall that an \emph{additive category} is a category enriched over commutative monoids and that an \emph{additive functor} is a functor between additive categories which preserves sums of morphisms and zero morphisms.
\begin{lemma}
Let $(\mathsf{C},\otimes,I)$ and $(\mathsf{D},\otimes,I)$ be symmetric monoidal categories. Let $F \dashv U$ be an adjunction where $U\colon \mathsf{C} \rightarrow \mathsf{D}$ underlies a symmetric monoidal functor $(U,m_\otimes,m_I)$. Write $v\colon A \rightarrow UFA$ and $\epsilon\colon FUA \rightarrow A$ for the unit and counit of the adjunction. Suppose that the natural transformation $\nu_{A,B} \colon UA \otimes B \rightarrow U(A \otimes FB)$ defined by $\nu_{A,B}:=\mathsf{id}_U \otimes v;m_\otimes$ is a natural isomorphism. The following diagrams then commute:
\begin{equation} \label{diag-axiom-nu-1}
\begin{tikzcd}
U(A \otimes FB) \otimes UC \arrow[r, "\nu^{-1} \otimes \mathsf{id}"] \arrow[d, "m_\otimes"'] & UA \otimes B \otimes UC \arrow[d, "\mathsf{id} \otimes \sigma"]    &                                                                                              &                                                                    \\
U(A \otimes FB \otimes C) \arrow[d, "U(\mathsf{id} \otimes \sigma)"']                        & UA \otimes UC \otimes B \arrow[d, "m_\otimes \otimes \mathsf{id}"] & UA \otimes U(B \otimes FC) \arrow[d, "m_\otimes"'] \arrow[r, "\mathsf{id} \otimes \nu^{-1}"] & UA \otimes UB \otimes C \arrow[d, "m_\otimes \otimes \mathsf{id}"] \\
U(A \otimes C \otimes FB) \arrow[r, "\nu^{-1}"']                                             & {U(A \otimes C) \otimes B,}                                        & U(A \otimes B \otimes FC) \arrow[r, "\nu^{-1}"']                                             & U(A \otimes B)  \otimes C,                                   
\end{tikzcd}
\end{equation}
\begin{equation} \label{diag-axiom-nu-2}
\begin{tikzcd}
U(A \otimes FU(B \otimes FC)) \arrow[rr, "U(\mathsf{id} \otimes \epsilon)"] \arrow[d, "\nu^{-1}"'] &  & U(A \otimes B \otimes FC) \arrow[dd, "\nu^{-1}"] \\
UA \otimes U(B \otimes FC) \arrow[d, "\mathsf{id} \otimes \nu^{-1}"']                             &  &                                                  \\
UA \otimes UB \otimes C \arrow[rr, "m_\otimes \otimes \mathsf{id}"']                              &  & U(A \otimes B) \otimes C,                       
\end{tikzcd}
\end{equation}
\begin{equation} \label{inverted}
\begin{tikzcd}
U(A \otimes FC \otimes FB) \arrow[rr, "U(\mathsf{id} \otimes \sigma)"] \arrow[d, "\nu^{-1}"] &  & U(A \otimes FB \otimes FC) \arrow[d, "\nu^{-1}"]                    \\
U(A \otimes FC) \otimes B \arrow[d, "\nu^{-1} \otimes \mathsf{id}"]                          &  & U(A \otimes FB) \otimes C \arrow[d, "\nu^{-1} \otimes \mathsf{id}"] \\
UA \otimes C \otimes B \arrow[rr, "\mathsf{id} \otimes \sigma"']                             &  & UA \otimes B \otimes C.                                             
\end{tikzcd}
\end{equation}
\end{lemma}
\begin{proof}
These diagrams are shown to be commutative by first precomposing and postcomposing to obtain equivalent diagrams where only $\nu$ appears instead of $\nu^{-1}$, replacing $\nu$ by its definition, and using the identities satisfied by a symmetric monoidal functor and an adjunction.
\end{proof}
\begin{theorem} \label{prop:alg-mod+adj}
Let $(\mathsf{C},\otimes,I)$ and $(\mathsf{D},\otimes,I)$ be additive symmetric monoidal categories. Let $(T,\mu,\eta,\mathsf{m},\mathsf{u})$ be an algebra modality on $(\mathsf{C},\otimes,I)$. Let $F \dashv U$ be an adjunction where $U\colon \mathsf{C} \rightarrow \mathsf{D}$ is an additive functor underlying a symmetric monoidal functor $(U,m_\otimes,m_I)$. Write $v\colon A \rightarrow UFA$ and $\epsilon\colon FUA \rightarrow A$ for the unit and counit of the adjunction. Suppose that the natural transformation $\nu_{A,B} \colon UA \otimes B \rightarrow U(A \otimes FB)$ defined by $\nu_{A,B}:=\mathsf{id}_U \otimes v;m_\otimes$ is a natural isomorphism.

If $\mathsf{d}$ is a deriving transformation on the algebra modality $(T,\mu,\eta,\mathsf{m},\mathsf{u})$ then we obtain a deriving transformation $\mathsf{d}'$ on the algebra modality $(K,\mu',\eta',\mathsf{m}',\mathsf{u}')$ on $(\mathsf{D},\otimes,I)$ from Lemma \ref{lemma-alg-mod} which we define as
\begin{equation*}
\begin{tikzcd}
KA=UTFA \arrow[rr, "U\mathsf{d}F"] &  & U(TFA \otimes FA) \arrow[rr, "\nu^{-1}"] &  & UTFA \otimes A.
\end{tikzcd}
\end{equation*}
\end{theorem}
\begin{proof}
The naturality of $\mathsf{d}'$ follows from the naturalities of $\mathsf{d}$ and $\nu$. The linear rule for $\mathsf{d}'$ follows from the linear rule for $\mathsf{d}$. The product rule for $\mathsf{d}'$ follows from the product rule for $\mathsf{d}$ and Commutative diagrams \ref{diag-axiom-nu-1}. The chain rule for $\mathsf{d}'$ follows from the chain rule for $\mathsf{d}$ and Commutative diagram \ref{diag-axiom-nu-2}. The interchange rule for $\mathsf{d}'$ follows from the interchange rule for $\mathsf{d}$ and Commutative diagram \ref{inverted}.
\end{proof}
\begin{example} \label{example:der-free-ring}
The symmetric algebra modality $S$ on $\mathsf{Ab}$ admits a (unique) deriving transformation $\mathsf{d}$. It can be checked that the natural transformation $\nu_{A,B}\colon UA \otimes_\mathbb{N} B \rightarrow U(A \otimes_\mathbb{Z} (\mathbb{Z} \otimes_\mathbb{N}B))$ is in fact a natural isomorphism. We thus obtain a deriving transformation $\mathsf{d}'$ on the free commutative ring algebra modality $S'$ on $\mathsf{cMon}$. By Theorem \ref{prop:binary-seely-uniqueness}, $\mathsf{d}'$ is the unique deriving transformation on $S'$. Note that it can be proved that both the domain and codomain of $\nu_{A,B}$ are abelian groups, so that $\nu_{A,B}$ is an isomorphism of abelian groups. 
\end{example}
\begin{example}  \label{example:der-free-C}
The free $C^\infty$-ring algebra modality $S^\infty$ on $\mathbb{R}$-$\mathsf{Vec}$ admits a (unique) deriving transformation $\mathsf{d}$. It can be checked that the natural transformation $\nu_{A,B}\colon UA \otimes_\mathbb{Q} B \rightarrow U(A \otimes_\mathbb{R}(\mathbb{R} \otimes_\mathbb{Q}B))$ is in fact a natural isomorphism. We thus obtain a deriving transformation $\mathsf{d}'$ on the free $C^\infty$-ring algebra modality on $\mathbb{Q}$-$\mathsf{Vec}$. Note that the domain and codomain of $\nu_{A,B}$ can be made into real vector spaces, so that $\nu_{A,B}$ is in fact an isomorphism of real vector spaces.
\end{example}
To sum up: 
\begin{enumerate}
\item \emph{The free $C^\infty$-ring algebra modality $T^\infty$ on $(\mathbb{Q}$-$\mathsf{Vec},\otimes,\mathbb{Q})$ has neither the binary nor the nullary Seely isomorphism and admits a deriving transformation.}
\item \emph{The free commutative ring algebra modality $S'$ on $(\mathsf{cMon},\otimes,\mathbb{N})$ has the binary but not the nullary Seely isomorphism and admits a unique deriving transformation.}
\end{enumerate}
We will prove in the next subsection that the deriving transformation on $T^\infty$ is in fact unique.
\subsection{Uniqueness of the deriving transformation on $T^\infty$} \label{sec:unique-T}
We will first collect material on locally finitely presentable categories from \cite{CLLW2021,Kelly2005,Kelly1982FiniteLimits}. We refer the reader to \cite[Section 4.6]{Kelly2005} for the notion of a filtered category and to \cite[Section 5.1]{Kelly2005} for the notion of a dense full subcategory. Let $\mathsf{C}$ be a locally small category. As in \cite{Kelly1982FiniteLimits}, we define a \emph{small} category as a locally small category whose collection of isomorphism classes is a set. An object $X$ of $\mathsf{C}$ is \emph{finitely presentable} if the functor $\mathsf{C}(X,-)$ preserves filtered colimits, that is, colimits of functors whose domains are small and filtered. We write $\mathsf{C}_f$ for the full subcategory of finitely-presentable objects of $\mathsf{C}$. We say that $\mathsf{C}$ is \emph{locally finitely presentable} (l.f.p.) if $\mathsf{C}$ is cocomplete and the full subcategory $\mathsf{C}_f$ is small and dense in $\mathsf{C}$ \cite[Corollary 7.3]{Kelly1982FiniteLimits}. A functor between l.f.p.~categories is \emph{finitary} if it preserves filtered colimits.
\begin{example}
For every field $k$, a $k$-vector space is finitely-presentable iff it is finite-dimensional \cite[Examples 2.2 (1)]{AdamekMiliusMoss2012}.
\end{example}
By \cite[Corollary 7.3 and Theorem 7.2 (ii)]{Kelly1982FiniteLimits}, we have the following.
\begin{proposition} \label{prop:lfp-filtered-colimit}
Let $\mathsf{C}$ be a l.f.p.~category and let $I\colon \mathsf{C}_f \hookrightarrow \mathsf{C}$ be the inclusion. For every $A \in \mathsf{C}$, there exists a small filtered category $\mathsf{J}$ and a functor $F\colon \mathsf{J} \rightarrow \mathsf{C}_f$ such that $A$ is a colimit of $IF$. We write $A=\mathrm{colim}_{j \in \mathsf{J}}\,IF(j)$. 
\end{proposition}
We will write $U$ for the restriction of scalars functor from $\mathbb{R}$-$\mathsf{Vec}$ to $\mathbb{Q}$-$\mathsf{Vec}$. Let $M \in \mathbb{Q}$-$\mathsf{Vec}$. From Proposition \ref{prop:lfp-filtered-colimit}, we have $M=\mathrm{colim}_{j \in \mathsf{J}}\,IF(j)$, where the $IF(j)$ are finite-dimensional rational vector spaces. Let $\iota_j\colon IF(j) \rightarrow M$ be the cocone maps. Recall that from Example \ref{ex:new-monad-free-C}, we have
\begin{equation*}
T^\infty=U \circ S^\infty \circ (\mathbb{R} \otimes_\mathbb{Q}-),
\end{equation*}
where $\mathbb{R} \otimes_\mathbb{Q}-\colon \mathbb{Q}$-$\mathsf{Vec} \rightarrow \mathbb{R}$-$\mathsf{Vec}$ is a left adjoint (whose right adjoint is the restriction of scalars functor, see Subsection \ref{sec:ext-res-der}), $S^\infty\colon\mathbb{R}$-$\mathsf{Vec} \rightarrow \mathbb{R}$-$\mathsf{Vec}$ is finitary (see Section \ref{sec:unique-S-infty}), and $U\colon\mathbb{R}$-$\mathsf{Vec}\rightarrow \mathbb{Q}$-$\mathsf{Vec}$ is a left adjoint (whose right adjoint is the coextension of scalars functor, see Subsection \ref{sec:ext-res-der}). Since left adjoints preserve all colimits and finitary functors preserve filtered colimits, the composite $T^\infty$ preserves filtered colimits.

We thus obtain that $T^\infty M=\mathrm{colim}_{j \in \mathsf{J}}\,T^\infty IF(j)$. In particular, the family of morphisms $(T^\infty\iota_j\colon T^\infty IF(j) \rightarrow T^\infty M)_{j \in \mathsf{J}}$ is jointly epimorphic. But for every $j \in \mathsf{J}$, there exists an $n_j \in \mathbb{N}$ and an isomorphism $\phi_j \colon \mathbb{Q}^{n_j} \overset{\sim}{\rightarrow} IF(j)$. 
Since the $T^\infty \phi_j$ are isomorphisms, we deduce that the family 
\begin{equation} \label{jointly-epi}
(T^\infty \psi_j\colon T^\infty \mathbb{Q}^{n_j} \rightarrow T^\infty M)_{j \in \mathsf{J}},
\end{equation}
is jointly epimorphic, where $\psi_j:=\iota_j \circ \phi_j$. We will now use essentially the same technique as in Section \ref{sec:unique-S-infty} to prove that if $\mathsf{d}_1$ and $\mathsf{d}_2$ are two deriving transformations on $T^\infty$, then they agree on $\mathbb{Q}^n$ for every $n \ge 0$. We will then use the jointly epimorphic family \ref{jointly-epi} to prove that in fact they agree on all rational vector spaces. It will imply that the deriving transformation $\mathsf{d}'$ from Example \ref{example:der-free-C} is the unique deriving transformation on $T^\infty$. From now on, we fix an arbitrary deriving transformation $\mathsf{d}^*$ on $T^\infty$.
\begin{lemma} \label{def-partial}
Let $n \ge 0$ and $a_n\colon \mathbb{R} \otimes_\mathbb{Q} \mathbb{Q}^n \rightarrow \mathbb{R}^n$ be the isomorphism of real vector spaces given by $a_n(\lambda \otimes (k_1,\dots,k_n))=(\lambda k_1,\dots,\lambda k_n)$. Define a $\mathbb{Q}$-linear map $\partial_n$ by the following commutative diagram:
\begin{equation*}
\begin{tikzcd}
UC^\infty(\mathbb{R}^n) \arrow[rr, "\partial_n"] \arrow[d, "US^\infty a_n^{-1}"'] &  & UC^\infty(\mathbb{R}^n) \otimes_\mathbb{Q} \mathbb{Q}^n \arrow[d, "US^\infty a_n^{-1} \otimes \mathsf{id}"] \\
T^\infty(\mathbb{Q}^n) \arrow[rr, "\mathsf{d}^*_{\mathbb{Q}^n}"']                 &  &  T^\infty(\mathbb{Q}^n) \otimes_\mathbb{Q} \mathbb{Q}^n.                                                   
\end{tikzcd}
\end{equation*}
The $\mathbb{Q}$-linear map $\partial_n$ is an $\mathbb{R}$-linear derivation from the commutative $\mathbb{R}$-algebra $C^\infty(\mathbb{R}^n)$ to the $C^\infty(\mathbb{R}^n)$-module $UC^\infty(\mathbb{R}^n) \otimes_\mathbb{Q} \mathbb{Q}^n$.
\end{lemma}
\begin{proof}
We will first prove that for every rational vector space $M$, the $\mathbb{Q}$-linear map 
\begin{equation*}
\mathsf{d}^*_M \colon T^\infty M \rightarrow T^\infty M \otimes_\mathbb{Q} M
\end{equation*}
is an $\mathbb{R}$-linear derivation. The rational vector space $T^\infty M=US^\infty(\mathbb{R} \otimes_\mathbb{Q} M)$ underlies the commutative $\mathbb{R}$-algebra $S^\infty(\mathbb{R} \otimes_\mathbb{Q} M)$. Moreover, the $T^\infty M$-module $T^\infty M \otimes_\mathbb{Q} M$ can be seen as the real vector space underlying an $S^\infty(\mathbb{R} \otimes_\mathbb{Q}M)$-module. (It can readily be checked from the definition of the monoid multiplication $\mathsf{m}'$ on $T^\infty$ in Lemma \ref{lemma-alg-mod} that the underlying commutative rigs of $T^\infty M$ and $S^\infty(\mathbb{R} \otimes_\mathbb{Q} M)$ are identical. By an abuse of language, we talk of $T^\infty M$-modules, although from the type of the functor $T^\infty$, the object $T^\infty M$ is technically speaking a rational vector space and not a rig. By the same abuse of language, we can talk of $S^\infty(\mathbb{R} \otimes_\mathbb{Q} M)$-modules. These two abuses of language refer to the exact same mathematical objects: modules over the common underlying rig.) It thus makes sense to ask whether $\mathsf{d}^*_M$ is an $\mathbb{R}$-linear derivation.
Let $\mathsf{u}\colon \mathbb{R} \rightarrow S^\infty(\mathbb{R} \otimes_\mathbb{Q} M)$ be the $\mathbb{R}$-linear monoid unit map of $S^\infty$.
By the product rule for $\mathsf{d}^*$, we have
\begin{equation} \label{first-for-lin}
\mathsf{d}^*_M(\lambda t)=\mathsf{d}^*_M(\mathsf{u}(\lambda)t)=\mathsf{u}(\lambda)\mathsf{d}^*_M(t)+t\mathsf{d}^*_M(\mathsf{u}(\lambda))=\lambda \mathsf{d}^*_M(t)+t\mathsf{d}^*_M(\mathsf{u}(\lambda)).
\end{equation}
But from the following commutative diagram:
\begin{equation*}
\begin{tikzcd}
U\mathbb{R} \arrow[rd, "U\mathsf{u}_{\mathbb{R} \otimes_\mathbb{Q} M}"'] \arrow[r, "U\mathsf{u}_{\mathbb{R} \otimes_\mathbb{Q} 0}"] & T^\infty 0 \arrow[r, "\mathsf{d}_0^*"] \arrow[d, "T^\infty 0"] & T^\infty 0 \otimes_\mathbb{Q} 0 \arrow[d, "T^\infty0 \otimes 0=0"] \\
                                                                                                  & T^\infty M \arrow[r, "\mathsf{d}^*_M"']                      & T^\infty M \otimes_\mathbb{Q} M,                                  
\end{tikzcd}
\end{equation*}
we obtain that $U\mathsf{u}_{\mathbb{R} \otimes_\mathbb{Q} M};\mathsf{d}^*_M=0$ for every rational vector space $M$. In particular, $\mathsf{d}^*_M(\mathsf{u}(\lambda))=0$. We deduce from Equation \ref{first-for-lin} that $\mathsf{d}^*_M(\lambda t)=\lambda \mathsf{d}^*_M(t)$. Hence, $\mathsf{d}^*_M$ is  an $\mathbb{R}$-linear map.
Moreover, the product rule for $\mathsf{d}^*$ gives that
\begin{equation*}
\mathsf{d}^*_M(st)=s\mathsf{d}^*_M(t)+t\mathsf{d}^*_M(s)
\end{equation*}
for all $s,t \in S^\infty(\mathbb{R} \otimes_\mathbb{Q} M)$. We conclude that $\mathsf{d}^*_M$ is an $\mathbb{R}$-linear derivation from the commutative $\mathbb{R}$-algebra $S^\infty(\mathbb{R} \otimes_\mathbb{Q} M)$ to the $S^\infty(\mathbb{R} \otimes_\mathbb{Q} M)$-module $T^\infty M \otimes_\mathbb{Q} M$.

We now set $M:=\mathbb{Q} ^n$. We thus have an $\mathbb{R}$-linear derivation $\mathsf{d}^*_{\mathbb{Q} ^n}$ from the commutative $\mathbb{R}$-algebra $S^\infty(\mathbb{R} \otimes_\mathbb{Q} \mathbb{Q} ^n)$ to the $S^\infty(\mathbb{R} \otimes_\mathbb{Q} \mathbb{Q}^n)$-module $T^\infty\mathbb{Q}^n \otimes_\mathbb{Q} \mathbb{Q} ^n$. Note the following facts:
\begin{enumerate}
\item $US^\infty a_n^{-1}$ underlies a homomorphism of $\mathbb{R}$-algebras: indeed, $S^\infty$ sends every $\mathbb{R}$-linear map to a homomorphism of $\mathbb{R}$-algebras by the axioms of a differential category;
\item $US^\infty a_n^{-1} \otimes \mathsf{id}$ is $\mathbb{R}$-linear (when its domain and codomain are seen as real vector spaces): it follows from $US^\infty a_n^{-1}$ underlying an $\mathbb{R}$-linear map since $\mathbb{R}$ acts on the left factors $UC^\infty(\mathbb{R}^n)$ and $T^\infty(\mathbb{Q}^n)$;
\item for all $\lambda \in T^\infty(\mathbb{Q}^n)$ and $x \in T^\infty(\mathbb{Q}^n) \otimes_\mathbb{Q}\mathbb{Q}^n$, we have  
\begin{equation*}
(US^\infty a_n \otimes \mathsf{id})(\lambda x)=US^\infty(a_n)(\lambda)\cdot (US^\infty a_n \otimes \mathsf{id})(x).
\end{equation*}
\end{enumerate}
From (1) and (2) and since $\mathsf{d}^*_{\mathbb{Q}^n}$ is $\mathbb{R}$-linear, we immediately obtain that $\partial_n=(US^\infty a_n \otimes \mathsf{id}) \circ \mathsf{d}^*_{\mathbb{Q}^n} \circ US^\infty a_{n}^{-1}$ is $\mathbb{R}$-linear. A straightforward computation using first (1), then $\mathsf{d}^*_{\mathbb{Q}^n}(ab)=a\mathsf{d}^*_{\mathbb{Q}^n}(b)+b\mathsf{d}^*_{\mathbb{Q}^n}(a)$, and finally (3), shows that $\partial_n(ab)=a\partial_n(b)+b\partial_n(a)$. We conclude that $\partial_n$ is an $\mathbb{R}$-linear derivation.
\end{proof}
\begin{lemma} \label{der-pi-mon}
Let $(e_1,\dots,e_n)$ be the canonical basis of the rational vector space $\mathbb{Q}^n$. For all $n \ge 1$ and $1 \le i \le n$, the $\mathbb{R}$-linear derivation $\partial_n$ satisfies $\partial_n(\pi_i)=1 \otimes e_i$.
\end{lemma}
\begin{proof}
We will use the equality of the two external paths from $\mathbb{Q} ^n$ to $UC^\infty(\mathbb{R}^n) \otimes_\mathbb{Q} \mathbb{Q} ^n$ in the following commutative diagram:
\begin{equation*}
\begin{tikzcd}
U\mathbb{R}^n \arrow[r, "U\eta_{\mathbb{R}^n}"]                                                                                 & UC^\infty(\mathbb{R}^n) \arrow[rr, "\partial_n"]                                                                                                                                               &  & UC^\infty(\mathbb{R}^n) \otimes_\mathbb{Q} \mathbb{Q}^n                                                                                                                                                 \\
U(\mathbb{R} \otimes_\mathbb{Q} \mathbb{Q}^n) \arrow[r, "U\eta_{\mathbb{R} \otimes_\mathbb{Q} \mathbb{Q}^n}"] \arrow[u, "Ua_n"] & T^\infty(\mathbb{Q}^n) \arrow[rr, "\mathsf{d}^*_{\mathbb{Q}^n}"] \arrow[u, "US^\infty a_n"]                                                                                                    &  &  T^\infty(\mathbb{Q}^n) \otimes_\mathbb{Q} \mathbb{Q}^n \arrow[u, "US^\infty a_n \otimes \mathsf{id}"]                                                                                                  \\
                                                                                                                             & \mathbb{Q}^n \arrow[u, "\eta'_{\mathbb{Q}^n}"] \arrow[rru, "\mathsf{u}'_{\mathbb{Q}^n} \otimes \mathsf{id}" description] \arrow[lu, "v_{\mathbb{Q}^n}"] \arrow[rr, "m_\mathbb{Q} \otimes \mathsf{id}"'] &  & U\mathbb{R} \otimes_\mathbb{Q} \mathbb{Q}^n. \arrow[u, "U\mathsf{u}_{\mathbb{R} \otimes \mathbb{Q}^n} \otimes \mathsf{id}" description] \arrow[uu, "U\mathsf{u}_{\mathbb{R}^n} \otimes \mathsf{id}"', bend right=71]
\end{tikzcd}
\end{equation*}
We compute the left-hand external path applied to $e_i$:
\begin{equation*}
\partial_n(\eta_{\mathbb{R}^n}(a_n(v_{\mathbb{Q} ^n}(e_i))))=\partial_n(\eta_{\mathbb{R}^n}(a_n(1 \otimes e_i)))=\partial_n(\eta_{\mathbb{R}^n}(e_i))=\partial_n(\pi_i).
\end{equation*}
We now compute the right-hand external path applied to $e_i$:
\begin{equation*}
(U\mathsf{u}_{\mathbb{R}^n} \otimes \mathsf{id})((m_\mathbb{Q}  \otimes \mathsf{id})(e_i))=(U\mathsf{u}_{\mathbb{R}^n} \otimes \mathsf{id})(1 \otimes e_i)=1 \otimes e_i.
\end{equation*}
We conclude that $\partial_n(\pi_i)=1 \otimes e_i$.
\end{proof}
Define the $\mathbb{Q}$-linear map $\phi_i\colon UC^\infty(\mathbb{R}^n) \otimes_\mathbb{Q} \mathbb{Q} ^n \rightarrow UC^\infty(\mathbb{R}^n)$ by $\phi_i(f \otimes (k_1,\dots,k_n))=k_if$. For every $\lambda \in \mathbb{R}$, we have $\phi_i(\lambda f \otimes (k_1,\dots,k_n))=\lambda k_if$, hence $\phi_i$ is in fact an $\mathbb{R}$-linear map. The composite $\partial_n;\phi_i$ is a map from $UC^\infty(\mathbb{R}^n)$ to itself. But since this composite is $\mathbb{R}$-linear, we can define $D_i^n\colon C^\infty(\mathbb{R}^n) \rightarrow C^\infty(\mathbb{R}^n)$ as the $\mathbb{R}$-linear map whose underlying $\mathbb{Q}$-linear map is $\partial_n;\phi_i$.
\begin{lemma} \label{Di-cMon}
For all $n \ge 1$ and $1 \le i \le n$, the function $D_i^n\colon C^\infty(\mathbb{R}^n) \rightarrow C^\infty(\mathbb{R}^n)$ is an $\mathbb{R}$-linear derivation which satisfies $D_i^n(\pi_j)=\delta_{i,j}$ for every $1 \le j \le n$.
\end{lemma}
\begin{proof}
We already know that $D_i^n$ is $\mathbb{R}$-linear. Let $f,g \in C^\infty(\mathbb{R}^n)$. We have $D_i^n(fg)=\phi_i(\partial_n(fg))=\phi_i(f\partial_n(g)+g\partial_n(f))$. Write $\partial_n(f)=\underset{1 \le k \le n}{\sum}f_k \otimes e_k$ and $\partial_n(g)=\underset{1 \le k \le n}{\sum}g_k \otimes e_k$. We then have $\phi_i(f\partial_n(g)+g\partial_n(f))=\phi_i\left(\underset{1 \le k \le n}{\sum}fg_k \otimes e_k+\underset{1 \le k \le n}{\sum}gf_k \otimes e_k\right)=fg_i+gf_i=fD_i^n(g)+gD_i^n(f)$. Thus $D_i^n$ is an $\mathbb{R}$-linear derivation. Furthermore, Lemma \ref{der-pi-mon} implies that $D_i^n(\pi_j)=\phi_i(\partial_n(\pi_j))=\phi_i(1 \otimes e_j)=\delta_{i,j}$.
\end{proof}
\begin{proposition} \label{prop-dn-smooth}
For every $n \ge 0$, the $\mathbb{R}$-linear derivation $\partial_n$ satisfies
\begin{equation*}
\partial_n(f)=\underset{1 \le i \le n}{\sum}\frac{\partial f}{\partial x_i} \otimes e_i
\end{equation*}
for every $f \in C^\infty(\mathbb{R}^n)$.
\end{proposition}
\begin{proof}
It suffices to note that Lemma \ref{lemma-Di} generalizes to the following statement, with same proof:

\emph{Let $n \ge 1$ and $1 \le i \le n$. If $F_i\colon C^\infty(\mathbb{R}^n) \rightarrow C^\infty(\mathbb{R}^n)$ is an $\mathbb{R}$-linear derivation such that $F_i(\pi_j)=\delta_{i,j}$ for every $1 \le j \le n$, then $F_i=\frac{\partial}{\partial x_i}$.}

Let $n \ge 1$. Combining the above result with Lemma \ref{Di-cMon} gives that $D_i^n=\frac{\partial}{\partial x_i}$ for every $1 \le i \le n$. Let $f \in C^\infty(\mathbb{R}^n)$. We have $\partial_n(f)=\underset{1 \le i \le n}{\sum}f_i \otimes e_i$ for some $f_i \in C^\infty(\mathbb{R}^n)$. Applying $\phi_i\colon UC^\infty(\mathbb{R}^n) \otimes_\mathbb{Q} \mathbb{Q}^n \rightarrow UC^\infty(\mathbb{R}^n)$ to this identity, we obtain that $f_i=D_i^n(f)$. We thus obtain the desired result for every $n \ge 1$. We also need to prove that $\partial_0=0$. This is immediate since the codomain of $\partial_0$ is $UC^\infty(\mathbb{R}^0) \otimes_\mathbb{Q} \mathbb{Q}^0 \simeq UC^\infty(\mathbb{R}^0) \otimes_\mathbb{Q} 0 \simeq 0$.
\end{proof}
\begin{corollary} \label{cor:unicity-T-infty}
If $\mathsf{d}^1$ and $\mathsf{d}^2$ are two deriving transformations on $T^\infty$ then $\mathsf{d}^1=\mathsf{d}^2$. 
\end{corollary}
\begin{proof}
Let $n \ge 0$. Define the two $\mathbb{Q}$-linear maps $\partial_n^i$ for $i \in \{1,2\}$
from $UC^\infty(\mathbb{R}^n)$ to $UC^\infty(\mathbb{R}^n) \otimes_\mathbb{Q} \mathbb{Q} ^n$ by the formula
\begin{equation*}
\partial_n^i=US^\infty a_n^{-1};\mathsf{d}_{\mathbb{Q} ^n}^i;US^\infty a_n \otimes \mathsf{id}
\end{equation*}
as in Lemma \ref{def-partial}. Proposition \ref{prop-dn-smooth} implies that
\begin{equation*}
\partial_n^i(f)=\underset{1 \le k \le n}{\sum}\frac{\partial f}{\partial x_k} \otimes e_k
\end{equation*}
for every $f \in C^\infty(\mathbb{R}^n)$. We thus have $\partial_n^1=\partial_n^2$. Since $US^\infty a_n^{-1}$ and $US^\infty a_n \otimes \mathsf{id}$ are isomorphisms of rational vector spaces, it implies that $\mathsf{d}^1_{\mathbb{Q} ^n}=\mathsf{d}^2_{\mathbb{Q} ^n}$. 

Now, let $M$ be any rational vector space. Consider the family of $\mathbb{Q}$-linear maps $(T^\infty \psi_j)_{j \in \mathsf{J}}$  in Equation \ref{jointly-epi}, which is jointly epimorphic in $\mathbb{Q}$-$\mathsf{Vec}$. For any $i \in \{1,2\}$ and for every $j \in \mathsf{J}$, the following diagram commutes:
\begin{equation*}
\begin{tikzcd}
T^\infty\mathbb{Q} ^{n_j} \arrow[d, "T^\infty \psi_j"'] \arrow[r, "\mathsf{d}_{\mathbb{Q} ^{n_j}}^i"] & T^\infty\mathbb{Q} ^{n_j} \otimes_\mathbb{Q} \mathbb{Q} ^{n_j} \arrow[d, "T^\infty \psi_j \otimes \psi_j"] \\
T^\infty M \arrow[r, "\mathsf{d}_M^i"']                                                             & T^\infty M \otimes_\mathbb{Q} M.                                                                        
\end{tikzcd}
\end{equation*}
We deduce that $\mathsf{d}_M^1=\mathsf{d}_M^2$ for every rational vector space $M$.
\end{proof}
Corollary \ref{cor:unicity-T-infty} implies that $\mathsf{d}'$ is the unique deriving transformation on $T^\infty$.
\section*{AI tool disclosure} \label{sec:AI}
Several results in this paper have been obtained with the help of AI. We did not find the proof of Theorem \ref{prop:binary-seely-uniqueness} by ourselves despite some effort. It was generated by the LLM GPT-6 Astra (medium reasoning effort), accessed through ChatGPT, during a session on September 10, 2026. We asked the LLM to try to answer the question whether an algebra modality with the binary Seely isomorphism, on an additive symmetric monoidal category with finite coproducts, admits at most one deriving transformation. We explained that Lemay had already proved in \cite{Lemay2026Coderelictions} that there is at most one deriving transformation when the algebra modality has both Seely isomorphisms. The LLM directly found that there is already uniqueness with only the binary Seely isomorphism. We verified the proof, rewrote it to assimilate it, make it easier to understand, and adapt the notations to those we use in this paper. The result of this process is contained in this paper.

The proof of Theorem \ref{prop:binary-seely-uniqueness} is the main result in this paper which has been obtained by an LLM. Other notable contributions by LLMs are: 
\begin{itemize}
\item The observation by Claude Sonnet that the natural transformation $\nu_{A,B}$ in Theorem \ref{prop:alg-mod+adj} is actually a natural isomorphism when the adjunction $F \dashv U$ is given by extension-restriction of scalars. This is why we took this hypothesis in Theorem \ref{prop:alg-mod+adj}.
\item The suggestion by Gemini 3.5 Flash, accessed automatically through a Google search, of the map on $k[x]$ which truncates a polynomial to the sum of the monomials of degree $0$ and $1$, as an endofunction on $k[x]$ which preserves $+$, $0$, and $1$ but not the multiplication. This has been used as an inspiration for the proof of Proposition \ref{prop:no-Seely-2-for-T-f+01}.
\item A proof by Gemini 3.5 Flash, accessed automatically through a Google search, that the set $D=\{(i,i)~|~i \in \mathbb{N}\}$ can't be written as a finite union of subsets of $\mathbb{N}^2$ of the form $A \times B$ where $A$ and $B$ are subsets of $\mathbb{N}$. This has been used in the proof of Proposition \ref{prop:no-Seely-2-for-T-f+01}.
\end{itemize}

LLMs have also been used for purely editorial purposes: Figure \ref{fig:comb} was first generated by Claude Sonnet (without the caption). However, the result was not very legible and we had to modify it considerably. The BibTeX code for most bibliography items has been generated by GPT-6 Astra. This paper has been reviewed by GPT-6 Astra, which resulted in corrections of typos, grammar mistakes, and mathematical imprecisions. GPT-6 Astra has also been used to automatically change our initial notations for an algebra modality to those commonly used in the literature, such as in \cite{LemaySava2026}. 

The interactions with LLMs other than the one about Theorem \ref{prop:binary-seely-uniqueness} happened in October 2026 or earlier.
\bibliographystyle{plain} 
\bibliography{references-many-new} 

\begin{thebibliography}{10}

\bibitem{AdamekMiliusMoss2012}
J.~Ad{\'a}mek, S.~Milius, and L.~S. Moss.
\newblock On finitary functors and their presentations.
\newblock In {\em Coalgebraic Methods in Computer Science}, volume 7399 of {\em
  Lecture Notes in Computer Science}, pages 51--70. Springer, 2012.

\bibitem{BCLS2020}
R.~F. Blute, J.~R.~B. Cockett, J.-S.~P. Lemay, and R.~A.~G. Seely.
\newblock Differential categories revisited.
\newblock {\em Applied Categorical Structures}, 28(2):171--235, 2020.

\bibitem{BCS2006}
R.~F. Blute, J.~R.~B. Cockett, and R.~A.~G. Seely.
\newblock Differential categories.
\newblock {\em Mathematical Structures in Computer Science}, 16(6):1049--1083,
  2006.

\bibitem{BCS2009}
R.~F. Blute, J.~R.~B. Cockett, and R.~A.~G. Seely.
\newblock {Cartesian} differential categories.
\newblock {\em Theory and Applications of Categories}, 22(23):622--672, 2009.

\bibitem{BurrisSankappanavar2012}
Stanley Burris and H.~P. Sankappanavar.
\newblock {\em A Course in Universal Algebra}.
\newblock Self-published online edition, 2012.
\newblock Millennium Edition of the 1981 Springer-Verlag Graduate Texts in
  Mathematics volume 78. Freely available at
  \url{https://math.uwaterloo.ca/~snburris/htdocs/ualg.html}.

\bibitem{CC2014}
J.~R.~B. Cockett and G.~S.~H. Cruttwell.
\newblock Differential structure, tangent structure, and {SDG}.
\newblock {\em Applied Categorical Structures}, 22(2):331--417, 2014.

\bibitem{CLLW2021}
G.~S.~H. Cruttwell, J.-S.~P. Lemay, and R.~B.~B. Lucyshyn-Wright.
\newblock Integral and differential structure on the free {$C^{\infty}$}-ring
  modality.
\newblock {\em Cahiers de Topologie et G{\'e}om{\'e}trie Diff{\'e}rentielle
  Cat{\'e}goriques}, 62(2):116--176, 2021.

\bibitem{Ehrhard2018}
T.~Ehrhard.
\newblock An introduction to differential linear logic: proof-nets, models and
  antiderivatives.
\newblock {\em Mathematical Structures in Computer Science}, 28(7):995--1060,
  2018.

\bibitem{FauskHuMay2003}
H.~Fausk, P.~Hu, and J.~P. May.
\newblock Isomorphisms between left and right adjoints.
\newblock {\em Theory and Applications of Categories}, 11(4):107--131, 2003.

\bibitem{Girard1987}
J.-Y. Girard.
\newblock Linear logic.
\newblock {\em Theoretical Computer Science}, 50(1):1--101, 1987.

\bibitem{Golan1999Semirings}
J.~S. Golan.
\newblock {\em Semirings and their Applications}.
\newblock Kluwer Academic Publishers, Dordrecht, 1999.

\bibitem{Kelly1982FiniteLimits}
G.~M. Kelly.
\newblock Structures defined by finite limits in the enriched context, {I}.
\newblock {\em Cahiers de topologie et g\'eom\'etrie diff\'erentielle},
  23(1):3--42, 1982.

\bibitem{Kelly2005}
G.~M. Kelly.
\newblock Basic concepts of enriched category theory.
\newblock {\em Reprints in Theory and Applications of Categories}, (10):1--136,
  2005.
\newblock Reprint of the 1982 original.

\bibitem{Lemay2021Coderelictions}
J.-S.~P. Lemay.
\newblock Coderelictions for free exponential modalities.
\newblock In Fabio Gadducci and Alexandra Silva, editors, {\em 9th Conference
  on Algebra and Coalgebra in Computer Science (CALCO 2021)}, volume 211 of
  {\em Leibniz International Proceedings in Informatics (LIPIcs)}, pages
  19:1--19:21, Dagstuhl, Germany, 2021. Schloss Dagstuhl -- Leibniz-Zentrum
  f{\"u}r Informatik.

\bibitem{Lemay2022Uniqueness}
J.-S.~P. Lemay.
\newblock Uniqueness of differentiation in differential categories, 2022.
\newblock Talk at the Category Theory Octoberfest (online), based on joint work
  with Marie Kerjean.

\bibitem{Lemay2026Coderelictions}
J.-S.~P. Lemay.
\newblock Additive enrichment from coderelictions.
\newblock {\em Logical Methods in Computer Science}, 22(1), 2026.

\bibitem{LemaySava2026}
J.-S.~P. Lemay and C.~Sava.
\newblock Derivations as algebras.
\newblock {\em Journal of the London Mathematical Society}, 114(2):e70650,
  2026.

\bibitem{Stack}
levap (https://math.stackexchange.com/users/32262/levap).
\newblock What is an example of a smooth function in ${C}^\infty(\mathbb{R}^2)$
  which is not contained in ${C}^\infty(\mathbb{R})\otimes
  {C}^\infty(\mathbb{R})$.
\newblock Mathematics Stack Exchange.
\newblock URL:https://math.stackexchange.com/q/2244646 (version: 2017-04-21).

\bibitem{Magid1994}
Andy~R. Magid.
\newblock {\em Lectures on Differential {Galois} Theory}, volume~7 of {\em
  University Lecture Series}.
\newblock American Mathematical Society, Providence, RI, 1994.

\bibitem{Nestruev2020}
J.~Nestruev.
\newblock {\em Smooth Manifolds and Observables}, volume 220 of {\em Graduate
  Texts in Mathematics}.
\newblock Springer, Cham, 2nd edition, 2020.

\bibitem{PV2007}
E.~Palmgren and S.~J. Vickers.
\newblock Partial {Horn} logic and cartesian categories.
\newblock {\em Annals of Pure and Applied Logic}, 145(3):314--353, 2007.

\bibitem{Rosicky1984}
J.~Rosick{\'y}.
\newblock Abstract tangent functors.
\newblock {\em Diagrammes}, 12:JR1--JR11, 1984.

\end{thebibliography}
\end{document}